\documentclass[12pt]{amsart}
 
\usepackage[a4paper,top=3.6cm,bottom=3.7cm,inner=2.3cm,outer=2.3cm]{geometry}
\usepackage[pdftex]{graphicx,graphics}
\usepackage{color,ifpdf}
\usepackage{amsmath,amssymb,amsfonts,dsfont,enumerate,url,mathrsfs}
\usepackage[pdftex]{hyperref}
 \usepackage{epsfig}

\usepackage{graphics,graphicx,psfrag}

\usepackage{graphicx,tikz} 
\usetikzlibrary{shapes.symbols}
 \usetikzlibrary{shapes.geometric}
 \usetikzlibrary{shapes.callouts}
 \usetikzlibrary{shapes.misc}
 \usetikzlibrary{shapes.multipart}
  \usetikzlibrary{shapes.arrows}

\newtheorem{theorem}{Theorem}[section]
\newtheorem{proposition}[theorem]{Proposition}

\newtheorem{lemma}[theorem]{Lemma}
\newtheorem{definition}[theorem]{Definition}
\newtheorem{corollary}[theorem]{Corollary}

\newcommand\ep{{\varepsilon}}

\newcommand\zu{[0,1]}

\newcommand\mk{\medskip}
\newcommand\sk{\smallskip}

\newcommand{\R}{\mathbb{R}}
\newcommand{\N}{\mathbb{N}}

\newcommand\supp{\mathrm{supp}}

\newcommand\wI{{I}^{(n)}}
\newcommand\support{\mbox{Supp}}
\newcommand\Pzeroeta{\mathcal{P}_\eta}

\newcommand{\IA}{I^{\mathcal{A}}}
\newcommand{\IAc}{I^{\mathcal{A}^c}} 
\newcommand\ds{\displaystyle}
\newcommand\dimi{\underline \dim}

\newcommand\mup{{\mu_p}}

\newcommand\G{G_{1,2}}

\newcommand\Tau{T}
\newcommand\Taue{T^\eta}
\newcommand\Tauet{\widetilde T^\eta}
\newcommand\coefg{\delta_{1,2}}

\newcommand{\locloc}{{}}
\newcommand{\wep}{\widetilde \ep}
\newcommand\weta{{\widetilde \eta}}

\newcommand\dimu{\underline{\dim}}
\newcommand\dimo{\overline{\dim}}
\newcommand\ml{multifractal }
\newcommand\nue{\nu_{\eta}}
\newcommand\muu{{\mu_{p_1}}}
\newcommand\mud{{\mu_{p_2}}}
\newcommand\jA{j^{\mathcal{A}}}
\newcommand\jAc{j^{\mathcal{A}^c}}
\newcommand\AAA{\mathcal{A}}
\newcommand\BBB{\mathcal{B}}
\newcommand\AAAc{{\mathcal{A}^c}}
\newcommand\BBBc{{\mathcal{B}^c}}

\begin{document}
\title[On the multivariate multifractal formalism]{On the multivariate multifractal formalism: examples and counter-examples}

\author{St\'ephane Seuret}

\address{St\'ephane Seuret,   Univ Paris Est Creteil, Univ Gustave Eiffel, CNRS, LAMA UMR8050, F-94010 Creteil, France}
\email{seuret@u-pec.fr}

 
%
%
\begin{abstract}
In this article, we investigate the bivariate \ml analysis of pairs of Borel probability measures. We prove that, contrarily to what happens in the univariate case, the natural extension of the Legendre spectrum does not yield an upper bound for the  bivariate \ml spectrum. For this we build a pair of measures for which the two spectra have disjoint supports. Then we study the bivariate \ml behavior of an archetypical pair of   randomly correlated  measures,  which give new, surprising,     behaviors, enriching the narrow class of measures for which such an analysis is achieved.
\end{abstract}


\maketitle

\section{Introduction}

Our purpose here is to  further  investigate the  multivariate multifractal analysis of measures by studying the validity of a multifractal formalism, and by illustrating  the diversity of possible behaviors with the example  of  a pair of randomly correlated binomial Bernoulli cascades, which is the archetype of two correlated multifractal objects.

 Recall that the lower and upper  local dimensions of a Borel measure $\mu$ defined on $\zu$ at $x$ are respectively defined as 
$$\dimu (\mu,x)=\liminf_{r \to 0^+}\frac{\log \mu(B(x,r))}{ \log r}  \ \text{ and } \  {\overline \dim}(\mu,x)=\limsup_{r\to 0^+}\frac{\log \mu(B(x,r))}{ \log r },
$$
and $\dim (\mu,x)$ stands for the common value of  $\dimu (\mu,x)$ and $\dimo(\mu,x)$ when these quantities agree.
Then, using the Hausdorff dimension $\dim$, the multifractal spectrum 
\begin{equation}
\label{eq-defdmu}
H\mapsto D_\mu(H) =  \dim \underline E_\mu(H), \mbox{ where  } \underline E_\mu(H):=\left \{x\in[0,1] : {\underline \dim_\locloc}(\mu,x)=H\right \} 
\end{equation}
describes the distribution of the singularities of $\mu$. We emphasize that   the lower local dimension, defined for every point in the support of $\mu$, is considered here.

An  important quantity is the $L^q$-spectrum $\tau_{\mu}$ of a probability measure $\mu$ supported by $\zu$.  Define $\mathcal{D}_j $   as the set of dyadic intervals of generation $j\geq 0$, i.e. 
$$\mathcal{D}_j = \{[k2^{-j},(k+1)2^{-j}): k\in \{0,..., 2^1-1\}\},$$
and $\mathcal{D}=\bigcup_{j\geq 0}\mathcal{D}_j$ the set of dyadic intervals of $\zu$. Then the $L^q$-spectrum $\tau_{\mu}$ is  
\begin{equation}
\label{deftaumu}
\tau_\mu(q) = \liminf_{j\to +\infty} \tau_{\mu,j}(q) , \ \mbox{ where } \tau_{\mu,j}(q) = \frac{-1}{j} \log_2 \sum_{I\in \mathcal{D}_j} \mu(3I)^q,
\end{equation}
where for any interval $I$, $3I$ stands for the interval with same center as $I$ and radius three times that of $I$.
It is standard that for Borel measures $\mu$ on $\zu$ (and in $\zu^d$), the multifractal spectrum $D_\mu$ is bounded above by the Legendre spectrum of the $L^q$-spectrum, i.e. one always has $D_\mu(H)\leq \tau_\mu^*(H) := \inf_{q\in \R} (q H-\tau_\mu(q))$ \cite{FrischParisi,Halsey,BRMICHPEY}.
When the equality $D_\mu(H)= \tau_\mu^*(H)$ holds, $\mu$ is said to satisfy the multifractal formalism at $H$.

The investigation of \ml properties of measures (and functions) is now a standard issue in fractal geometry, ergodic theory  and geometric measure theory, and has been investigated  for large classes of deterministic and random Borel probability measures. Among many examples, let us mention  Gibbs measures invariant by various dynamical systems \cite{BRMICHPEY,BBP,FAN,FENG2007,BarralFeng2}, self-similar and self-affine measures \cite{LauNgai,FengLau,Barral-Feng-21}, Mandelbrot cascades \cite{BMPSPUM}, Baire typical measures \cite{BuS2,Bay}. This is still an active research field, with still open conjectures regarding the value of $D_\mu$ and the validity of the \ml formalism especially for self-similar and self-affine measures, and connexions with many  mathematical fields. 

The validity of the multifractal formalism reflects a deep relationship between local behaviors and geometrical properties of the measure $\mu$  and global statistics (the $L^q$-spectrum) computed on $\mu$ at various scales. In addition, this key relationship allows the multifractal spectrum to be estimated. Indeed, while the \ml spectrum is not computable on (discretized, finite) data,  several robust algorithms have been developed to estimate the $L^q$-spectrum, see \cite{Wendt2007d} and the references therein for instance. Then, when the \ml formalism holds, the \ml spectrum becomes accessible as the Legendre transform of the estimated $L^q$-spectrum. This \ml spectrum has  proved to be a relevant quantity  in signal and image processing, both as an analyzing  and as a classification method. It is a widely used tool in many applications, for instance physics and geophysics, EEG and ECG analyses, urban data, .. and it is important to continue developing random and deterministic models with properties fitting those of real  data signals and images.

\medskip

Our purpose here is to consider the simultaneous \ml analysis of measures, which is referred to as multivariate \ml analysis. More precisely, for a family of measures $(\mu_i)_{i=1,...,n}$ supported on $\zu$, this   consists first  in considering   the sets
\begin{equation}
\label{defemu1mu2}
\underline E_{\mu_1,...,\mu_n}(H_1,...,H_n)=\{x\in \zu: \dimu (\mu_i,x)=H_i \mbox{ for all }i\in\{1,...,n\}\}
\end{equation}
and then in computing the associated multivariate \ml spectrum
\begin{equation}
\label{defDmu1mu2}
D_{\mu_1,...,\mu_n}(H_1,...,H_n)= \dim \underline E_{\mu_1,...,\mu_n}(H_1,...,H_n).            
\end{equation} 

Not only  is it a natural question to analyze the $(\mu_i)$ simultaneously from a mathematical standpoint, but also it is a key issue in many applications where phenomena consist  naturally  of several signals/images. For instance, in neurosciences, EEG, scanners and MRI are built from several views of the same object; in geography, various measures and maps (population density, housing, revenues, altitude,...) are needed to fully describe an area or a city, An important challenge is then, on top of describing each set of data, to uncover their correlation and the way the influence each other.  The   multivariate \ml spectrum \eqref{defDmu1mu2}, and the  multivariate $L^q$-spectrum   associated with it, are   relevant indicators to achieve such studies by unfolding existing correlations between fractal and \ml parameters.

\smallskip

This line of research is quite recent, even if there is already a collection of mathematical results in this area.
This has been   investigated in \cite{BarrSau,BarSauSchmel}   for two Gibbs measures $\mu$ and $\nu$ invariant by the same dynamical system: there,  the estimation of the Hausdorff dimension of the sets $ \{ x\in \zu: \dim (\mu,x)=H_1,\dim(\nu,x)=H_2\}$ (where the limit local dimension is used) is achieved. Such studies are  more involved in general than the multifractal analysis of a single measure, for at least two reasons. First, since the level sets in \eqref{defemu1mu2} can be rewritten as $\underline E_{\mu_1,...,\mu_n}(H_1,...,H_n) = \bigcap_{i=1}^n \underline E_{\mu_i}(H_i)$, performing the multivariate multifractal  analysis amounts to investigating the size of intersections of (fractal-like) sets, which is known to be a delicate issue (see  for instance  Marstrand slicing theorems \cite{Marstrand54,Mattila75,Fal93,MATTILA} or recent progresses on intersections of Cantor sets \cite{Shmerkin_annals,Wu_annals}. Second,  there is no multifractal formalism yielding an easy upper bound for the bivariate spectrum \cite{Multivariate-PRSA,ACHA2018}.  Indeed, consider the natural extension of the $L^q$-spectrum \eqref{deftaumu} to the multivariate setting:
\begin{eqnarray}
\label{deftaumu3}
 \tau_{\mu_1,...,\mu_n}(q_1,...,q_n)  & = &  \liminf_{j\to +\infty} \tau_{\mu_1,...,\mu_n,j}(q_1,...,q_n), \\ \nonumber
 \ \mbox{ where } \   \tau_{\mu_1,...,\mu_n,j}(q_1,...,q_n) & =&  \frac{-1}{j} \log_2 \sum_{I\in \mathcal{D}_j}  \mu_1(3I)^{q_1}....\mu_n(3I)^{q_n}.
\end{eqnarray}
 The  Legendre transform  of $  \tau_{\mu_1,...,\mu_n}$ and the multivariate associated \ml formalism are defined as follows:
\begin{definition}
A family of probability measures $(\mu_i)_{i=1,...,n}$ is said to satisfy the   \ml formalism at $(H_1,...,,H_n)$ whenever 
\begin{equation}
\label{bi-form}
D_{\mu_1,...,\mu_n}(H_1,...,H_n) = \tau_{\mu_1,...,\mu_n}^*(H_1,...,H_n)  := \inf_{q_1,...,q_n\in \R} \Big(\sum_{i=1}^n q_iH_i  - \tau_{\mu_1,...,\mu_n}(q_1,...,q_n)  \Big).
\end{equation}
\end{definition}

As said before, the   multifractal formalism is an important issue, and a large literature is devoted to establishing its validity on typical, self-similar, self-affine, deterministic and random measures. However it is essentially only concerned with the univariate situation.

Barreira, Saussol and Schmeling studied some variational principles for a finite number of invariant measures valid for the limit local dimension. 
Olsen investigated in \cite{Olsen2,Olsen_2005} the situation where the $\mu_i$ are self-similar measures satisfying the Open Set Condition, and computed the spectrum $D_{\mu_1,...,\mu_n}$ but {again using the limit local dimensions}, not the lower lower dimension of $\mu_i$, which is an interesting but somehow simpler issue. In particular, he proves that a multivariate formalism holds with these quantities when the measures satisfy the Strong Open Set Condition - this will absolutely be not the case   in our situation. See also \cite{Meneveau90,Zhi} where some preliminary results are proved regarding multivariate analysis.

More generally, several authors have investigated the very close problem of determining, for a given compact $C\in \R^n$, the size of the set
$$\Big\{x\in \zu: A  \Big(\Big(\frac{\log \mu_1(B(x,2^{-j}))}{\log j}, ... ,\frac{\log \mu_n(B(x,2^{-j}))}{\log j}\Big)_{j\geq 1}\Big)\in C\Big\},$$
where the notation  $A(((z_{1,j},...,z_{n,j})_{j\geq 1}) )$ stands for the set of accumulation points of the sequence $((z_{1,j},...,z_{n,j})_{j\geq 1}) $ in $\R^n$. For instance, Attia and Barral \cite{Attia-Barral} found the dimension of such sets when the $\mu_i$ is the probability distribution associated with  branching random walks on trees. These situations do not cover our case, nor do they complete a similar analysis of the Legendre transform of the $L^q$-spectrum. More recently, results have been obtained in \cite{ACHA2018,Multivariate-PRSA} concerning multivariate  \ml analysis of families of functions and binomial measures, we will come back to these results below.

\mk
 
 As said before, our purpose is twofold. First we investigate the relationship between $ D_{\mu_1,...,\mu_n}$ and $ \tau^*_{\mu_1,...,\mu_n}$, and then we focus on a pair of random correlated measures, to illustrate how complex and interesting the bivariate \mk analysis can be.

\medskip

Our first result is that in general, the supports of $ D_{\mu_1,...,\mu_n}$ and $ \tau^*_{\mu_1,...,\mu_n}$ may even not intersect.
\begin{theorem}
\label{th-supports}
There exist two probability measures $\mu$ and $\nu$ supported by $\zu$ such that Supp($D_{\mu,\nu})\cap$Supp($\tau_{\mu,\nu}^*)=\emptyset$.
\end{theorem}
Theorem \ref{th-supports} immediately extends to higher-dimensional measures and to the multivariate case (with more than two measures), by just considering the measures $\mu$ and $\nu$ of the theorem as higher dimensional measures. Hence contrarily to what happens for one single measure, the Legendre transform of the $L^q$-spectrum does not provide an upper bound for the \ml spectrum - this is due to the fact that in the sum \eqref{deftaumu3} defining $ \tau_{\mu_1,...,\mu_n,j}$, the quantities $\mu_i(3I)^{q_i}$ may not be large (or small) simultaneously at the same generations. 

Let us mention that an analog of Theorem \ref{th-supports} was proved for functions in \cite{Multivariate-PRSA}. Observe also that   Theorem \ref{th-supports} does not disqualify  the Legendre spectrum as an irrelevant quantity, but rather emphasizes that for the \ml formalism to be verified, additional conditions (to be identified) should hold.

\mk

Our second objective is to perform   the bivariate multifractal analysis of a simple but rich model of  pair of random  correlated Bernoulli measures, hence enriching the very narrow family of  objects for which the bivariate \ml analysis is performed. Beyond its own interest (and the somewhat surprising results obtained below), this is justified by the fact that in many physical models, to detect particular events  one must simultaneously analyze two or more correlated signals, and it is a key question to determine how the correlation parameters can influence the bivariate Legendre and \ml spectra, and reciprocally whether one can deduce from the spectra some relevant information about the correlations between signals.
\medskip

Let us recall how the Bernoulli binomial measure  $\mup$ is built.  
 
Fix $p\in \zu$ and    $\mu_p (\zu)=1$. We proceed iteratively as follows. Assuming that $\mu_p(I)$ is fixed for every $I\in \mathcal{D}_j$, then calling respectively  $I_0 \in \mathcal{D}_{j+1}$ and $I_1\in \mathcal{D}_{j+1}$ the left  and right dyadic subintervals of $I$, one sets $\mu_p(I_0) = p\mu_p(I)$ and $\mu_p(I_1)=(1-p)\mu_p(I)$. This process defines iteratively the (unique) Borel probability measure $\mu_p$ on $\zu$, called the binomial measure with parameter $p$. The measure $\mu_p$ can also be viewed as the Gibbs measure invariant by the shift $x\mapsto 2x$ on the torus $\mathbb{T}^1$ associated with the potential $\phi(x) = p{\bf 1\!\!\!1}_{[0,1/2)}(x)+(1- p){\bf 1\!\!\!1}_{[1/2,1)}(x)$.

\mk

The complete bivariate multifractal analysis of $(\mu_{p_1},\mu_{p_2})$ is performed in \cite{ACHA2018,Multivariate-PRSA}, for every choice of the two  parameters   $p_1,p_2\in (0,1)$. It leads to surprising results: first, the multifractal spectrum drastically depends on the relative position of $p_1$ and $p_2$ with respect to 1/2. Second, there is no  obvious \ml formalism that holds for the pair  $(\mu_{p_1},\mu_{p_2})$.  
More precisely, in \cite{ACHA2018,Multivariate-PRSA} it is proved that not only the multifractal formalism $ D_{\mu_{p_1},\mu_{p_2}}= \tau^*_{\mu_{p_1},\mu_{p_2}} $ does not hold when $p_1<1/2<p_2$, but even worse the  inequality converse to the  one true in the univariate case   holds: $ \tau_{\mu_{p_1},\mu_{p_2}}^*(H_1,H_2) < D_{\mu_{p_1},\mu_{p_2}}(H_1,H_2)$ for most pairs $(H_1,H_2)$. Still, the supports of the Legendre transform $ \tau_{\mu_{p_1},\mu_{p_2}}^*$ and $ D_{\mu_{p_1},\mu_{p_2}}$ coincide, and one could think that there is a relationship between the two supports.

The randomness considered in this paper consists in changing, at random  generations, the parameter $p_1$ by a parameter $p_2$.  For instance, choosing $p_2=1-p_1$ amounts to  flip  at random generations the left and right sides of the construction in $\mu_{p_1}$.

More precisely, let $\eta\in (0,1)$ and $(Y_j)_{j\geq 1}$ be a sequence of i.i.d. random Bernoulli variables of parameter $\eta$ on a fixed probability space.
Consider the random set of indices:
\begin{equation}
\label{defA}
\mathcal{A} = \{j\in \N^*: Y_j=1\}. 
\end{equation}

Now fix $p_1,p_2\in (0,1)$, and  consider the random measure $\nue$ defined as follows. Starting with $\nue(\zu)=1$, we modify  the iterative construction of the binomial measure by changing the parameters at random generations as follows:
\begin{itemize}
\item
when $Y_j=1$, i.e. $j\in \mathcal{A}$, for every $I\in \mathcal{D}_{j-1}$,   the weight $p_1\nue(I)$ is assigned to the left subinterval of $\mathcal{D}_j$ of $I$, and $(1-p_1)\nue(I)$ to the right subinterval of $I$.
\item
when $Y_j=0$, i.e. $j\notin \mathcal{A}$, for every $I\in \mathcal{D}_{j-1}$,   the weight $p_2\nue(I)$  is assigned to the left subinterval of $\mathcal{D}_j$ of $I$, and $(1-p_2)\nue(I)$ to the right subinterval of $I$.
\end{itemize}
The limit object is a random probability measure $\nue$ supported on $\zu$, depending on $p_1$, $p_2$, and  $\eta$ via the sequence $(Y_j)_{j\geq 1}$. The multifractal analysis of $\nue$ itself is  quite standard, let us mention in particular \cite{FanWeighted} which considers weighted ergodic averages, of which the \ml analysis of $\nue$  can be viewed as a particular situation. We focus on the bivariate multifractal  properties of the pair ($\mu_{p_1},\nue)$.

Before stating our theorems, let us quickly recall standard results on the binomial measures and introduce some  notations (see also Section \ref{sec-bernoulli} for more information). For $p\in (0,1)$, the scaling function of the binomial measure $\mu_p$ is 
\begin{equation}
\label{taumup}
\mbox{for every  integer } j\geq 1, \ \ \tau_{\mu_p,j}(q) = \tau_{\mup}(q)=- \log_2(p^q+(1-p)^q),
\end{equation}
 the spectrum $D_{\mu_p}$ is concave  on the interval $[H_{p,\min},H_{p,\max}]\subset]0,+\infty[ $, where $(H_{p,\min},H_{p,\max}) = (-\log_2 p,\log_2(1-p))$ when $p>1/2$ and $(H_{p,\min},H_{p,\max}) = (-\log_2 (1-p),\log_2 p)$ when $p < 1/2$.  In addition, $D_{\mu_p}$ is real analytic on the interior of its support, $D_{\mu_p}(H_{p,\min}) = D_{\mu_p}(H_{p,\max}) = 0$,  and  $D_{\mu_p }$ is symmetric with respect to  $H_{p,s}:=(H_{p,\min}+H_{p,\max})/2$, where it reaches its maximum 1.  The dimension of $\mu_p$ (also its entropy) equals $\dim \mup =H_{p,e}:= -p\log_2p-(1-p)\log_2(1-p)$. Finally, $\mu_p$ satisfies the multifractal formalism:  $D_{\mu_p}(H) = \tau_{\mu_p}^*(H)$ for every $[H_{p,\min},H_{p,\max}]$.

\begin{center}
\begin{figure}
 \includegraphics[width=7.8cm,height=7cm]{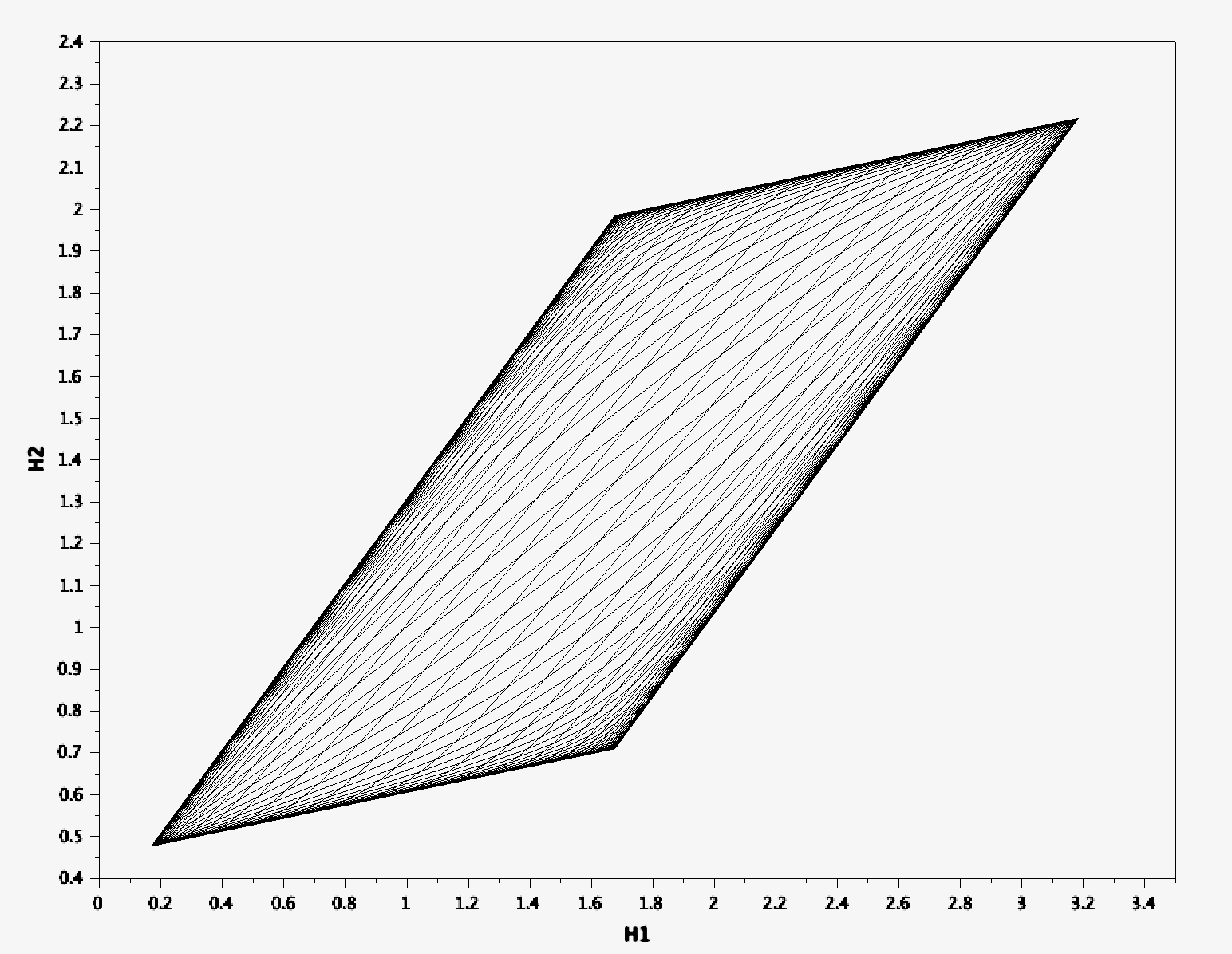} \  \ 
 \includegraphics[width=7.8cm,height=7cm]{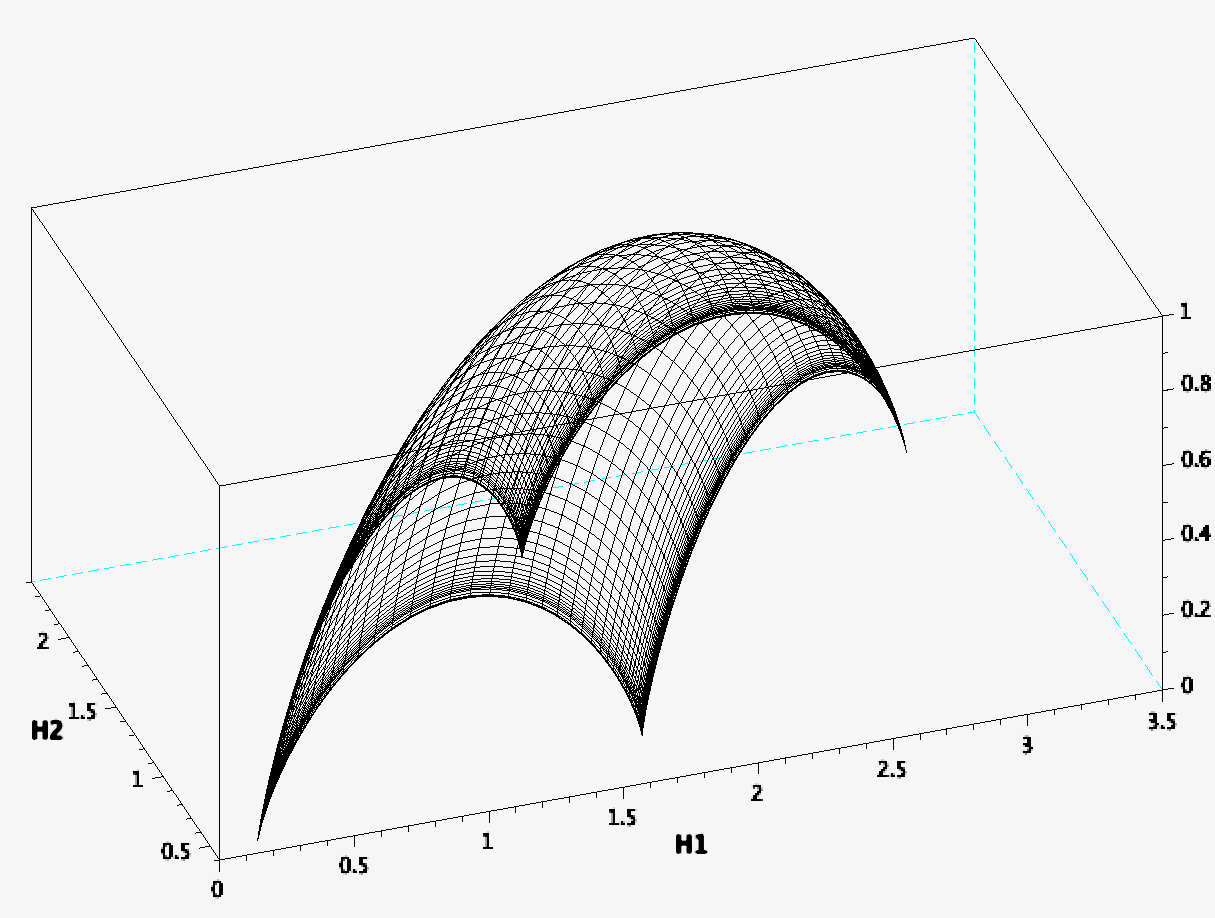}   
  \caption{Case  $0<p_1,p_2<1/2$ or  $1/2<p_1,p_2<1$: {\bf Left:} Support of $D_{\muu,\nue}$. {\bf Right:}   Bivariate multifractal spectrum $D_{\muu,\nue}$ of $(\mu_{p_1}, \nue)$.}
\label{fig-correlee-eta}
 \end{figure}
 \end{center}

Next, let us introduce for $q_1,q_2,H\in \R$ and $\eta\in [0,1]$ the following quantities: 
\begin{eqnarray*}
\label{defT}
   \Tau(q_1,q_2)  & =&  - \log_2 (p_1^{q_1}p_2^{q_2}+ (1-p_1)^{q_1}(1-p_2)^{q_2})\\
\label{defTe}
   \Taue(q_1,q_2)  & =&  \eta  \tau_{\mu_{p_1}}(q_1+q_2) + (1-\eta)   \Tau(q_1,q_2) \\
\label{defTet}
   \Tauet(q_1,q_2)  & =&    -q_1\log_2(1-p_1)-q_2( \eta \log_2p_1+(1-\eta)\log_2p_2) .
 \end{eqnarray*}

Our second   theorem deals with the case where the two parameters $(p_1,p_2)$ are both greater or both larger than 1/2.
\begin{theorem}
\label{mainth1}
Let $0<p_1,p_2<1/2$ or  $1/2<p_1,p_2<1$.
Let $\eta\in[0,1]$, and consider the pair of measures  $(\mu_{p_1}, \nue)$.
With probability one:
\begin{enumerate}

\item
The bivariate $L^q$-spectrum  of $(\mu_{p_1},\nue)$  is 
$
  \tau_{\mu_{p_1},\nue} = \Taue
$.

\sk \item
The pair $(\mu_{p_1}, \nue)$ satisfies the bivariate multifractal formalism everywhere on their common support, which is a deterministic parallelogram $\Pzeroeta$.
\end{enumerate}

\end{theorem}

Theorem \ref{mainth1} extends the results of  \cite{ACHA2018,Multivariate-PRSA} which correspond to $\eta=0$. Calling
\begin{eqnarray}
\label{def-coefg}
 \coefg & =& \frac{-\log_2(p_2)+\log_2(1-p_2)  }{-\log_2(p_1)+\log_2(1-p_1)  },
\end{eqnarray}
we will see that the support of $D_{\mu_{p_1},\nue}$ is the    parallelogram $\Pzeroeta$ generated by  four straight lines, those   with slopes $1$ or $\coefg$  passing by the points   $(H_{1,\min}, \eta H_{1,\min}+(1-\eta)H_{2,\min} )$ or 
$(H_{1,\max},  H_{1,\max}+(1-\eta)H_{2,\max}  )$, see   Section  \ref{sec-case1} for details, and  Figure \ref{fig-correlee-eta} for the graphs of $D_{\mu_{p_1},\nue}$ and $  \tau^*_{\mu_{p_1},\nue}$.

\sk
 
 The case where  $(p_1,p_2)$ are not on the same side of 1/2 is much more delicate.  
Observe that $ \coefg>0$ when $p_1$ and $p_2$ are on the same side of $1/2$, while  $ \coefg<0$ otherwise.
\begin{theorem}
\label{mainth2}
Let $0<p_1< 1/2< p_2<1/2$.
Let $\eta\in[0,1]$, and consider the pair of measures  $(\mu_{p_1}, \nue)$.

With probability one:
\begin{enumerate}

\item The bivariate $L^q$-spectrum  of $(\mu_{p_1},\nue)$  is  $ \tau_{\mu_{p_1},\nue} = \min( \Taue   ,     \Tauet)$, and the support of $ \tau_{\mu_{p_1},\nue}^* $ is is a deterministic pentagon $\mathcal{P}^\eta_1$.
\item
 The support of the bivariate \ml spectrum  $D_{\mu_{p_1},\nue}$    is  a deterministic pentagon $\mathcal{P}^\eta_2$ which is strictly larger  than $\mathcal{P}^\eta_1$ when $\eta\neq 0$ or $\eta\neq 1$.
\end{enumerate}

\end{theorem}

 The  values of   $D_{\mu_{p_1},\nue}$ and the pentagons are given in Section \ref{sec-case2}, and are plotted in Figures \ref{fig-pentagons}, \ref{fig-taueta2*}  and \ref{fig-taueta1*}.

 Besides the measures provided by Theorem \ref{th-supports}, the pair $(\mu_{p_1},\nue)$ is to our knowledge the first example  of pair of measures for which  the supports of $D_{\mu_{p_1},\nue}$ and $ ( \tau_{\mu_{p_1},\nue})^*$ are proved not to coincide. The proportion between the area of   $\mathcal{P}^\eta_1$  and  $\mathcal{P}^\eta_2\setminus \mathcal{P}^\eta_1$ depends on the parameter $\eta$, and somehow   characterizes  the correlation between $\mu_{p_1}$ and $\nue$.

\begin{corollary}
The pair $(\mu_{p_1}, \nue)$ does not satisfy the bivariate multifractal formalism everywhere. 
\end{corollary}

 The fact that the bivariate multifractal formalism does not hold is immediate as soon as the supports of $ \tau_{\mu_{p_1},\nue}^* $ and $D_{\mu_{p_1},\nue}$ do not coincide. Here, the two spectra do not even coincide  on a large part of $\mathcal{P}^\eta_1$ (which is the intersection of the supports), again depending on $\eta$, see Section \ref{sec-case2} for some discussions.  
  The pictures plotted in Figure \ref{fig-taueta1*} show that the two spectra have   distinct shapes, even if they coincide on some region of the plane. More precisely:
 \begin{itemize}
 \item
 The shape of the Legendre spectrum $ \tau_{\mu_{p_1},\nue}^* $ is quite easily understandable,    remembering that  $\tau_{\mu_{p_1},\nue}=\min( \Taue   ,     \Tauet)$. Indeed,  $ \tau_{\mu_{p_1},\nue}^* $ is composed in one region of a smooth part $(\Taue)^*$ (like in the case where $p_1$ and $p_2$ are located on the same side of $1/2$), and in a    second region of a cone with center $(H_{1,\min},\eta H_{1,\max}+(1-\eta)H_{2,\min})$ and tangent to $ (\Taue)^*$. 
 \item
 The shape of $D_{\mu_{p_1},\nue}$  results from various constraints on the possible simultaneous behaviors of $\mu_{p_1}$ and $\nue$ at a given point $t\in \zu$. It is worth noticing that many possible pairs $(H_1,H_2)$ are not seen by the Legendre spectrum, since $\mathcal{P}^\eta_1$ is strictly smaller than $\mathcal{P}^\eta_2$. 
 \item
 The bivariate spectrum   $D_{\mu_{p_1},\nue}$  lies above the Legendre spectrum  $ \tau_{\mu_{p_1},\nue}^* $ on a large part of  $\mathcal{P}^\eta_1$, so investigating the region where   the \ml formalism holds is not key here.  \end{itemize}
 
 The main reason justifying  these differences is that  for a given $t\in \zu$, it may happen that  the liminf local dimensions $ \dimu (\muu,x)$ and $\dimu (\nue,x)$   may reached at different scales, in the sense that $\ds \frac{\log \mu_{p_1}(B(t,r))}{\log r} \sim  \dimu (\muu,x)$ and  $\ds \frac{\log \nue(B(t,r'))}{\log r'} \sim  \dimu (\nue,x)$ at very different   radii $r$ and $r'$ (this does not happen when considering {\em limit} local dimensions). Then,   since the computation of $\tau_{\mu_{p_1},\nue}$ \eqref{deftaumu3} only involves the measures of dyadic intervals  at the same scale $2^{-j}$,  the possible non-simultaneity implies that  $\tau_{\mu_{p_1},\nue}$ cannot catch all possible  behaviors. 
 Describing and quantifying this variety of behaviors     is   the main issue in the following and in the computations.

\begin{center}
\begin{figure} 
 \includegraphics[width=7.8cm,height=6cm]{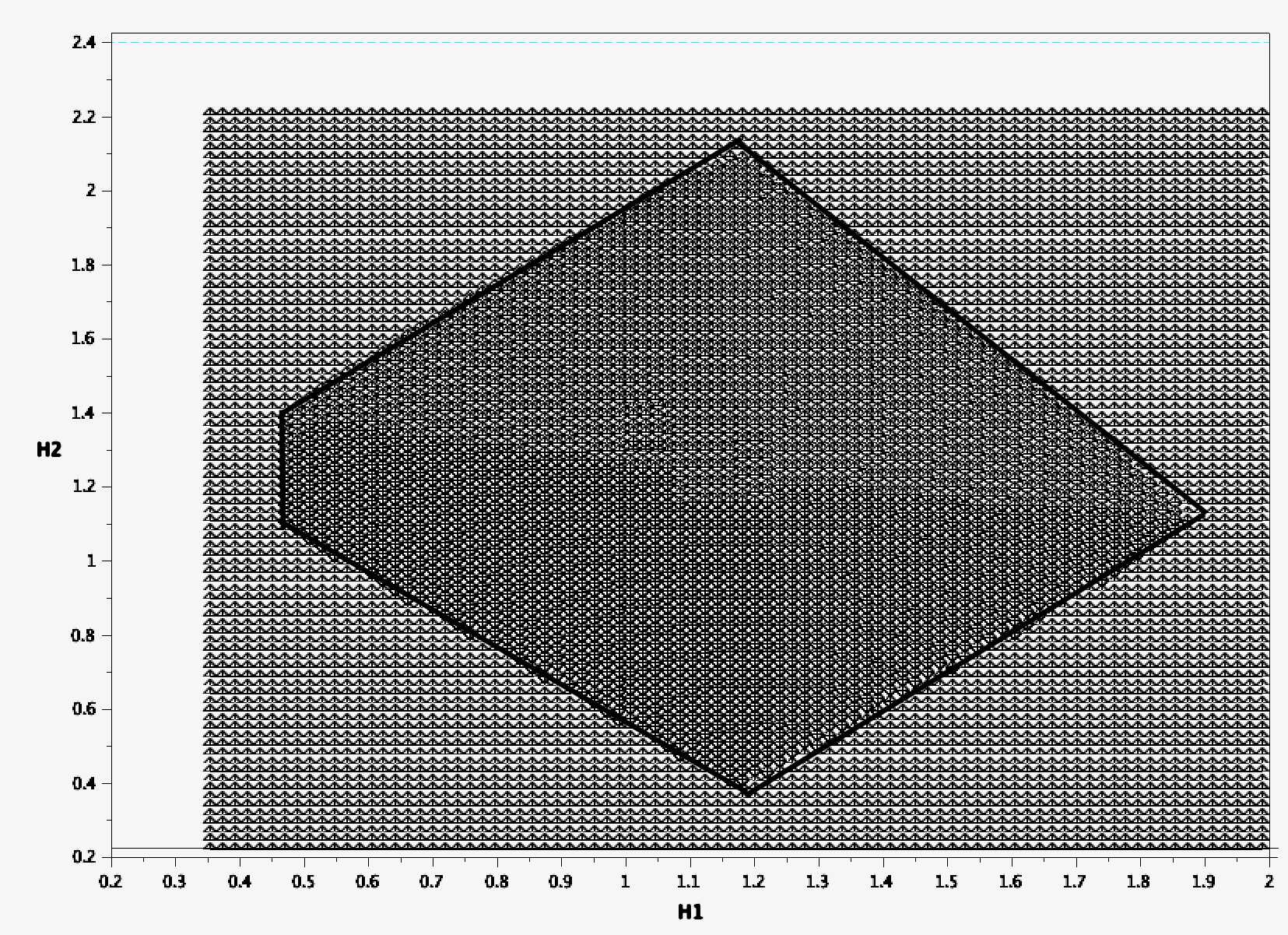} \   
 \includegraphics[width=7.8cm,height=6cm]{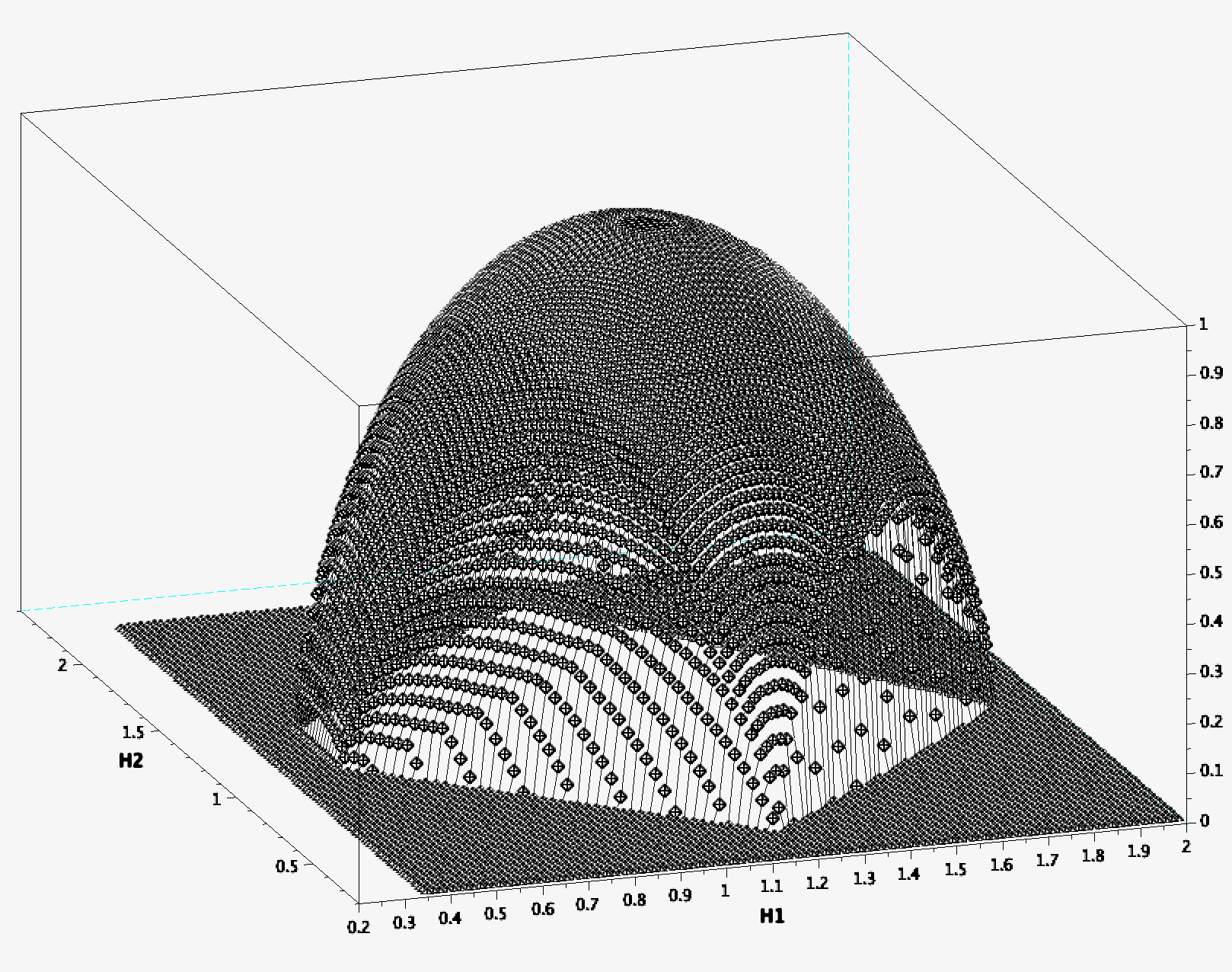}   \\ 
 \includegraphics[width=7.8cm,height=6cm]{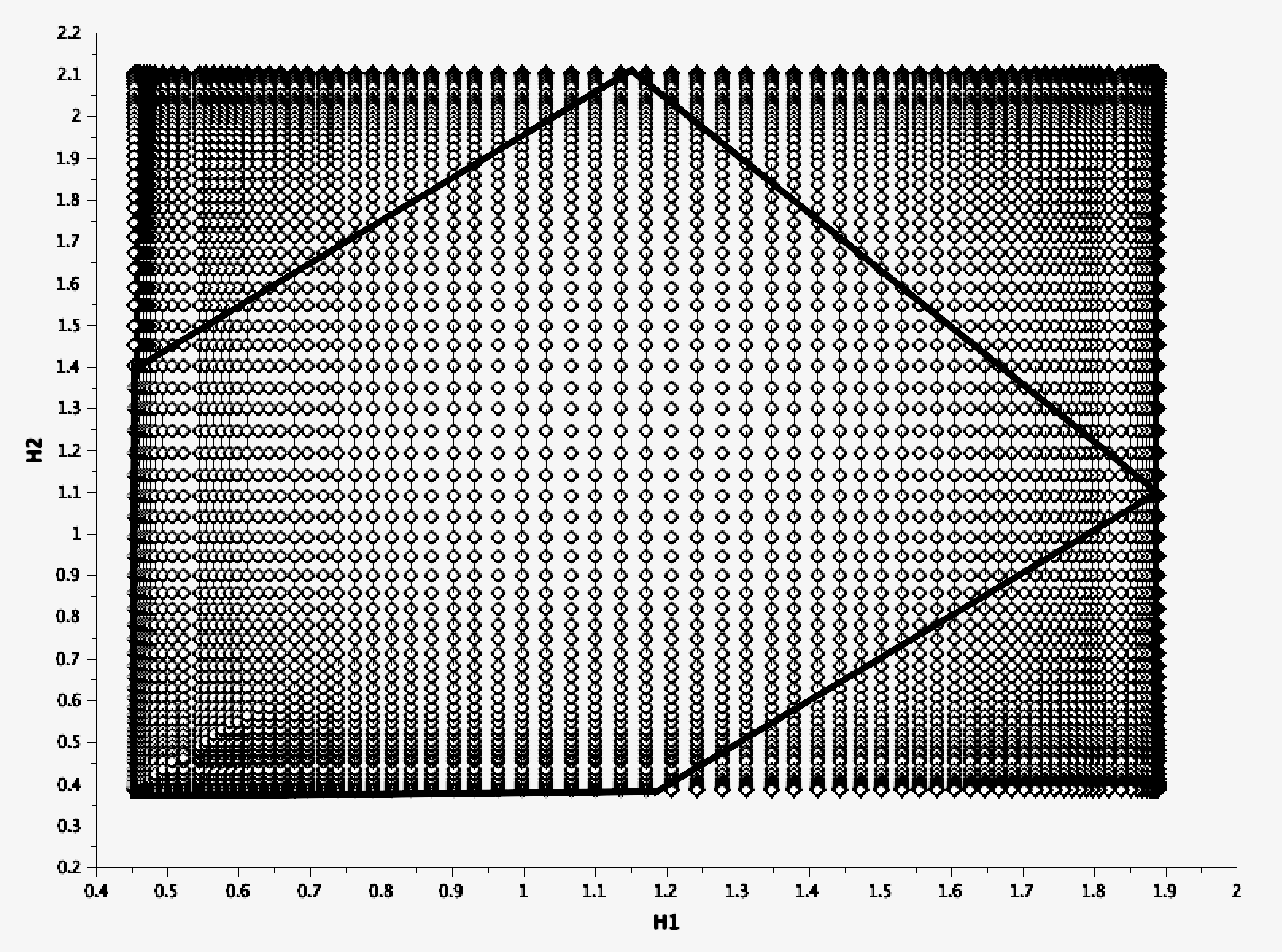} \   \includegraphics[width=7.8cm,height=6cm]{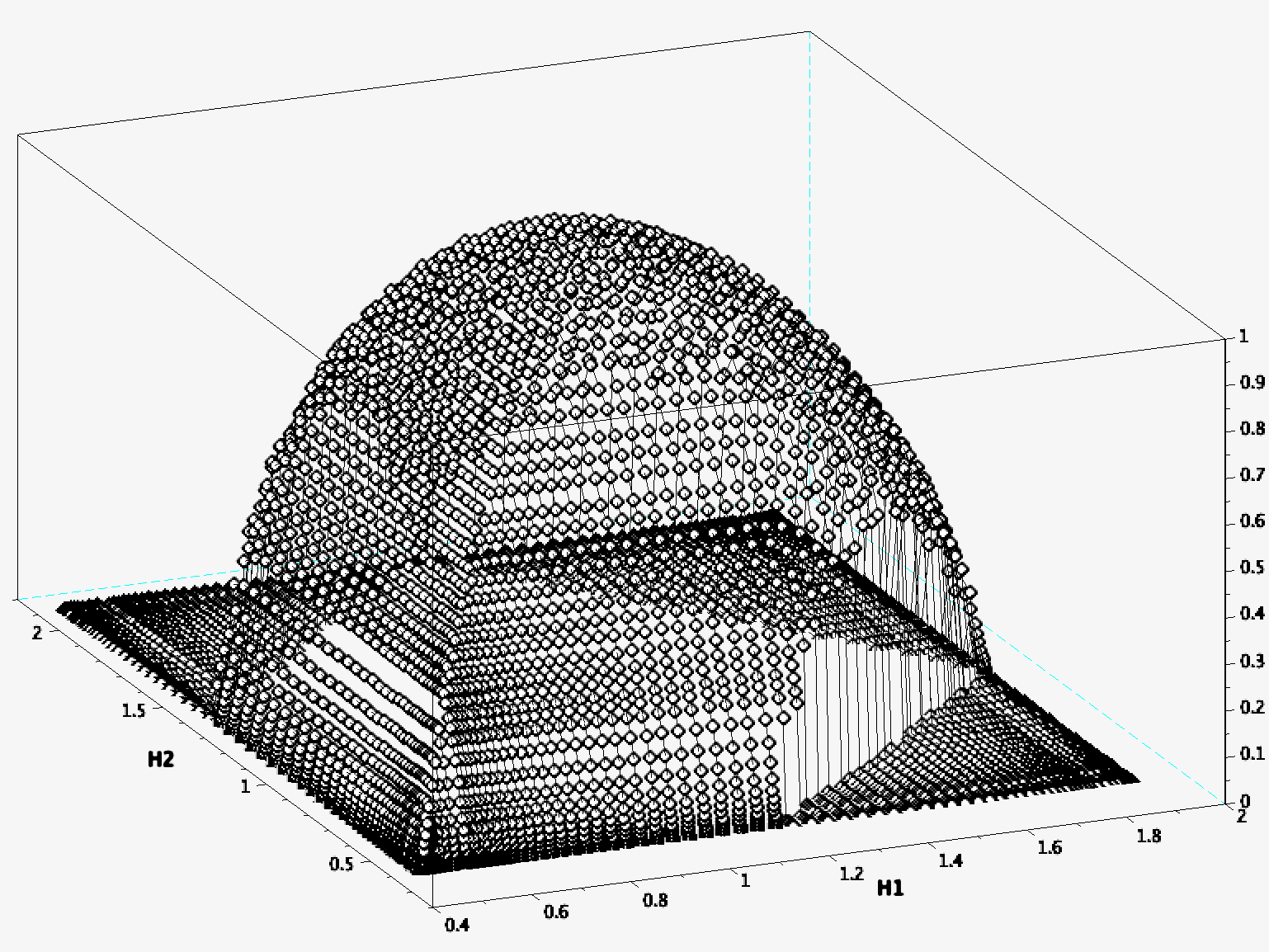}   \\ 
 \caption{Case $0<p_1=0.27<1/2<p_2=0.8$ and $\eta=0.5$. {\bf Top left:} Support of the Legendre spectrum $\tau_{\muu,\nue}^*$  {\bf Top right:} Legendre spectrum $\tau_{\muu,\nue}^*$. {\bf Bottom left:} Support of the multifractal  spectrum $D_{\muu,\nue}$  {\bf Bottom right:} Bivariate multifractal spectrum $D_{\muu,\nue}$. The supports are   distinct (but not disjoint) pentagons. }
\label{fig-taueta1*}
 \end{figure}
 \end{center}

The article is organized as follows.
 In Section \ref{sec-supports}, we set up the notations, and prove Theorem \ref{th-supports}. Next, in Section \ref{sec-bernoulli}, we recall useful facts on Bernoulli measures, and prove additional ones.  Section \ref{sec_preliminary} contains preliminary results useful for the rest of the article. Theorem \ref{mainth1} is proved in Section \ref{sec-case1}. Finally, Theorem \ref{mainth2}, which is much more delicate, is obtained in Section \ref{sec-case2}. A subtle joint analysis of the possible simultaneous behaviors of $\muu(I)$ and $\nue(I)$, for all dyadic intervals $I$,  is necessary. 
 
 \medskip
 
 Let us conclude by underlining the richness of such analyses, and the necessity to enrich the classes of objects for which the multivariate \ml analysis be performed.

\section{Some notations and proof of Theorem \ref{th-supports} }
\label{sec-supports}

\subsection{Notations} 

We adopt the  convention that  for   $a,b\in \R$, $[a,b]$ stands for the interval $[\min(a,b),\max(a,b)]$, and that $0^q=0$, for every $q\in \R$.

For every $j\geq 1$, let $\Sigma_j=(\{0,1\})^j$ be the set of words of length $j$ on the alphabet $\{0,1\}$, and   $ \Sigma=(\{0,1\})^\N$ and  $\Sigma^*=\bigcup_{j\ge 0} \Sigma_j$ be respectively the sets of infinite and finite words. The length of $w\in \Sigma^*$ is denoted by $|w|$, and $|w|=\infty $ when $w\in \Sigma$. When $w\in \Sigma^*$, we usually write $w=w_1\cdots w_{|w|}$ and $w\in \Sigma_{|w|}$, and $w=w_1\cdots w_n\cdots$ when $w\in \Sigma$. Then if $j\leq |w|$, $w_{|j}$ is the prefix of $w$ of length $j$. When $w\in \Sigma^*$, the cylinder $\lfloor w\rfloor$ is the set of words $w'\in \Sigma\cup\Sigma^*$ whose prefix of length $|w|$ equals $w$.

The concatenation of two words $w\in\Sigma^*$ and $w' \in \Sigma^* \cup\Sigma$ is denoted by $ww'$. 

%

We define  for $w\in \Sigma^*\cup \Sigma$
\begin{equation}
\label{defxw}
x_w=  \sum_{\ell=1}^{|w|} w_\ell 2^{-\ell} \in \zu
\end{equation}
and every real number  $x\in \zu$ can be written $x= x_w$ for some $w\in  \Sigma^*\cup \Sigma$, the decomposition being unique except for dyadic numbers $x_w$ for $w\in \Sigma^*$.

The dyadic intervals are encoded using words   $w\in \Sigma_j$ by 
$$
I_w= \Big[\sum_{\ell =1}^{j}w_\ell 2^{-\ell}, 2^{-j}+\sum_{\ell=1}^{j}w_\ell 2^{-\ell}\Big ],
$$
and we will talk indifferently about $I_w$ or   the cylinder $\lfloor w\rfloor$, and   write $I\in \Sigma_j$ or $I\in \mathcal{D}_j$ depending on the context.

For every interval $I\in \mathcal{D}_j$, we write 
$$I^+=I+2^{-j} \ \ \mbox{ and } \ \ I^-=I-2^{-j}.$$

For every $x\in \zu$, call $I_j(x)$ the unique dyadic interval  of $\Sigma_j$ containing $x$, $I_j ^+(x)=I_j(x)+2^{-j}$, $I_j ^-(x)=I_j(x)-2^{-j}$, and $\wI_j(x) =   I_j(x)\cup I_j^+(x) \cup I_j^-(x) $.

 For an interval $I=[a,b]$ and $\ep>0$, we use the notation
 $$I\pm\ep =I+[-\ep,+\ep] = [a-\ep,b+\ep].$$
 
 Finally, the dimension of  a Borel probability measure $\mu$ is 
 $$\dim(\mu) =\inf \{\dim_H(A): \mu(A)=1 \}.$$

\subsection{Proof of Theorem \ref{th-supports}}

We are going to  build two Borel probability measures $\nu_1$ and $\nu_2$   such that  
 $\support(D_{\nu_1,\nu_2})=\{ (1/2,1/2),(1/2,+\infty),(+\infty,1/2),(+\infty,+\infty) \}$, while  $\support(\tau^*_{\nu_1,\nu_2})$ is the triangle with edges $(1,1/2),(1/2,1),(1,1)$.
In particular, the two supports do not intersect.

For this, we will alternate between two simple construction schemes: 
\begin{itemize}
\item[(P1)]
Starting from a dyadic interval $I \in \Sigma_{2j}$ with  positive $\nu_i$-mass, the $\nu_i$-mass of every $J\in \Sigma_{2(j+1)}$ with $J\subset I$ is put to $\nu(I)/4$.

\item[(P2)]
Starting from a dyadic interval $I \in \Sigma_{2j}$ with  positive $\nu_i$-mass, the $\nu_i$-mass of the first and the last of the four subintervals  $J\in \Sigma_{2(j+1)}$ with $J\subset I$ is put to $\nu(I)/2$, and the mass of the two others is put to 0.
\end{itemize}

The first one mimics the Lebesgue measure, while the second one is an intermediary step leading to a Cantor set.

Let $(\ep_n:=2^{-(n+2)})_{n\geq 1}$, and  $(j^i_n)_{n\geq 1}$, for $i=1,...,4$, be four sequences of integers   tending  fast to infinity, whose values will be   given inductively along the construction. 

Start with $\nu_1(\zu)=\nu_2(\zu)=1$.
Let us assume that at generation $2j^4_n$, 
 for $i\in \{1,2\}$, for every $I\in \Sigma_{2j_4^n} $,  when $\nu_i(I)\neq 0$, {\bf $\nu_i(I)$ depends only on $i$} and 
 \begin{equation}
 \label{encadr-nu}
 2^{-2j_4^n(1+\ep_n)}= |I|^{1+\ep_n} \leq \nu_i(I) \leq |I|^{1-\ep_n}=2^{-2j_4^n(1-\ep_n)} .
 \end{equation}
In particular, the measures $\nu_i$ are "doubling at scale $2^{-2j_4^n}$", i.e. for every $I\in \Sigma_{2j_4^n}$ with $\nu(I)>0$, $\nu(I)\leq \nu(3I)\leq 3\nu(I)$. This implies that the two sums $  \sum_{I\in \Sigma_{2j^4_n} } \nu_1(I)^{q_1}  \nu_2(I)^{q_2} $ and $  \sum_{I\in \Sigma_{2j^4_n} } \nu_1(3I)^{q_1}  \nu_2(3I)^{q_2} $ are equivalent up to universal constants, and we shall   compute $  \tau_{ \nu_1, \nu_2} $ in this section by replacing $3I$ in formula \eqref {deftaumu3} by $I$.

Assume also that 
\begin{equation}
 \label{encadr-cardinal-0}
2^{ 2j^4_n(1-2\ep_n)}\leq  \#\{I\in \Sigma_{2j^4_n}: \nu_1(I)\nu_2(I)\neq 0\}  \leq 2^{2j^4_n(1+2\ep_n)}.
 \end{equation}

For $q\in \R$, let us write $\mbox{sgn}(q)=+1$ when $q\geq 0$ and $\mbox{sgn}(q)=-1$  when $q< 0$. 
The previous inequalities  yield 
\begin{align*}
  \sum_{I\in \Sigma_{2j^4_n} } \nu_1(I)^{q_1}  \nu_2(I)^{q_2}    & \geq    2^{2j^4_n(1-2\ep_n)}  2^{-2j_4^n(1 +\mbox{\tiny sgn}(q_1)  \ep_n)q_1} 2^{-2j_4^n(1 +\mbox{\tiny sgn}(q_2)  \ep_n)q_2}  \\
    \sum_{I\in \Sigma_{2j^4_n} } \nu_1(I)^{q_1}  \nu_2(I)^{q_2}    & \leq    2^{2j^4_n(1+2\ep_n)}  2^{-2j_4^n(1 -\mbox{\tiny sgn}(q_1)  \ep_n)q_1} 2^{-2j_4^n(1 -\mbox{\tiny sgn}(q_2)  \ep_n)q_2}  .
   \end{align*}
Hence,
 \begin{align}
 \label{encad-tau-1}
    |\tau_{\nu_1,\nu_2,2j^4_n}(q_1,q_2) -(q_1+q_2-1)| \leq  ( |q_1|+|q_2|+2) \ep_n.
    \end{align}

These conditions are satisfied at generation $j^1_0=j^2_0=j^3_0=j^4_0=0$. Next, we proceed iteratively  using 4 consecutive steps:

 \sk
 
{\bf Step 1:} 
As long as $j^4_{n} \leq j\leq j^1_{n+1}-1$, the two measures are constructed simultaneously on $\Sigma_{2j} $ by iterating scheme (P1) for $\nu_1$ and   (P2) for $\nu_2$. More precisely,  for every $I\in \Sigma_{2j} $ such that $\nu_1(I)>0$, apply (P1) to define $\nu_1$ on the intervals of $\Sigma_{2(j+1)}$, and similarly for every $I\in \Sigma_{2j} $ such that $\nu_2(I)>0$, apply (P2) to define $\nu_2$ on the intervals of $\Sigma_{2(j+1)}$.


Let us choose $j_{n+1}^1$ as the smallest integer such that $\frac {5j^4_n+1}{j^1_{n+1}}\leq \ep_{n+1}/2$.

At the last step, at generation $2j_1^{n+1}$:
\begin{itemize}
 \item
 for every $I\in \Sigma_{2j_{n+1}^1} $,  either $\nu_1(I)=0$ or  $\nu_1(I)=\nu_1(J)2^{-2(j_{n+1}^1-j^4_n)}$.
 Our choice for $j^1_{n+1}$ and \eqref{encadr-nu} give that 
 \begin{align}
 \label{encadr-nu1-1}
2^{-2j_{n+1}^1(1+\ep_{n+1})} &\leq 2^{-2j_4^n(1+\ep_n)} 2^{-2(j_{n+1}^1-j^4_n) } \leq  \nu_1(I)\\
\nonumber
&\leq  2^{-2j_4^n(1-\ep_n)} 2^{-2(j_{n+1}^1-j^4_n)}\leq 2^{-2j_{n+1}^1(1-\ep_{n+1})}.
\end{align}
 
Also, for every $j\in \{2j^4_n,...,2j^1_{n+1}\}$ and  $I\in \Sigma_{j}$, 
\begin{align}
 \label{encadr-nu1-11}
2^{-j(1+\ep_{n})} \leq \nu_1(I) \leq 2^{-j(1-\ep_{n})}.
\end{align}

\item
 for every $I\in\Sigma_{2j_{n+1}^1}$, either $\nu_2(I)=0$ or $\nu_2(I)=\nu_2(J) 2^{-(j_{n+1}^1-j^4_n)}$. Then
 \begin{align}
 \label{encadr-nu2-1}
2^{-j_{n+1}^1(1+\ep_{n+1})} &\leq 2^{-2j_4^n(1+\ep_n)} 2^{-(j_{n+1}^1-j^4_n) } \leq  \nu_2(I)\\
\nonumber
&\leq  2^{-2j_4^n(1-\ep_n)} 2^{-(j_{n+1}^1-j^4_n)}\leq 2^{-j_{n+1}^1(1-\ep_{n+1})}.
\end{align}

Also, for every $j\in \{2j^4_n,...,2j^1_{n+1}\}$ and  $I\in \Sigma_{j}$, 
\begin{align}
 \label{encadr-nu2-11}
2^{-j(1+\ep_{n})} \leq  \nu_2(I) \leq 2^{-j/2(1-\ep_{n})}.
\end{align}
\end{itemize}
Intuitively, at an interval $I$ of generation  $2j_{n+1}^1$ such that $\nu_1(I)\nu_2(I)\neq 0$,  one has $\nu_1(I)\sim |I|$ while $\nu_2(I)\sim |I|^{1/2}$.

\sk

Observe that by the scheme (P2) and by \eqref{encadr-cardinal-0}, 
\begin{align}
 \label{encadr-cardinal-1}
2^{ j^1_{n+1}(1-\ep_{n+1})} \leq 2^{ j^1_{n+1}-j^4_n} 2^{2j^4_n(1-2\ep_n) }& \leq  \#\{I\in \Sigma_{2j^1_{n+1}}: \nu_1(I)\nu_2(I)\neq 0\} \\
\nonumber
& \leq 2^{ j^1_{n+1}-j^4_n} 2^{2j^4_n(1+2\ep_n)}\leq 2^{ j^1_{n+1}(1+\ep_{n+1})}.
 \end{align}
 So,  even if there are intervals $I$ such that $\nu_1(I)\neq 0$ and $\nu_2(I)=0$ (or the opposite situation), they are heuristically not the most numerous ones.
 
 \sk

Similar computations as in the preliminary step give 
\begin{align*}
  \sum_{I\in \Sigma_{2j^1_{n+1} } } \nu_1(I)^{q_1}  \nu_2(I)^{q_2}    & \geq    2^{j^1_{n+1}  (1-\ep_{n+1})}  2^{-2j^1_{n+1} (1 +\mbox{\tiny sgn}(q_1)  \ep_{n+1})q_1} 2^{- j_1^{n+1}(1 +\mbox{\tiny sgn}(q_2)  \ep_{n+1})q_2}  \\
    \sum_{I\in \Sigma_{2j^1_{n+1}} } \nu_1(I)^{q_1}  \nu_2(I)^{q_2}    & \leq     2^{j^1_{n+1}  (1+\ep_{n+1})}  2^{-2j^1_{n+1} (1 -\mbox{\tiny sgn}(q_1)  \ep_{n+1})q_1} 2^{- j_1^{n+1}(1 -\mbox{\tiny sgn}(q_2)  \ep_{n+1})q_2}   .
   \end{align*}
Hence,
 \begin{align}
 \label{encad-tau-2}
    |\tau_{\nu_1,\nu_2,2j^1_{n+1}}(q_1,q_2) -(q_1+q_2/2-1/2 )| \leq  ( |q_1|+|q_2|+2) \ep_{n+1}.
    \end{align}
 
Also, for every $j\in \{2j^4_n,...,2j^1_{n+1}\}$,  
$$   \tau_{\nu_1,\nu_2,j } (q_1,q_2)  \in [q_1+q_2/2-1/2,  q_1+q_2-1] \pm( |q_1|+|q_2|+2) \ep_{n} ,$$
recalling that  $[a,b]$ stands for the interval $[\min(a,b),\max(a,b)]$. 

\mk

{\bf  Step 2:} Next we iterate the same  scheme (P1) for $\nu_1$ and  $\nu_2$ until generation $2 j^2_{n+1}$ as follows: for every  $j \in \{j^1_{n+1},...,j^2_{n+1}-1\} $, for every $I\in \Sigma_{2j} $ such that $\nu_i(I)>0$ ($i\in \{1,2\}$), apply (P1) to define $\nu_i$ on the intervals of $\Sigma_{2(j+1)}$. Again, at every generation $j\leq 2j^2_{n+1}$, the value of $\nu_i(I)$ (when $\nu_i(I)>0$) for $I\in \Sigma_{j}$ depends on $j$  only.

%

Set  $j_{n+1}^2$ as the  smallest integer such that $\frac {5j^1_{n+1}}{j^2 _{n+1}}\leq \ep_{n+1}/2$.

At generation $2j^2_{n+1}$, one has

\begin{itemize}
 \item
 for every $I\in \Sigma_{2j_{n+1}^2} $,  either $\nu_1(I)=0$ or  $\nu_1(I)=\nu_1(J)2^{-2(j_{n+1}^2-j_{n+1}^1)}$. Then 
 \eqref{encadr-nu1-1} yields   
 \begin{align}
 \label{encadr2-nu1-2}
2^{-2j_{n+1}^2(1+\ep_{n+1})} &\leq 2^{-2j_{n+1}^1(1+\ep_{n+1})} 2^{-2(j_{n+1}^2-j_{n+1}^1) } \leq  \nu_1(I)\\
\nonumber
&\leq  2^{-2j_{n+1}^1(1-\ep_{n+1})} 2^{-2(j_{n+1}^2-j_{n+1}^1)}\leq 2^{-2j_{n+1}^2(1-\ep_{n+1})}.
\end{align}
 
Also, for every $j\in \{2j^1_{n+1},...,2j_{n+1}^2\}$ and  $I\in \Sigma_{j}$, 
\begin{align}
 \label{encadr2-nu1-21}
2^{-j(1+\ep_{n})} \leq \nu_1(I) \leq 2^{-j(1-\ep_{n})}.
\end{align}

\item
 for every $I\in\Sigma_{2j_{n+1}^2}$, either $\nu_2(I)=0$ or $\nu_2(I)=\nu_2(J) 2^{-2(j_{n+1}^2-j_{n+1}^1)}$. Then
 \begin{align}
 \label{encadr2-nu2-2}
2^{-2j_{n+1}^2(1+\ep_{n+1})} &\leq 2^{-2j_{n+1}^1(1+\ep_{n+1})} 2^{-2(j_{n+1}^2-j_{n+1}^1) } \leq  \nu_2(I)\\
\nonumber
&\leq  2^{-2j_{n+1}^1(1-\ep_{n+1})} 2^{-2(j_{n+1}^2-j_{n+1}^1)}\leq 2^{-2j_{n+1}^2(1-\ep_{n+1})}.
\end{align}

Also, for every $j\in \{2j^1_{n+1},...,2j^2_{n+1}\}$ and  $I\in \Sigma_{j}$, 
\begin{align}
 \label{encadr2-nu2-21}
2^{-j(1+\ep_{n})} \leq  \nu_2(I) \leq 2^{-j/2(1-\ep_{n})}.
\end{align}
\end{itemize}
Intuitively, at an interval $I$ of generation  $j_{n+1}^2$, one has $\nu_i(I)\sim |I|$ for $i\in \{1,2\}$.

\sk

The same observations as in {\bf Step 1} give 
\begin{align}
 \label{encadr2-cardinal-2}
2^{2 j_{n+1}^2(1-\ep_{n+1})} =2^{2( j_{n+1}^2- j^1_{n+1})} 2^{j_{n+1}^1(1-2\ep_n) }& \leq  \#\{I\in \Sigma_{2j^2_{n+1}}: \nu_1(I)\nu_2(I)\neq 0\} \\
\nonumber
& \leq 2^{2(j_{n+1}^2- j^1_{n+1}) } 2^{2j_{n+1}^1(1+2\ep_n)}\leq 2^{2 j_{n+1}^2(1+\ep_{n+1})}.
 \end{align}
 and
 \begin{align*}
  \sum_{I\in \Sigma_{2j_{n+1}^2} } \nu_1(I)^{q_1}  \nu_2(I)^{q_2}    & \geq    2^{2j_{n+1}^2  (1-\ep_{n+1})}  2^{-2j_{n+1}^2 (1 +\mbox{\tiny sgn}(q_1)  \ep_{n+1})q_1} 2^{-j_{n+1}^2(1 +\mbox{\tiny sgn}(q_2)  \ep_n+1)q_2}  \\
    \sum_{I\in \Sigma_{2j^2_{n+1}} } \nu_1(I)^{q_1}  \nu_2(I)^{q_2}    & \leq     2^{2j_{n+1}^2  (1+\ep_{n+1})}  2^{-2j_{n+1}^2 (1 -\mbox{\tiny sgn}(q_1)  \ep_{n+1})q_1} 2^{- j_{n+1}^2(1 -\mbox{\tiny sgn}(q_2)  \ep_n+1)q_2}   .
   \end{align*}
Hence,
 \begin{align}
 \label{encad-tau-3}
    |\tau_{\nu_1,\nu_2,2j_{n+1}^2}(q_1,q_2) -(q_1+q_2-1 )| \leq  ( |q_1|+|q_2|+2) \ep_{n+1}.
    \end{align}

Also, for every $j\in \{2j^1_{n+1},..., 2j_{n+1}^2\}$ and  $I\in \Sigma_{j}$, 
$$   \tau_{\nu_1,\nu_2,j } (q_1,q_2)  \in [q_1+q_2/2-1/2,  q_1+q_2-1] \pm( |q_1|+|q_2|+2) \ep_{n} $$

\mk
 
{\bf  Step 3:}  Observe that we are back to the situation before applying {\bf Step 1}. Next, one proceeds as in Step 1, except that we apply (P2) to $\nu_1$ and (P1) to $\nu_2$. The computations are the same, except that the roles of $\nu_1$ and $\nu_2$ are switched. 

The same lines of computations give for $j^3_{n+1}$ large enough:
\begin{itemize}
 \item[(i)]
 for every $I\in \Sigma_{2j_{n+1}^3} $,  either $\nu_1(I)=0$ or 
  \begin{align}
 \label{encadr-nu1-3}
2^{-j_{n+1}^3(1+\ep_{n+1})} &\leq    \nu_1(I)  2^{-j_{n+1}^3(1-\ep_{n+1})}.
\end{align}
 
For every $j\in \{2j^2_{n+1},...,2j^3_{n+1}\}$, 
\begin{align}
 \label{encadr-nu1-31}
2^{-j/2(1+\ep_{n})} \leq \nu_1(I) \leq 2^{-j(1-\ep_{n})}.
\end{align}

\item[(ii)]
 for every $I\in\Sigma_{2j_{n+1}^3}$, either $\nu_2(I)=0$ or 
 \begin{align}
 \label{encadr-nu2-3}
2^{-j_{n+1}^3(1+\ep_{n+1})} &\leq   \nu_2(I) \leq 2^{-j_{n+1}^3(1-\ep_{n+1})}.
\end{align}
For every $j\in \{2j^2_{n+1},...,2j^3_{n+1}\}$, 
\begin{align}
 \label{encadr-nu2-31}
2^{-j(1+\ep_{n})} \leq  \nu_2(I) \leq 2^{-j/2(1-\ep_{n})}.
\end{align}
\item[(iii)] One has
\begin{align}
 \label{encadr-cardinal-2}
2^{ j^3_{n+1}(1-\ep_{n+1})}   \leq  \#\{I\in \Sigma_{2j^3_{n+1}}: \nu_1(I)\nu_2(I)\neq 0\}  \leq 2^{ j^3_{n+1}(1+\ep_{n+1})}.
 \end{align}
 \item One has
  \begin{align}
 \label{encad-tau-4}
    |\tau_{\nu_1,\nu_2,2j^3_{n+1}}(q_1,q_2) -(q_1/2+q_2-1/2 )| \leq  ( |q_1|+|q_2|+2) \ep_{n+1}.
    \end{align}
\item
For every $j\in \{2j^2_{n+1},...,2j^3_{n+1}\}$,  
$$   \tau_{\nu_1,\nu_2,j } (q_1,q_2)  \in [q_1/2+q_2-1/2,  q_1+q_2-1] \pm( |q_1|+|q_2|+2) \ep_{n} $$
 
\end{itemize}

\mk

{\bf  Step 4:}  One applies (P1) to the two measures, and for a suitable choice of $j^4_{n+1}$, the estimates \eqref{encadr2-nu1-2},  \eqref{encadr2-nu1-21}, \eqref{encadr2-nu2-2}, \eqref{encadr2-nu2-21}, \eqref{encadr2-cardinal-2} and \eqref{encad-tau-3}   hold at a generation $j^4_{n+1}$ chosen sufficiently large with respect to $j^3_{n+1}$. 

\mk


Observe that after {\bf Step 4}, we are back to the initial conditions (before applying {\bf Step 1}). Hence the  process can be iteratively applied and the two measures $\nu_1$ and $\nu_2$ are simultaneously constructed. Also, the two measures are by construction   doubling on their support, i.e. for every $I\in \Sigma_{j}$ with $\nu(I)>0$, $\nu(3I)\leq 3\nu(I)$.

\mk

One now collects the information:

\begin{center}
\begin{figure}
 \includegraphics[width=7.8cm,height=7cm]{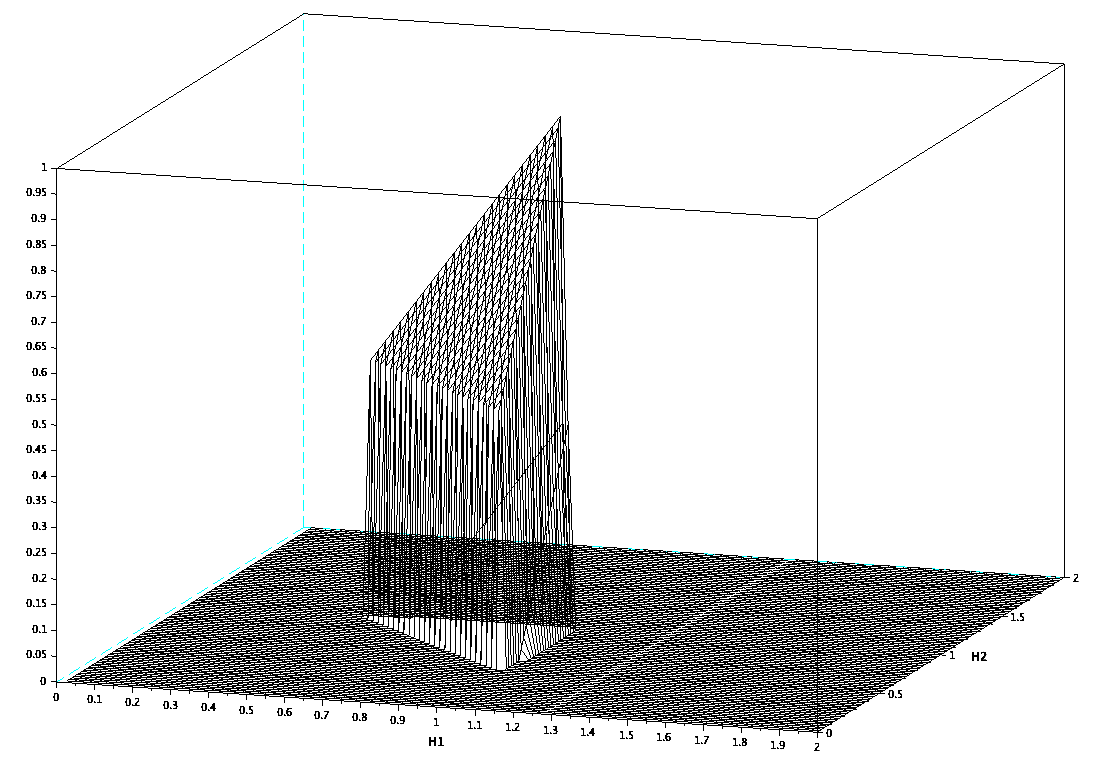}   
  \caption{Legendre transform of $\tau_{\nu_1,\nu_2}$.}
\label{fig-except}
 \end{figure}
 \end{center}

\begin{itemize}

\item
Let $i\in \{1,2\}$, $\ep>0$, and $x\in \supp(\nu_i)$. By construction, \eqref{encadr-nu2-1} and \eqref{encadr-nu2-3}, there exist  infinitely many generations $j =2 j^1_n$ ot $j= 2j^3_n$ such that $\ep_n\leq \ep$,   $ 2^{-j/2(1+\ep)}  \leq \nu_i(B(x,2^{-j}))  = \nu_i(I_j(x)) \leq 2^{-j/2(1-\ep)}$.

Also, for every $j$ large, $ 2^{-j/2(1+\ep)}\leq \nu_i(B(x,2^{-j}))  \leq 2^{-j(1-\ep)}$ by \eqref{encadr-nu1-11}, \eqref{encadr-nu2-11}, \eqref{encadr2-nu1-21}, \eqref{encadr2-nu2-21} and the same bounds hold at  {\bf Step 3} and {\bf 4}.
Hence $\underline{\dim}(\nu_i,x)=1/2$.

\sk
\item
As a consequence of the first item, the bivariate multifractal spectrum of $(\nu_1,\nu_2)$ is:
\begin{align*}
D_{\nu_1,\nu_2}(H_1,H_2) =\begin{cases}
  \dim(\supp(\nu_1)\cap  \supp(\nu_2) )& \mbox{ when }H_1=H_2=1/2\\
 \dim(\supp(\nu_1)\cap  \supp(\nu_2)^c) & \mbox{ when }H_1=1/2, H_2=+\infty\\
 \dim(\supp(\nu_1)^c\cap  \supp(\nu_2)) & \mbox{ when }H_1=+\infty, H_2=1/2\\
  \dim(\supp(\nu_1)^c\cap  \supp(\nu_2)^c )& \mbox{ when }H_1=H_2=+\infty.
 \end{cases}
\end{align*}
 Even if it not necessary for our purpose, one can check that $ \dim(\supp(\nu_1)\cap  \supp(\nu_2) )=1/2$, $ \dim(\supp(\nu_1)^c\cap  \supp(\nu_2)^c )=1$, $ \dim(\supp(\nu_i)\cap  \supp(\nu_j)^c) =1/2$ for $i\neq j$.

\sk
\item
 By \eqref{encad-tau-2}, \eqref{encad-tau-3}, \eqref{encad-tau-4} and \eqref{encad-tau-1} (which is also obtained after {\bf Step 4}) and the inequalities that follow these equations, and recalling that $\ep_n\to 0$ when $n$ tends to infinity, one deduces that
 \begin{align*}\label{tau-munu} 
 \tau_{\nu_1,\nu_2}(q_1,q_2)&=\min(q_1+q_2-1,q_1/2+q_2-1/2, q_1+q_2/2-1/2)\\
 &= \begin{cases}
\ q_1+q_2-1& \mbox{ when } q_1\leq 1 \mbox{ and } q_2\leq 1\\
\ q_1/2+q_2-1/2& \mbox{ when } q_1> 1 \mbox{ and }  q_2\leq q_1\\
\ q_1+q_2/2-1/2& \mbox{ when } q_2> 1 \mbox{ and }  q_1 <q_2.
 \end{cases}
 \end{align*}
 Hence,  $ \tau_{\nu_1,\nu_2}$ is made of three affine parts.  A quick computation shows that $ \tau_{\nu_1,\nu_2}^*(H_1,H_2) =( H_1+H_2-1)\cdot {\bf 1\!\!\!1}_{\mathcal{T}}(H_1,H_2)$, where $\mathcal{T}$ is the triangle with edges $(1/2,1)$, $(1,1/2)$, $(1,1)$, which does not intersect the support of $D_{\nu_1,\nu_2}$.
\end{itemize}

\sk

The pair $({\nu_1,\nu_2})$ satisfies the conclusions of Theorem \ref{th-supports}.

\section{Standard results on Bernoulli measures}
\label{sec-bernoulli}

\subsection{Large deviations and multifractal properties of Bernoulli measures}

Let  $\mu  $ be a probability measure on $\zu$.  

For $H\in\R^+$, in addition to $\underline E_\mu(H)$ defined in \eqref{eq-defdmu},  other level   sets associated with local dimensions are needed:
\begin{equation*}
 \overline E_\mu(H)=\left \{x\in[0,1] : {\overline \dim_\locloc}(\mu,x)=H\right \} \ \mbox{ and }   E_\mu(H)= \underline E_\mu(H)\cap   \overline E_\mu(H).
 \end{equation*}

Observe that, recalling the notations of the beginning of Section \ref{sec-supports},  for every $x\in \zu $, 
\begin{align*} 
{\underline \dim_\locloc}(\mu,x)& =\liminf_{j \to +\infty}\frac{\log \mu(\wI_j(x) ) }{ -j}=\liminf_{j \to +\infty}\frac{\log  \max (\wI_j(x) ) }{ -j}\\ 
{\overline \dim_\locloc}(\mu,x)& =\limsup_{j \to +\infty}\frac{\log \mu(\wI_j(x) ) }{ -j}=\limsup_{j \to +\infty}\frac{\log  \max (\wI_j(x) ) }{ -j}.
\end{align*}

\begin{definition}\label{defnewsets}
For any  Borel measure $\mu$ supported by $\zu$, and $H\geq 0$, define  
\begin{eqnarray*}
E^{\leq}_\mu(H) =\{x \in [0,1]: {\dim_\locloc}(\mu,x) \leq H\}  \ \text{ and }  \ \  E^{\geq}_\mu(H) =\{x \in [0,1]: {\dim_\locloc}(\mu,x) \geq H \}.
\end{eqnarray*}
The sets $\underline{E}^{\leq}_\mu(H)$, $\underline{E}^{\geq}_\mu(H) $, $\overline{E}^{\leq}_\mu(H)$, $\overline{E}^{\geq}_\mu(H) $ are defined similarly using the lower and upper local dimensions, respectively.
%

\end{definition}
%
%

\begin{definition} 
\label{defsmu}
Let $\mu$ be a probability measure on $\zu$. 
For every set $I\subset \R^+$ and     $j\geq 1$, set
$$\mathcal{E}_\mu(j,I) =
  \left\{ w\in \Sigma_j: \frac{\log_2 \mu(I_w)}{ {-j}} \in
I \right\}.$$
 \end{definition}

Standard results on binomial measures are the following, see  \cite{BRMICHPEY,Olsen,BBP}.  
\begin{proposition}
\label{fm}   
Let $p\in (0,1)$ and recall that 
$${H_{p,\min}}  \, \leq   \, H_{p,e}:=-p\log p-(1-p)\log_2(1-p)\leq   \, H_{p,s} := (  H_{p,\min}+H_{p,\min})/2  \, \leq    \, H_{p,\max} .$$

\begin{enumerate}
\smallskip\item   $\dim\mup =  H_{p,e} $ and $\mup(E_{\mu_p} (H_{p,e}))=1$.

\smallskip\item  For every $H\ge 0$ and $F\in\{E,\underline E,\overline E\}$,  one has 
$$  \dim F_{\mu_p} (H)=\tau_{\mu_p} ^*(H)=D_{\mu_p} (H),$$ 
with  $F_{\mu_p} (H) = \emptyset$ if and only if $D_{\mu_p} (H)=-\infty$. 

\smallskip \item  For every $H  \in [H_{p,\min},H_{p,s}]$ (i.e., in the increasing part of the  spectrum $D_\mup$),  one has  
$$\dim\,
E^{\leq}_\mup(H) =\dim\, \underline{E}^{\leq}_\mup(H) =\dim\,
\overline{E}^{\leq}_\mup(H) = D_{\mu_p} (H).$$

\smallskip \item  For every $H  \in [H_{p,s}, H_{p,\max}]$ (i.e. in the decreasing part of   $D_\mup$),  one has  
$$\dim\,
E^{\geq}_\mup(H)=\dim\, \underline{E}^{\geq}_\mup(H) =\dim\,
\overline{E}^{\geq}_\mup(H) =  D_\mup(H).$$
  
\mk \item
 For every $ H_{\min} \leq \alpha\leq\beta\leq H_{\max}$, there exists a probability measure $\mu_{\alpha,\beta}$ supported by $\zu$ such that $\mu_{\alpha,\beta}(\underline{E}_\mup(\alpha)\cap\overline{E}_\mup(\beta))=1$ and for $\mu_{\alpha,\beta}$-almost every $x$, 
 \begin{equation}
\label{eq-mualphabeta}
\lim_{r\to 0+} \frac{\log \mu_{\alpha,\beta}(B(x,r)) }{\log r}=\lim_{r\to 0+} \frac{\log \mu_{\alpha,\beta}(\widetilde I_j(x)) }{\log 2^{-j}}= \min(D_{\mu_p}(\alpha),D_{\mu_p}(\beta)),
\end{equation}
 where $ \widetilde I\in \{I,I^+,I^-\}$.
 
 \sk\item
For every $\ep>0$ and every interval {$I \subset \R_+$ such that $I\cap [H_{p,\min},H_{p,\max}]\neq \emptyset$, there exists an integer $J_{I}$ such that for every $j\geq J_{I}$,
$$\left| \frac{ \log_2  \mathcal{E}_\mup(j,I) }{ j } - 
 \sup_{H\in I} 
 D_\mup(H) \right| \leq  \ep  .  $$}

%

\end{enumerate}
\end{proposition}

Finally, recall the properties of the Legendre transform: 
\begin{itemize} 
\sk\item  If $H=\tau_\mup'(q)$, then $ \tau_\mup(q) = D_\mup ^* (q) =  q H  - D_\mup(H) = q \tau_\mup'(q) - D_\mup(\tau_\mup'(q)) $.
\sk\item
If  $q= D_\mup'(H)$, then  $ D_\mup(H) = \tau_\mup ^* (H)  =   H q - \tau_\mup(q) = H D'_\mup(H) - \tau_\mup(D'_\mup(H)) $.
\sk\item
One has  ${H_{p,\min}}=\tau_{\mu_p}'(+\infty) $, $ H_p = \tau_{\mu_p}'(0) $, and $ {H_{p,\max}}= \tau_{\mu_p}'(-\infty)$.

\end{itemize}
These relationships will be used repeatedly in the following.

Denote $H\pm \ep=[H-\ep,H+\ep]$ and recall that  $[a,b]\pm\ep=[a-\ep,b+\ep]$.

Item (4) is often used under  the following form.   For every $[a,b]\subset [H_{\min},   H_{\max}]$ and $ \wep>0$, there exists $\ep>0$  and  a
generation  $J_{\ep,\wep}$ such that $j \geq J_{\ep,\wep}$ implies
\begin{equation}
\label{upperbound-1}
\left| \frac{ \log_2 \#  \mathcal{E}_\mup(j,[a,b]\pm\ep)  }{j } - 
 \sup_{H\in [a,b]\pm\ep} 
D_\mup(H) \right| \leq \wep.
\end{equation}
 
\section{Preliminary results to Theorems \ref{mainth1} and \ref{mainth2}}
\label{sec_preliminary}
 
\newcommand\wip{\widetilde p}

The following observation are used repeatedly in the following. 

\begin{lemma}
\label{lem-trivial}
For  every $\alpha\in (H_{p_1,\min},H_{p_1,\max})$, call  $\wip_\alpha$   the unique real number such that $\tau_{\muu}'(\wip_\alpha) = \alpha$.

Then for every $j\geq 1$, for every $w\in \Sigma_j$, $(\muu(I_w) = 2^{-j\alpha}) \ \Leftrightarrow (\mud(I_w) = 2^{-j\G(\alpha)})\  \Leftrightarrow (\mu_{\wip_\alpha}(I_w) = 2^{-j D_{\muu}(\alpha)})$,
where 
\begin{eqnarray}
\label{defg}
  \G(H)  & =&  \coefg  (H+\log_2(1-p_1)) -\log_2(1-p_2).
  \end{eqnarray}
\end{lemma}
\begin{proof}
Write $w=(w_1,...,w_j)\in \Sigma_j$ and call $N_i(w)=\#\{1 \leq \ell \leq j:w_\ell=i\}$ for $i\in \{0,1\}$.
One has $\mu_{p_i}(I_w) = p_i^{N_0(w)}(1-p_i)^{N_1(w)} = 2^{-j [-\log_2(p_i)N_0(w)/j-\log_2(1-p_i)N_1(w)/j]}$. The first equivalence follows by  noting  that $N_0(w)+N_1(w) =j$  and for $\gamma\in \zu$, $\G( -\gamma \log_2(p_1)- (1-\gamma)\log_2(1-p_1) ) =  -\gamma\log_2(p_2)-(1-\gamma)\log_2(1-p_2)$.

The second equivalence follows similarly using  that $D_{\muu}(\alpha) = \alpha \wip_\alpha - \tau_{\muu}(\wip_\alpha)$.
\end{proof}

Using this lemma and Proposition \ref{fm} yields the next useful lemma.

\begin{lemma}
\label{lem-renouv}
Let $\ep>0$, $p_1, p_2\in (0,1)$. Fix $\alpha\in (H_{p,\min},H_{p,\max})$, and let $\wip_\alpha$ be such that $\tau_{\muu}'(\wip_\alpha)=\alpha$. An interval $I\in \Sigma_j$ satisfies the property $\mathcal{P}(p_1,p_2,\alpha,n,j)$ when for every $\widetilde I\in \{I,I^+,I^-\}$, 
\begin{equation}
\label{triplet}
 \begin{cases}  
\ \  2^{-j(\alpha +2^{-n-2})} \leq  \muu( \widetilde I )  & \leq  2^{-j(\alpha-2^{-n-2})}  \hspace{-4mm}\sk \\
\ \  2^{-j(\G(\alpha )+2^{-n-2})} \leq   \mud( \widetilde I ) & \leq2  ^{-j(\G(\alpha)-2^{-n-2})}  \hspace{-4mm} \\
\ \  2^{-j(D_{\muu}(\alpha)+2^{-n-2})}&  \leq  \mu_{\wip_\alpha}( \widetilde I ) \leq  2^{-j(D_{\muu}(\alpha)-2^{-n-2})}   \hspace{-4mm} \ \ .
 \end{cases}  
\end{equation}

There exists a sequence of integers $(J_{p_1,p_2,n,\alpha,\ep})_{n\geq 1}$   such that for every $n\geq 1$,
$$ \mu_{p_\alpha}\big( \mathcal{E}_{p_1,p_2,n,\alpha,\ep} \big)  \geq 1- \ep 2^{-n-2} ,$$
where   
 \begin{align}
 \label{eq-renouv}
\mathcal{E}_{p_1,p_2,n,\alpha,\ep}= \left \{  x\in \zu :     \ \forall \,j \geq J_{p_1,p_2,n,\alpha,\ep},  \   I_j(x) \mbox{ satisfies }  \mathcal{P}(p_1,p_2,\alpha,n,j) \right \} .
 \end{align}
 \end{lemma}
\begin{proof}
This directly  follows from   item (1) of Proposition \ref{fm}, Lemma \ref{lem-trivial} and the fact that $H_{\wip_\alpha,e}=D_{\muu}(\alpha)$.\end{proof}

\subsection{Cardinality of $\mathcal{A}$}

We first set some notations.

Let $\mathcal{B}\subset \N^*$ ($\mathcal{B}$ can be finite of infinite). Writing $\mathcal{B} = \{j_1,j_2,...\}$ with $j_1<j_2<...$, for every  $w\in\Sigma^*\cup \Sigma$, we set  
$$j^{\mathcal{B}} = \#\{1\leq \ell \leq j: \ell\in \mathcal{B}\} \ \ \mbox{ and } \ \ w^\mathcal{B}= (w_{j})_{j\in \mathcal{B}, j\leq |w|},$$
 the word that contains only the letters $w_j$ of $w$ with indices $j\in \mathcal{B}$. For instance, $w_{|j} = w^{\{1,2,\cdots,j\}}$.

For a dyadic interval $I\in \mathcal{D}_j$, the short notation $I^{\mathcal{B}}$   denotes the interval   $I^{\mathcal{B}}= I_{w^\mathcal{B}}$, where $w\in \Sigma_j$ is such that $I=I_w$.

 Then $I_j^{\mathcal{B}}(x)$,  $I_j^{+,\mathcal{B}}(x)$  and $I_j^{-,\mathcal{B}}(x)$ refer to $(I_j(x))^\mathcal{B}$, $(I_j^+(x))^\mathcal{B}$ and $(I_j(x)^-)^\mathcal{B}$, and  $I_j^{(n,\mathcal{B})}(x) = (I_j(x))^\mathcal{B} \cup (I_j^+(x))^\mathcal{B} \cup(I_j(x)^-)^\mathcal{B}$. There is a subtlety here, since $I_j^{(n,\mathcal{B})}(x) \neq   {(I_j^\mathcal{B}(x)) }^{(n)}$:     $  {(I_j^\mathcal{B}(x)) }^{(n)}$ is constituted by 3 consecutive intervals of generation $j^{\mathcal{B}}$, while $I_j^{(n,\mathcal{B}) }(x)$ can be made of 1, 2 or 3  such intervals since some of $I_j^{\mathcal{B}}(x)$,  $I_j^{+,\mathcal{B}}(x)$  and $I_j^{-,\mathcal{B}}(x)$ may coincide.

\medskip

We obtain a precise control of  the variations of the cardinality of  integers in $\mathcal{A}$ defined by \eqref{defA} (it is the set of random indices where $p_1$ is replaced by $p_2$ in the construction of $\nue$).

\begin{lemma}
\label{lem-maj-cardinal}
For every integers $j_1<j_2$, let 
\begin{equation}
\label{def-Amn}
 \mathcal{A}_{j_1,j_2}=\mathcal{A}\cap \{j_1,...,j_2\} .
 \end{equation}
 
With probability one, there exists  a positive sequence $(\ep_J)_{J\geq 0}$ tending to zero and an integer $J_0$ such that for every $J \geq J_0$, for every $j\geq  J^{\ep_J} $, 
\begin{equation}
\label{majA_j}
   | \#\mathcal{A}_{J,J+ j-1}   -  \eta j  | \leq  \sqrt{j \log (j)/\ep_J}  .
   \end{equation}
\end{lemma}

\begin{proof}

The random variables  $\{Y_{j}: j\in \{J,...,J+j-1\}\}$ are i.i.d. Bernoulli with parameter $\eta$. 
By Hoeffding's inequality, 
$$\mathbb{P}(| \mathcal{A}_{J, J+j-1} -  \eta  j | \geq  \sqrt{j\log (j)/\ep_J} ) \leq  2  \exp\left( -2  \sqrt{j\log (j)/\ep_J } ^2/j\right )  = \frac{2}{j^{2/\ep_J}} .$$ 
 So, assuming that $\ep_J\leq 1$, 
 $$\sum_{j\geq J^{ \ep_J}} \mathbb{P}(| \mathcal{A}_{J, J+j-1} - \eta j | \geq   \sqrt{j\log (j)/\ep_J} )  \leq 2  \sum_{j\geq J^{ \ep_J}}  \ \frac{2}{j^{2/\ep_J}}  \leq C (J^{ \ep_J }) ^{ 1- 2/\ep_J } = C  J^{ \ep_J  - 2  }  ,$$ 
the last estimate being true if $J^{ \ep_J}\to +\infty$ when $J\to +\infty$, for some absolute constant $C>0$.

Let us choose $\ep_J= (\log(J))^{-1/2}$. Then $J^{ \ep_J}\to +\infty$ but $J^{ \ep_J} = o(J^\ep)$ for every $\ep>0$.

One then has 
$$ \sum_{J\geq 1} \sum_{j\geq J^{ \ep_J}} \mathbb{P}(| \mathcal{A}_{J, J+j-1} - \eta j | \geq   \sqrt{j\log (j)/\ep_J} )  <+\infty,$$ 
and the Borel-Cantelli lemma yields the result. 
\end{proof}

Observe that \eqref{majA_j} immediately implies that  for every $J \geq J_0$ and every $j\geq  J^{\ep_J} $, 
$$
   | \#\mathcal{A}^c_{J,J+j-1}   -(1-\eta) j| \leq  \sqrt{j \log (j)/\ep_J}  .
$$

\subsection{A key decomposition, and dimensions of level sets}

 Recall that notations of the beginning of Section \ref{sec-supports}.
Writing $\mathcal{A}^c=\mathbb{N}\setminus \mathcal{A}$, it is quite clear from the previous definitions that for every $I\in \mathcal{D}$, 
\begin{equation}
\label{decomp-nu}
\nue(I) = \mu_{p_1}(\IA)\mu_{p_2}(\IAc),
\end{equation}
and obviously $\mu_{p_1}(I) = \mu_{p_1}(\IA)\mu_{p_1}(\IAc)$.
Subsequently, for each $x\in \zu$, in order to find $\dimu(\muu,x)$ and $\dimu(\nue,x)$,   the  behaviors of  $\mu_{p_1}(\IA_j(x))$, $\mu_{p_1}(\IAc_j(x))$,  and $\mu_{p_2}(\IAc_j(x)))$, as well as the same quantities for the neighbors intervals  $I_j(x)^+$ and  $I_j(x)^-$, must be simultaneously quantified.

The following proposition is key for determining the bivariate multifractal  spectrum of $(\mu_{p_1},\nue)$.
For every subset $\mathcal{B}$ of $\N$, $p\in (0,1)$ and $0\leq \alpha\leq\beta$, define
\begin{align*}
E^{p,\mathcal{B} }_{\alpha,\beta} & =  \Big\{x\in \zu:  \alpha \leq \liminf_{j\to+\infty}  \frac{\log \mu_{p}( I^{n,\mathcal{B}}_j(x))}{-j} \leq \limsup_{j\to+\infty}  \frac{\log \mu_{p}(  I^{n,\mathcal{B}}_j(x))}{-j} \leq \beta\Big\} .
\end{align*}

The computation of the Hausdorff dimension of such sets, and their intersection taking first $\mathcal{B}=\mathcal{A}$ and then $\mathcal{B}=\mathcal{A}^c$,  is very close to already existing results on the dimension of sets defined in Olsen  \cite{Olsen2,Olsen_2005,Attia-Barral}. However the previous results do not exactly cover our case, so we give a short proof of the following results.

\begin{proposition}
\label{prop_majdimE}
Let $\mathcal{B} \subset  \N$ be such that $\lim_{j\to +\infty} \frac{ \#\{\mathcal{B}\cap \{1,...,j\}\}}{j} =\weta \in [0,1]$. 

For $0\leq p_1,p_2\leq 1$,  $H_{1,\min} \leq \alpha_1 \leq \beta_1\leq H_{1,\max}$ and $H_{2,\min} \leq \alpha_2\leq \beta_2 \leq H_{2,\max}$,   one has
\begin{align}
\label{eq-maj3}
\dim E^{p_1,\mathcal{B}}_{\alpha_1,\beta_1} \cap E^{p_2,\mathcal{B}^c}_{\alpha_2,\beta_2}  & = \weta  \max_{ H\in [\alpha_1,\beta_1]}D_{\mu_{p_1}}(H)  + (1-\weta )  \max_{ H\in [\alpha_2,\beta_2]}D_{\mu_{p_2}}(H) .
\end{align}
\end{proposition}
%

\begin{proof}

 Call $F$ the intersection set in the left-hand side of \eqref{eq-maj3}, and $s_F$ the real number on the right-hand side.

The upper bound follows from counting arguments. Let $i\in \{1,2\}$ and $\wep>0$. 
 
When $x\in   F$,  there exists $\ep>0$ and an integer $J_\ep$ such that for $j \geq J_\ep$, 
$ 2^{-j^\BBB (\beta_1+\ep/2)}\leq {\mu_{p_1}}( I^{n,\mathcal{B}}_j(x)) \leq 2^{-j^\BBB (\alpha_1 -\ep/2)}$  and $ 2^{-j^\BBBc (\beta_2+\ep/2)}\leq {\mu_{p_2}}( I^{n,\BBBc}_j(x)) \leq 2^{-j^\BBBc (\alpha_2 -\ep/2)}$.   In particular, $x$ is at distance at most $2^{-j+1}$ from an interval $I\in F_j$, where
\begin{align}
\label{def-fj}
 F_j:=\left\{I\in \mathcal{D}_j:  \frac{\log \mu_{p_1}(I^\BBB) }{\log j^\BBB} \in [\alpha_1,\beta_1]\pm\ep  \ \mbox{ and } \frac{\log \mu_{p_2}(I^\BBBc) }{\log j^\BBBc} \in [\alpha_2,\beta_2]\pm\ep  \right\},
 \end{align}

By \eqref{upperbound-1}, for some $\wep>0$ (that depends on $\ep$ and tends to 0 when $\ep\to0$), one can enlarge $J_\ep>0$ so that for every $j$ with $j^\BBB\geq J_\ep$ and $j^\BBBc\geq J_\ep$,
\begin{align}
\label{eqmajE1}
 \frac{ \log_2 \#  \mathcal{E}_{\mu_{p_1}}(j^\BBB,[\alpha_1,\beta_1]\pm\ep)  }{j^\BBB } & \leq  
 \max_{H\in [\alpha_1,\beta_1]\pm\ep} 
D_{\mu_{p_i}}(H) + \wep\\
\label{eqmajE2}
\frac{ \log_2 \#  \mathcal{E}_{\mu_{p_2}}(j^\BBBc,[\alpha_2,\beta_2]\pm\ep)  }{j^\BBBc } & \leq  
 \max_{H\in [\alpha_2,\beta_2]\pm\ep} 
D_{\mu_{p_2}}(H) + \wep.
\end{align}
 
%

Hence, for $j$ large enough, one has \begin{align*}
&  \#  F_j \leq  2^{j^\BBB( \max_{H\in [\alpha_1,\beta_1]\pm\ep} 
D_{\mu_{p_1}}(H) + \wep) }2^{j^\BBBc( \max_{H\in [\alpha_2,\beta_2]\pm\ep} 
D_{\mu_{p_2}}(H) + \wep) } := 2^{js_{F,j,\ep }}.
\end{align*}
Since $\lim_{j\to +\infty} \frac{ \#\{\mathcal{B}\cap \{1,...,j\}\}}{j} =\weta \in [0,1]$,   and using the continuity of $H
\mapsto D_{\mu_{p_i}}(H)$, one has $\lim_{\ep \to 0} \lim_{j\to+\infty } s_{F,j,\ep} = s_F$ (recall that $\wep$ depends on $\ep$ and tends to 0 with $\ep$).

Fix $\ep_1>0$ small, and choose $\ep$, $\wep$ and enlarge $J_\ep$ so that for every $j\geq J_\ep$,  $s_{F,j,\ep} < s_F+\ep_1$. 

From the observations above, one deduces that  for every $J\geq 1$, $
F  \subset  \bigcup_{j\geq J}  \bigcup_{I\in F_j} (I\pm 2^{-j}) 
$. Hence, fixing $s> s_{F}+2\ep_1$  and $\delta>0$, one gets for $J\geq J_\ep$ so large that $2^{-J}\leq \delta$ that 
$$
\mathcal{H}^s_\delta( F )  \leq    \sum_{j\geq J}  \sum_{I\in F_j } |I\pm 2^{-j}|^s   \leq  C \sum_{j\geq J} 2^{j(s_{F,j\ep}-s)} $$
which is the remainder of a convergent series by our choice for $s$. Hence, taking limits when $\eta\to 0$ gives $\mathcal{H}^s( F) <+\infty$ and $\dim F  \leq s$. This holds for every such $s$, and then for every $\ep_1>0$, hence the result.

\medskip

Next, we move to the lower bound in \eqref{eq-maj3}. Call $H_1 $ and $H_2$  exponents such that  $D_{\mu_{p_1}}(H_1)= \max_{ H\in [\alpha,\beta]}D_{\mu_{p_1}}(H)  $
 and $D_{\mu_{p_2}}(H_2)= \max_{ H\in [\alpha',\beta']}D_{\mu_{p_2}}(H)  $. Calling $  F ':=E^{p_1,\mathcal{B}}_{H_1,H_1} \cap E^{p_2,\mathcal{B}^c}_{H_2,H_2}$, since $  F' \subset F $, it is then enough to prove that  $\dim   F' =\weta  D_{\mu_{p_1}}(H_1) +(1-\weta) D_{\mu_{p_2}}(H_2)$.
 
 Call $\nu_i$ the measure given by item (5) of Proposition \ref{fm} associated with $\mu_{p_i}$ and $\alpha_i=\beta_i=H_i$. Then consider the measure $\nu$ defined for every $I\in \mathcal{D}_j$ by $\nu(I)=\nu_1(I^{\BBB})\nu_2(I^{\BBBc})$.
 From its definition it follows directly that for  $\nu$-almost every $x\in \zu$:
 \begin{itemize}
 \item by \eqref{eq-mualphabeta},   $\lim_{r\to 0} \frac{\log \mu_{p_1}(I^{\BBB,n}_j(x))}{\log 2^{-j^\BBB}} =  H_1$ and  $\lim_{r\to 0} \frac{\log \mu_{p_2}(I^{\BBBc,n}_j(x))}{\log 2^{-j^\BBBc}} =  H_2$. So $x\in  F'$.
 
\item $\lim_{r\to 0} \frac{\log \nu_{1}(I^{n,\BBB}_j(x))}{\log 2^{-j^\BBB}} =  D_{\mu_{p_1}}(H_1)$ and  $\lim_{r\to 0} \frac{\log \nu_{2}(I^{n,\BBBc}_j(x))}{\log 2^{-j^\BBBc}} =  D_{\mu_{p_2}}(H_2)$, from which one deduces that   $\lim_{r\to 0} \frac{\log \nu(B(x,r))}{\log r}  =\weta  D_{\mu_{p_1}}(H_1) +(1-\weta) D_{\mu_{p_2}}(H_2) $. In particular,    Billingsley's lemma yields $\dim (\nu ) \geq  \weta  D_{\mu_{p_1}}(H_1) +(1-\weta) D_{\mu_{p_2}}(H_2) =s_F$.
 \end{itemize}
 
 As a conclusion, $\dim   F'\geq \dim (\nu) \geq s_F$, hence the result.
\end{proof}

\section {The case where $p_1,p_2$ are on the same side of 1/2}
\label{sec-case1}
 
\subsection {Computation of the bivariate $L^q$-spectrum}

Let us start by a preliminary computation. For two Borel probability measures  $\mu,\nu$ on $\zu$, define 
 \begin{equation}
\label{deftaumu4}
\widetilde \tau_{\mu,\nu }(q_1,q_2) = \liminf_{j\to +\infty}\frac{-1}{j} \log_2 \sum_{I\in \mathcal{D}_j} .\mu(I)^{q_1}\nu(I)^{q_2},
\end{equation}
The difference between $\widetilde \tau_{\mu,\nu}$ and $ \tau_{\mu,\nu}$ defined in \eqref{deftaumu}  is that  $\widetilde \tau_{\mu,\nu}$ considers $ \mu(I)$ and  $ \nu(I)$ instead of $ \mu(3I)$ and $ \nu(3I)$. In general, it is less relevant than $ \tau_{\mu,\nu}$ since it is highly dependent on the choice of the dyadic basis (while $\tau_{\mu,\nu}$  is not because the intervals  $3I$ overlap a lot).

\begin{lemma}
\label{lem-estimtau1}
With probability one, for every choice of $p_1,p_2 \in (0,1)$, one has $\widetilde \tau_{\muu,\mud} =  \Tau $ and $\widetilde \tau_{\muu,\nue} = \Taue$. Also, the liminf in \eqref{deftaumu4} and \eqref{deftaumu3}
 are limits.
\end{lemma}
\begin{proof}
The result for $\widetilde \tau_{\muu,\mud} $ is obtained in \cite{ACHA2018}, we reproduce it quickly. Call for $w\in \Sigma_j$ 
\begin{align*}
N_w^0=\# \{1\leq \ell\leq j: w_\ell=0 \} \ \mbox{ and } \ N_w^1=\#\{1\leq \ell\leq j: w_\ell=1\}=j- N_w^0.
\end{align*}
Then 
\begin{align*}
 \sum_{I\in \mathcal{D}_j} \muu(I)^{q_1}\mud(I)^{q_2} & =  \sum_{w\in \Sigma_j} \muu(I_w)^{q_1}\mud(I_w)^{q_2}\\
 & =    \sum_{w\in \Sigma_j}  (p_1)^{q_1  N_w^0}(1-p_1)^{q_1 N_w^1}  (p_2)^{q_2 N_w^0}(1-p_2)^{q_2 N_w^1} \\
 & = (p_1^{q_1}p_2^{q_2} + (1-p_1)^{q_1}(1-p_2)^{q_2})^j = 2^{-j\Tau(q_1,q_2)},
  \end{align*}
  hence the result.
Next, one sees that 
\begin{align*}
 \sum_{I\in \mathcal{D}_j} \muu(I)^{q_1}\nue(I)^{q_2} & =  \sum_{I\in \mathcal{D}_j} \muu(\IA)^{q_1+q_2}\muu(\IAc)^{q_1} \mud(\IAc)^{q_2} \\
 & =  \sum_{I\in \mathcal{D}_{j^{\mathcal{A}}}}\muu(I)^{q_1+q_2}  \sum_{I\in \mathcal{D}_{j^{\mathcal{A}^c}} } \muu(I)^{q_1} \mud(I)^{q_2} \\
 & = 2^{-\jA  \tau_{\muu}(q_1+q_2) -\jAc \Tau(q_1,q_2)}.
 \end{align*}
By Lemma \ref{lem-maj-cardinal}, noting that $\jA=\#\mathcal{A}_{1,j}$ and $\lim_{j\to+\infty} \frac{\jA}{j}=(1-\eta)$ with probability 1, one concludes that  
$$\lim_{j\to+\infty} \widetilde \tau_{\muu,\nue,j}(q_1,q_2) = (1-\eta)  \tau_{\muu}(q_1+q_2) +\eta \Tau(q_1,q_2) =\Taue(q_1,q_2).$$
\end{proof}

Observe that Lemma \ref{lem-maj-cardinal} implies that almost surely, for some constant $C>0$,
\begin{equation}
\label{taux-conv-tau}
| \widetilde \tau_{\muu,\nue,j}(q_1,q_2) -\Taue(q_1,q_2)|\leq C \left( \frac{\log j}{j}\right)^{1/2}.
\end{equation}

Recall the convention that $[\alpha,\beta]$ is the segment with endpoints $\alpha$ and $\beta$, even if $\beta<\alpha$.

\begin{lemma}
\label{lem-legendre1} For every choice of $p_1$ and $p_2$, the Legendre transform of  the mapping $\Tau$ is supported by the segment $[  (-\log_2(1-p_1) ,  \G(- \log_2(1-p_1))), 
(- \log_2 p_1,  \G(- \log_2 p_1 )]$, and for every    $H\in [- \log_2(1-p_1), - \log_2 p_1]$, $\Tau^*(H,  \G(H))= D_{\mu_1}(H)$.

\end{lemma}
\begin{proof}
For every $(q_1,q_2)\in \R^2$, let us write 
\begin{align*}
H_1 & = \frac{\partial}{\partial q_1} \Tau(q_1,q_2)   = \frac{-p_1^{q_1}p_2^{q_2} \log_2(p_1)-(1-p_1)^{q_1} (1-p_2)^{q_2} \log_2(1-p_1)}{p_1^{q_1}p_2^{q_2}+(1-p_1)^{q_1} (1-p_2)^{q_2}} \\
H_2& = \frac{\partial}{\partial q_2} \Tau(q_1,q_2)  = \frac{-p_1^{q_1}p_2^{q_2} \log_2(p_2)-(1-p_1)^{q_1} (1-p_2)^{q_2} \log_2(1-p_2)}{p_1^{q_1}p_2^{q_2}+(1-p_1)^{q_1} (1-p_2)^{q_2}}=   \G(H_1).
\end{align*}
There is an affine relationship between the two partial derivatives, so the support of $\Tau^*$ is  the segment $[- \log_2(1-p_1),  \G(- \log_2(1-p_1))), 
(- \log_2 p_1 ,  \G(- \log_2 p_1))] $, which coincides with the segment  $[(- \log_2(1-p_1),  - \log_2(1-p_2)), 
(- \log_2 p_1,  - \log_2 p_2)] $ when $p_1,p_2$ are on the same side of $1/2$, and equals $[(- \log_2(1-p_1) ,  - \log_2 p_2), 
(- \log_2 p_1,  - \log_2(1-p_2)] $ when 1/2 is located in the middle of $p_1$ and $p_2$.

Next, fix some $H  \in (- \log_2(1-p_1),- \log_2 p_1)$. Consider the unique $q \in \R$ such that $H =\tau_{\mu_1}'(q )= \frac{-p_1^{q } \log_2(p_1) -  (1-p_1)^{q}  \log_2(1-p_1)}{p_1^{q }+(1-p_1)^{q }} = - x \log_2(p_1) -(1-x)\log_2(1-p_1)$, with $x= \frac{p_1^{q } }{p_1^{q }+(1-p_1)^{q }} $. 
Observe that   combined with the definition of $H_1$ above, this remark implies that  
\begin{equation}
\label{fix-x}
x= \frac{p_1^{q} }{p_1^{q}+(1-p_1)^{q}} =\frac{p_1^{q_1}p_2^{q_2}  }{p_1^{q_1}p_2^{q_2}+(1-p_1)^{q_1} (1-p_2)^{q_2}} .
\end{equation}
In particular, after simplification, 
$$q\log_2 \frac{1-p_1}{p_1}= \log_2 ((1-p_1)^{q_1} (1-p_2)^{q_2})-\log_2 p_1^{q_1}p_2^{q_2}.$$

By Legendre transform, one has 
$$D_{\mu_1}(H) = qH-\tau_{\mu_{p_1}}(q )=sH+ \log_2(p_1^{q}+(1-p_1)^{q}) = q(H-\log_2p_1) +\log_2x  .$$

The basic properties of the 2-dimensional Legendre transform imply that the minimum in $\Tau^*(H,  \G(H) )   = \inf_{q_1,q_2\in \R} (q_1H +q_2   \G(H) -\Tau(q_1,q_2)) $ is reached at  the pairs $(q_1,q_2)$ for which   $\frac{\partial}{\partial q_1} \Tau(q_1,q_2)  =H$ and $\frac{\partial}{\partial q_2} \Tau(q_1,q_2) =  \G(H)$.
Hence, using \eqref{fix-x} and recalling \eqref{defg},
\begin{align*}
\Tau^*(H,  \G(H) ) & = \inf_{q_1,q_2\in \R: \eqref{fix-x} \mbox{\tiny holds true}} (q_1H +q_2   \G(H) -\Tau(q_1,q_2))\\
& = \inf_{q_1,q_2\in \R: \eqref{fix-x} \mbox{\tiny holds true}} (q_1(H +\log_2 p_1)+q_2(   \G(H)+\log_2 p_2)   +\log_2 x)\\
& = \inf_{q_1,q_2\in \R: \eqref{fix-x} \mbox{\tiny holds true}} ((q_1+ q_2\coefg  )(H +\log_2p_1 )   +\log_2 x).
\end{align*}
 To conclude, one observes that 
\begin{eqnarray}
\label{equa-q}
q_1+ q_2\coefg  &= & q_1+q_2\frac{\log(p_2/(1-p_2))}{\log(p_1/(1-p_1))}= \frac{q_1\log(p_1/(1-p_1))+q_2 {\log(p_2/(1-p_2))} }{\log(p_1/(1-p_1))}\\
& = & \frac{\log p_1^{q_1}p_2^{q_2} +\log_2 (1-p_1)^{q_1}(1-p_2)^{q_2}}{ \log_2(p_1/(1-p_1))} =q,
\nonumber
\end{eqnarray}
which implies that $T^*(H,\G(H))=D_\muu(H)$, hence the result.  
 \end{proof}

The support  of   $\Taue$ is  a parallelogram, this will follow in particular from the multifractal formalism in Section \ref{sec-bi-form-simple}.

\mk

The computation of $ \tau_{\muu,\nue}$ is a bit more delicate, and   depends on the relative positions of $p_1$ and $p_2$.

\begin{lemma}
\label{lem-estimtau2}
Assume that $0<p_1,p_2 <1/2 $ or $1/2<p_1,p_2 <1 $. Then  $  \tau_{\muu,\nue}  = \Taue$.
\end{lemma}
\begin{proof}
Without loss of generality, we assume $p_1,p_2\leq 1/2$.

For $w\in \Sigma_j$  one has  $\muu(3I_{w}) =  \muu(I_{w}^-) +\muu(I_{w}) +\muu(I^+_{w}) $.   Call $\widetilde w\in \Sigma_j$ the   word    for which $\muu(I_{\widetilde w})$ is maximal amongst  the three terms. Hence $\muu(I_{\widetilde w})\leq \muu(3I_{w})  \leq 3 \muu(I_{\widetilde w})$. 

Observe that the fact that $\muu(I_{\widetilde w})$ is maximal depends only on the numbers of zeros and ones in  $w$ and $\widetilde w$ in $\Sigma_j$.
 Since $p_1 $ and $p_2$ are both less than 1/2 (meaning that $p_i\leq 1-p_i$),   the same inequality holds true for $\nue$:  $\nue(I_{\widetilde w})$ is maximal amongst $\nue(I_{w'})$, $\nue(I_{w})$, and $\nue(I^+_{w}) $, and $\nue(I_{\widetilde w}) \leq \nue(3I_{w})  \leq 3 \nue(I_{\widetilde w})$.

Subsequently,  for a constant $C$ depending on the values (and on the sign) of $q_1$ and $q_2$, 
\begin{align*}
 \sum_{w\in \Sigma_j}  \muu(3I_w)^{q_1}\nue(3 I_w)^{q_2} &   \leq C  \sum_{w\in \Sigma_j} \muu(I_w)^{q_1} \nue( I_w)^{q_2}  = C \cdot 2^{-j \widetilde \tau_{\muu,\nue,j}(q_1,q_2)}.
 \end{align*}

Next, for every $w\in \Sigma_j$, setting $w'=w01\in \Sigma_{j+2}$, then $3I_{w'}\subset I_w$ and $ p_1(1-p_1)  \muu (I_{w}) =   \mu(I_{w'})\leq \muu(3I_{w'}) \leq \muu (I_{w}) $, and similarly for some constant $C$ that depends on $p_1$ and $p_2$, $ C \nue(I_{w}) \leq   \nue(3I_{w'}) \leq \nue(I_{w}) $. Hence, for some other constant $C'$ that depends on the $p_i$ and $q_i$, one gets
   $$ 2^{-j \widetilde \tau_{\muu,\nue,j}(q_1,q_2)} =   \sum_{w \in \Sigma_j} \muu(  I _w)^{q_1} \nue( I _w )^{q_2}  \leq C'  \sum_{w' \in \Sigma_{j+2} } \muu(3 I _{w'})^{q_1} \nue(3I _{w'} )^{q_2}, $$
  As a conclusion, $   \widetilde \tau_{\muu,\nue,j+2}(q_1,q_2)   \leq  \tau_{\muu,\nue,j+2}(q_1,q_2)  \leq \widetilde \tau_{\muu,\nue,j}(q_1,q_2)$.   By Lemma \ref{lem-estimtau1},  $ \widetilde \tau_{\muu,\nue,j}(q_1,q_2) $ converges to $ \Taue(q_1,q_2) $ when $j$ tends to infinity, hence the result.
\end{proof}

\subsection {Computation of the bivariate \ml spectrum}

We find a parametrization of $D_{\muu,\nue}$   allowing us to compute it explicitly and to establish     the \ml formalism  in the next subsection.  
\begin{theorem}
\label{th-spec1}
Let $0<p_1,p_2<1/2$ or  $1/2<p_1,p_2<1$.

Let $\eta\in[0,1]$, and consider the mapping 
\begin{align*}
\nonumber
\mathcal{M}: [H_{1,\min},H_{1,\max}]^2& \longrightarrow   & (\R^+)^2 \hspace{35mm} \\
 (\alpha , \beta )\hspace{9mm} & \longmapsto  & (\eta \alpha+(1-\eta)\beta, \eta \alpha+(1-\eta) \G(\beta)). 
\end{align*}
With probability one, the bivariate spectrum  $D_{\muu,\nue}$ is parametrized as follows: for every pair $(\alpha,\beta)\in [H_{1,\min},H_{1,\max}]\times [H_{2,\min},H_{2,\max}]$, call $(H_1,H_2) = \mathcal{M}(\alpha,\beta)$. Then
\begin{align}   
\label{speccor}
D_{\muu,\nue}(H_1,H_2) = \eta D_{\muu}(\alpha)+(1-\eta)D_{\muu}(\beta),
\end{align}
and $D_{\muu,\nue}(H_1,H_2) =-\infty$ whenever $(H_1,H_2)$ cannot be written as $\mathcal{M}(\alpha,\beta)$ for some pair  $(\alpha,\beta)\in [H_{1,\min},H_{1,\max}]\times [H_{2,\min},H_{2,\max}]$.

\end{theorem}

In the following proofs we   always assume that $0<p_1<p_2<1/2$, and  in this case  $H_{i,\min}= -\log_2(1-p_i)$ and $H_{i,\max}= - \log_2 p_i$. The other case  $1/2<p_1<p_2<1$ is   similar.  

We start by identifying the support of the bivariate spectrum.
\begin{lemma}
\label{lemm-parall}
With probability one, the support of $D_{\muu,\nue}$ is the parallelogram $\Pzeroeta	=\mathcal{M} ([H_{1,\min},H_{1,\max}]^2)$.

\end{lemma}
\begin{proof}

Since $0<p_1 < p_2<1/2$, the slope   of the affine mapping $\G$ is $\coefg >1$.  The other cases are similar. 

We fix  $H_1$. 

Consider $x$   with $\dimu(\muu,x)=H_1$. There   exists a sequence $(\ep_j)$ depending on $x$ and tending to 0 such that  :
\begin{itemize}
\item[(i)]
for every $j$ and $\widetilde I\in \{I,I^+,I^-\}$, $\muu( \widetilde I_j(x)) =2^{-jH'_1}$ where $H'_1\in [H_1-\ep_j, H_{1,\max}+\ep_j]$, 
\item[(ii)]
 there are infinitely many integers $j$ such that $\muu( \widetilde I_j(x)) \geq 2^{-j(H_1+\ep_j)}$ for some $\widetilde I\in \{I,I^+,I^-\}$.
\end{itemize}

We then  look for  the possible values for $H_2 $. 
Let us   write $\muu(I)=\muu(I^\AAA)\muu(I^\AAAc) = 2^{-j^\AAA \alpha} 2^{-j^\AAAc \beta}$ for $\alpha,\beta \in [H_{1,\min},H_{1,\max}]$. Recalling \eqref{decomp-nu}, one has $\nue(I) =   \mu_{p_1}(\IA)\mu_{p_2}(\IAc) = 2 ^{-j^\AAA \alpha} 2^{-j^\AAAc\beta'}$ where $\beta'=\G(\beta) \in [H_{2,\min},H_{2,\max}]$ by Lemma \ref{lem-trivial}.

%
%
%
%
%

Writing $\muu(I)=2^{-j H_1}$, then $\nue(I) = 2^{-jH'_2}$, where
\begin{equation}
\label{eq-syst}
\begin{cases}
H'_1= \frac{j^\AAA }{j} \alpha +\frac{j^\AAAc }{j} \beta \\
H'_2=\frac{j^\AAA }{j}  \alpha +\frac{j^\AAAc }{j} \G(\beta)  = H'_1 +\frac{j^\AAA }{j} (\G(\beta) -\beta).
\end{cases}
\end{equation}
Recall that by Lemma \ref{lem-maj-cardinal}, with probability one,  $j^\AAA\sim \eta j$ and $j^\AAAc\sim (1-\eta)j$,   when $j$ tends to $+\infty$. So it is enough to consider the pairs $(H'_1,H'_2)$ such that   $(H'_1,H'_2)=(\eta\alpha+(1-\eta)\beta,\eta\alpha+(1-\eta)\G(\beta))= \mathcal{M}(\alpha,\beta)$.

Let us  write 
\begin{align*}
[a_1,b_1]& =  \{H': (H_1,H')=\mathcal{M}(\alpha,\beta) \mbox{ for some } (\alpha,\beta)\in [H_{1,\min},H_{1,\max}]^2\} \\
[a_2,b_2] & = \{H': (H_1',H')=\mathcal{M}(\alpha,\beta)  \mbox{ for some } (\alpha,\beta)\in [H_{1,\min},H_{1,\max}]^2 \mbox{ and }H'_1\geq H_1\}   \\
[a_3,b_3]& =  \{H': (H_1,H') \mbox{ belongs to the support of } D_{\muu,\nue} \}.
\end{align*}

First,  assume that a pair ($\alpha,\beta)\in (H_{1,\min}, H_{1,\max})$ satisfies  $H_1=\eta\alpha+(1-\eta)\beta$, call $H'=\eta\alpha+(1-\eta)\G(\beta)$. Then, let $x$  and $y$ be such that $\dim(\muu,x)=\alpha$ and    $\dim(\muu,y)=\beta$ - such points with limit local dimensions exist by Proposition \ref{fm}. Consider the real number $z\in \zu$ such that for every integer $j$, $I_j^\mathcal{A} (z)=I_{j^\mathcal{A}} (x)$ and  $I_j^{\mathcal{A}^c} (z)=I_{j^{\mathcal{A}^c}}(y)$. One sees that  $\dim(\muu,z)=\eta\alpha+(1-\eta)\beta =H_1$ and $\dim(\nue,z)=\eta\alpha+(1-\eta)\G(\beta )=H'$, with probability one (this happens simultaneously for all such points $z$, since the randomness comes only from the set $\mathcal{A}$).
 Hence,   $H'\in [a_3,b_3]$ and so  $[a_1,b_1]\subset [a_3,b_3]$.

\sk
 
Second, 
remark   that by item (i) above, $ H'_1 = \eta \alpha + (1- \eta)\beta \geq H_1$,  and by item (ii)  the equality $ H_1 = \eta \alpha + (1- \eta)\beta $ holds infinitely many times. This last point implies in particular that $b_1=b_3$.

\sk

Assume that $a_3<a_1$. This means that $a_3=a_2<a_1$, and so there must be $H'_1>H_1$ such that $(H'_1,a_3)= \mathcal{M}(\alpha,\beta)$ for some pair $(\alpha,\beta)$. This is not possible, since the choice of $p_1,p_2$ imposes that  $\G$ is an increasing mapping with slope $>1$.

As a conclusion,  $[a_3,b_3] =  \mathcal{M} ([H_{1,\min},H_{1,\max}]^2)\cap \{(H_1,H):H\in \R^+\}$, hence the result.
\end{proof}

One can push the description  a bit further and describe the parallelogram $\Pzeroeta$.
 We just proved that  it is enough to consider the pairs $(\alpha,\beta)$ for which  $ H_1 = \eta \alpha + (1- \eta)\beta $, and  to consider the system $(H_1,H_2)=\mathcal{M}(\alpha,\beta)$, i.e. 
$$
 (S) \ \ 
\begin{cases}
 H_1= \eta \alpha + (1- \eta)\beta \\
 H_2= \eta \alpha + (1- \eta)\G(\beta)  =  H_1+(1-\eta) ( \G(\beta) -\beta).
\end{cases}
$$

Given the preceding remarks, $H_1$ being fixed,  obtaining the lower and upper bound for $H_2$ amounts to minimizing  and maximizing $\beta$, respectively. There are four situations according to the value of $H_1$:
\begin{enumerate}
\sk\item
$H_{1,\min}\leq H_1 \leq \eta H_{1,\max}+(1-\eta )H_{1,\min}$: the largest possible value for $\alpha$ is reached when $\beta=H_{1,\min}$, so the smallest possible value of $H_2$ is $H_1+(1-\eta) ( \G(H_{1,\min}) -H_{1,\min})$. The mapping $H_1\mapsto H_1+(1-\eta) ( \G(H_{1,\min}) -H_{1,\min})$  is the straight line of slope 1 passing through the point $(H_{1,\min}, \eta H_{1,\min}+(1-\eta)H_{2,\min})$.

\sk\item
$\eta H_{1,\max}+(1-\eta )H_{1,\min} \leq H_1    \leq H_{1,\max}$: the largest possible value for $\alpha$ is $H_{1,\max}$, and the corresponding $\beta$ in (S) equals  $\beta= (H_1-\eta H_{1,\max})/(1-\eta)$. Hence the smallest possible value of $H_2$ is $H_1+(1-\eta) ( \G( (H_1-\eta H_{1,\max})/(1-\eta)) - (H_1-\eta H_{1,\max})/(1-\eta))=\coefg (  H_1-H_{1,\min}) +\eta(H_{2,\max}-H_{2,\min})+H_{2,\min} $. One checks that when $H_1=H_{1,\max}$, this value equals $\eta H_{1,\max}+(1-\eta)H_{2,\max}$, so the corresponding mapping   $H_1\mapsto \coefg (  H_1-H_{1,\min}) +\eta(H_{2,\max}-H_{2,\min})+H_{2,\min}$  is the straight line of slope $\coefg$ passing through the point $(H_{1,\max}, \eta H_{1,\max}+(1-\eta)H_{2,\max})$.

\sk\item
$H_{1,\min}\leq H_1 \leq \eta H_{1,\min}+(1-\eta )H_{1,\max}$: the smallest possible value for $\alpha$ is $H_{1,\min}$, so the largest possible value of $\beta$ is $(H_1-\eta H_{1,\min}/(1-\eta)$ and that of $H_2$ is $H_1+(1-\eta) ( \G(H_1-\eta H_{1,\min}/(1-\eta)) -H_1-\eta H_{1,\min}/(1-\eta))  $, which is the straight line  of  slope $\coefg$ passing through the point $(H_{1,\min}, \eta H_{1,\min}+(1-\eta)H_{2,\min})$.

\sk\item
$\eta H_{1,\min}+(1-\eta )H_{1,\max} \leq H_1    \leq H_{1,\max}$: the largest possible value for $\beta$ is $H_{1,\max}$, so the largest possible value for $H_ 2$ is $H_1+(1-\eta) ( \G( H_{1,\max}) -  H_{1,\max})$, which is   the straight line of slope $1$ passing through the point  $(H_{1,\max}, \eta H_{1,\max}+(1-\eta)H_{2,\max})$.

 \end{enumerate}

The four cases give  the four corresponding straight lines constituting the frontier of the parallelogram $\mathcal{M} ([H_{1,\min},H_{1,\max}]^2)$.

\mk

One can check that when $0<p_2<p_1<1/2$, then $\coefg<1$ and the parallelogram is still generated by the four straight lines passing through $(H_{1,\min}, \eta H_{1,\min}+(1-\eta)H_{2,\min})$ or $(H_{1,\max}, \eta H_{1,\max}+(1-\eta)H_{2,\max})$ and with slope $1$ or $\coefg$. 


Now that   the support of $D_{\muu,\nue}$ is known, we compute its value at every possible  $(H_1,H_2)$.
\begin{lemma}
\label{lem-spec-correl}
Let $(H_1,H_2)= \mathcal{M} (\alpha,\beta)$ with $(\alpha,\beta)\in [H_{1,\min},H_{1,\max}]^2$. Then $D_{\muu,\nue}(H_1,H_2) =  \eta D_{\muu}(\alpha)+(1-\eta)D_{\muu}(\beta)$.
\end{lemma}

\begin{proof}
 
Let us start with the lower bound.
By \eqref{eq-maj3} applied with $\mathcal{B}=\mathcal{A}$, $p_1=p_2$, $\alpha_1=\beta_1=\alpha$ and $\alpha_2=\beta_2=\beta$, one has 
$ \dim E^{p_1,\AAA}_{\alpha,\alpha} \cap E^{p_1,\AAAc}_{\beta,\beta}    = \eta D_{\mu_{p_1}}(\alpha)  + (1-\eta )  D_{\mu_{p_1}}(\beta) $. Next observe that when $x\in E^{p_1,\AAA}_{\alpha,\alpha} \cap E^{p_1,\AAAc}_{\beta,\beta}  $,  by Lemma \ref{lem-trivial},
\begin{align*}
\lim_{j\to+\infty} \frac {\log(\muu(B(x,2^{-j} )))}{\log 2^{-j}}& = \eta\alpha+(1-\eta)\beta = H_1,\\
\lim_{j\to+\infty} \frac{ \log(\nue(B(x,2^{-j})))}{\log 2^{-j}}& = \eta\alpha+(1-\eta)\G(\beta)=H_2,
\end{align*}
 hence $x\in  \underline{E}_{\muu,\nue}(H_1,H_2)$. So $D_{\muu,\nue}(H_1,H_2)  \geq \dim_H E^{p_1,\AAA}_{\alpha,\alpha} \cap E^{p_1,\AAAc}_{\beta,\beta} = \eta D_{\muu}(\alpha)+(1-\eta)D_{\muu}(\beta)$.

 \mk
 
Let us prove the upper bound for the dimension of $ \underline{E}_{\muu,\nue}(H_1,H_2)$.

For $x\in \zu$, write $\muu(I_j(x))= \muu(I_j^{\AAA}(x))\muu(I_j^{\AAAc}(x)) = 2^{-j^\AAA \alpha_{j,x}} 2^{-j^\AAAc \beta_{j,x} }$, so that $\nue(I_j(x))= \muu(I_j^{\AAA}(x))\mud(I_j^{\AAAc}(x)) = 2^{-j^\AAA \alpha_{j,x} } 2^{-j^\AAAc \G(\beta_{j,x} )}$. 

Similarly, write $\muu(I_j^\pm(x))= \muu(I_j^{\pm,\AAA}(x))\muu(I_j^{\pm,\AAAc}(x)) = 2^{-j^\AAA \alpha_{j,x}^\pm} 2^{-j^\AAAc \beta_{j,x}^\pm}$, and  $\nue(I_j^\pm(x))= \muu(I_j^{\pm,\AAA}(x))\mud(I_j^{\pm,\AAAc}(x)) = 2^{-j^\AAA \alpha_{j,x}^\pm} 2^{-j^\AAAc \G(\beta_{j,x}^\pm)}$.
 
Consider $x\in \underline{E}_{\muu,\nue}(H_1,H_2)$ and $\ep>0$. For any  $j$ large enough, one must have for every $j$ that $\muu(F_j(x))\leq 2^{-j(H_1-\ep)}$  and $\nue(F_j(x))\leq 2^{-j(H_2-\ep)}$, for $F\in \{I,I^+,I^-\}$. Hence, recalling  that $j^\AAA\sim \eta j$ when $j\to+\infty$, one must have
$$ H_1-2\ep \leq \eta \alpha_{j,x} +(1-\eta)\beta_{j,x} \  \mbox{ and } \ H_2-2\ep \leq \eta \alpha_{j,x}+(1-\eta)\G(\beta_{j,x}),$$
and similarly for $\alpha_{j,x}^\pm  $ and $\beta^\pm_{j,x}$.  

%
%
%

Also, for an infinite number of integers $j_n$, one must have 
$$ H_1-2\ep \leq \eta \alpha_{j_n,x}+(1-\eta)\beta_{j_n,x}\leq H_1+2\ep$$
or the same inequalities for $\alpha_{j_n,x}^\pm  $ and $\beta^\pm_{j_n,x} $. Assume that this holds for $\alpha_{j_n,x}   $ and $\beta _{j_n,x}$. 

Since $(H_1,H_2)=\mathcal{M}(\alpha,\beta)$, one gets
\begin{align*}
-2\ep  & \leq \eta (\alpha_{j_n,x}-\alpha)+(1-\eta)(\beta_{j_n,x}-\beta)\leq  2\ep\\
 \mbox{ and } \hspace {3mm} -2\ep&  \leq \eta (\alpha_{j_n,x} -\alpha)+(1-\eta)(\G(\beta_{j_n,x})-\G(\beta)).
 \end{align*}
The mapping $\G$ being affine with slope $\coefg>1$, one deduces that $\beta_{j_n,x}\geq \beta- \frac{4\ep}{\coefg-1}$. In particular, $\beta_{j_n,x}$ cannot be much smaller than $\beta$, so $\alpha_{j_n,x}$ cannot be much larger than $\alpha$.

\medskip

Let $N\in \N^*$ so large that $1/N\leq \ep$,  and split the interval $[H_{1,\min}-\ep, \alpha+\ep]$ into $N$ intervals $K_1,...,K_N$ of equal length written $\gamma/N$. Write $K_i=[\alpha_{i,N},\alpha_{i+1,N}]$ with $\alpha_{1,N}=H_{1,\min}-\ep$ and $\alpha_{N+1,N}=\alpha+\ep$.
 
 By item (5) of Proposition \eqref{fm}, for every large integer $j$, for every $i\in \{1,...,N\}$, 
 $$\left| \frac{ \log_2  \mathcal{E}_\muu(j^\AAA,K_i) }{ j^\AAA } - 
 \sup_{H\in K_i} 
 D_\muu(H) \right| \leq  \ep /2 .  $$
By continuity of $D_\muu$, $N$ can be chosen so large that $\sup_{H\in K_i}D_{\muu}(H)  \leq D_\muu(\alpha_{i,N})+\ep/2$. This yields in particular that  for large integers $j$, $ \frac{ \log_2  \mathcal{E}_\muu(j^\AAA,K_i) }{ j^\AAA } \leq D_\muu(\alpha_{i,N})+\ep$. 
 
 Call $\beta_{i,N}$ the solution to $\eta\alpha_{i,N}+(1-\eta) \beta_{i,N}=H_1$, and call $\widetilde K_i= [\beta_{i+1,N}-1/N,\beta_{i,N}+1/N]$. The same argument as above yield that  for $N$ and $j$ large,  for every $i\in \{1,...,N\}$, 
 $ \frac{ \log_2  \mathcal{E}_\muu(j^\AAAc, \widetilde K_i) }{ j^\AAAc } \leq D_\muu(\beta_{i,N})+\ep$. 

From what precedes, one deduces that when $x\in \underline{E}_{\muu,\nue}(H_1,H_2)$, for infinitely many integers $j$, $\alpha_{j_n,x}$ (or resp., $\alpha_{j_n,x}^+$, $\alpha_{j_n,x}^-$) belongs to some interval $K_i$ and simultaneously  $\beta_{j_n,x}$ (or resp., $\beta_{j_n,x}^+$, $\beta_{j_n,x}^-$) belongs to the corresponding interval $\widetilde K_i$. This leads us to introduce, as  in Proposition \ref{prop_majdimE},  the set 
$$
 F_j:=\left\{I\in \mathcal{D}_j:  \frac{\log \muu(I^\AAA) }{\log j^\AAA} \in  K_i  \ \mbox{ and } \frac{\log \muu(I^\AAAc) }{\log j^\AAAc} \in \widetilde K_i  \mbox{ for some $i\in\{1,...,N\}$} \right\}.
$$
The same arguments as in Proposition \ref{prop_majdimE} combined with the upper bounds above show that for large integers $j$, 
$$\#F_j \leq  \sum_{i=1}^N 2^{j^\AAA (D_\muu(\alpha_{i,N})+\ep) }2^{j^\AAAc (D_\muu(\beta_{i,N})+\ep) }\leq N 2^{j(\delta +2\ep)}, $$
where $\delta=\max_{i\in\{1,...,N\}} \eta D_\muu(\alpha_{i,N})+(1-\eta) D_\muu(\beta_{i,N})$, and where we used again that $j^\AAA\sim \eta j$ when $j\to+\infty$.
The concavity of $D_\muu$ shows that $\delta\leq \max_{\alpha'\leq \alpha} \eta D_\muu(\alpha' )+(1-\eta) D_\muu((H_1-\eta\alpha')/(1-\eta)) \leq  \eta D_\muu(\alpha )+(1-\eta) D_\muu((H_1-\eta\alpha)/(1-\eta))= \eta D_\muu(\alpha )+(1-\eta) D_\muu(\beta)$.

We proved that every  $x\in  \underline{E}_{\muu,\nue}(H_1,H_2)$ is at distance $2^{-j}$ from an interval belonging to $F_j$, for infinitely many integers $j$. So the intervals $(I\pm 2^{-j})_{I\in F_j}$ form a covering of $ \underline{E}_{\muu,\nue}(H_1,H_2)$. Then, by a classical covering argument, the Hausdorff dimension of $ \underline{E}_{\muu,\nue}(H_1,H_2)$ is bounded above by $\eta D_\muu(\alpha )+(1-\eta) D_\muu(\beta)$, hence the conclusion.
\end{proof}

%
%
%
%
%
%
%

\subsection {Verification of the bivariate \ml formalism}
\label{sec-bi-form-simple}

The aim is here to check that $\tau_{\muu,\nue}^* = D_{\muu,\nue}$ almost surely.

Observe that $(q_1,q_2)\mapsto \tau_{\muu,\nue}(q_1,q_2)$ is $C^\infty$ and concave, as the sum of two $ C^\infty$  concave mappings. A direct computation gives
\begin{align*}
H_1:= \frac{\partial}{\partial q_1}  \tau_{\muu,\nue}(q_1,q_2)& = \eta \frac{\partial}{\partial q_1}   \tau_{\mu_{p_1}}(q_1+q_2) + (1-\eta)   \frac{\partial}{\partial q_1}  \Tau(q_1,q_2) \\
& = \eta \frac{-p_1^{q_1+q_2}\log_2(p_1)-(1-p_1)^{q_1+q_2}\log_2(1-p_1)}{p_1^{q_1+q_2} +(1-p_1)^{q_1+q_2}}\\
& \ \ \ + (1-\eta)   \frac{-p_1^{q_1}p_2^{q_2}\log_2(p_1)- (1-p_1)^{q_1}(1-p_2)^{q_2}\log_2(1-p_1)}{p_1^{q_1}p_2^{q_2}+ (1-p_1)^{q_1}(1-p_2)^{q_2}}.
\end{align*}
and  $H_2:= \frac{\partial}{\partial q_2}  \tau_{\muu,\nue}(q_1,q_2) $ is similarly computed. For every $(q_1,q_2)$, since $\tau_{\muu,\nue}$ is $ C^\infty$  strictly concave, one has $\tau_{\muu,\nue}^* (H_1,H_2) = q_1H_1+q_2H_2-\tau_{\muu,\nue}(q_1,q_2)$.

Let $q_\alpha =q_1+q_2$ and $q_\beta $ be the unique real number such that $\frac{p_1 ^{q_\beta} }{p_1 ^{q_\beta}+(1-p_1) ^{q_\beta}}  = \frac{p_1^{q_1}p_2^{q_2}}{p_1^{q_1}p_2^{q_2}+ (1-p_1)^{q_1}(1-p_2)^{q_2} }$. Call $\alpha=\tau_{\muu}'(q_\alpha) = -p_1 ^{q_\alpha} \log_2(p_1) -(1-p_1) ^{q_\alpha} \log_2(1-p_1) $ and $\beta=\tau_{\muu}'(q_\beta) = -p_1 ^{q_\beta} \log_2(p_1) -(1-p_1) ^{q_\beta} \log_2(1-p_1) $. Combining the previous formulas gives  $(H_1,H_2)= \mathcal{M} (\alpha,\beta)$ and $(\alpha,\beta)\in [H_{1,\min},H_{1,\max}]^2$. 
 
 By Lemma  \ref{lem-spec-correl}, $D_{\muu,\nue}(H_1,H_2) =  \eta D_{\muu}(\alpha)+(1-\eta)D_{\muu}(\beta)$. It remains us to check that  this formula coincides with $\tau_{\muu,\nue}^* (H_1,H_2) $. Using the above formulas and after computations, one sees that
 \begin{align*}
   D_{\muu}(\alpha)  & =  q_\alpha \alpha-\tau_{\muu}(q_\alpha)= q_\alpha(\alpha+\log_2(p_1) )   -(q_1+q_2)\log_2 p_1 -\tau_\muu(q_1+q_2)  \\
      D_{\muu}(\beta)  & =  q_\beta \beta-\tau_{\muu}(q_\beta)= q_\beta(\beta+\log_2(p_1) )   -q_1\log_2 p_1  -q_2\log_2 p_2 -\Taue(q_1,q_2)  .
   \end{align*}
So, remembering   \eqref{defTe},  to get the equality $\eta D_{\muu}(\alpha)  +(1-\eta)  D_{\muu}(\beta) = q_1H_1+q_2H_2-\tau_{\muu,\nue}(q_1,q_2)$, it is enough to check that 
 \begin{align*}
  &\eta(q_\alpha(\alpha+\log_2(p_1) )   -(q_1+q_2)\log_2 p_1 ) +(1-\eta) (q_\beta(\beta+\log_2(p_1) )   -q_1\log_2 p_1  -q_2\log_2 p_2) \\
  &= q_1H_1+q_2H_2 = \eta\alpha(q_1+q_2)  +(1-\eta)(  \beta(q_1+\coefg q_2) +\coefg q_2\log_2p_1 -q_2\log_2p_2).
     \end{align*}
This equality is verified   using   $q_\alpha=q_1+q_2$ and  the same manipulations as in \eqref{fix-x} and \eqref{equa-q} to get that $q_1+\coefg q_2=q_\beta$, and  the result follows.

\section {The case where $p_1 <1/2 < p_2$ }
\label{sec-case2}

The results for $p_2 <1/2 < p_1$  are similar and obtained by changing $\eta$ in $1-\eta$.

\subsection {Computation of the bivariate $L^q$-spectrum}

We prove that  for every $(q_1,q_2)$, $
 \tau_{\mu_{p_1},\nue}(q_1,q_2) =\min(\Taue(q_1,q_2)  ,   \Tauet(q_1,q_2) )$.
 
\begin{proof} 

Observe that using the fact that  $p_1<1-p_1$:
\begin{itemize}
\item
 for every $w\in \Sigma_{j-1}$, $\muu(I_{w0})$  and $\muu(I_{w1})$ (resp. $\nue(I_{w0})$  and $\nue(I_{w1})$) are of the same order of magnitude. So, in order to compute $\tau_{\muu,\nue}$, it is enough to consider in the sums defining $\tau_{\muu,\nue,j}$  only those intervals whose dyadic coding ends with a 1.
 \item
any word $w\in \Sigma_j\setminus \{1^j\} $ ending with a 1 can be written $w=w'01^\ell$, where $\ell \geq 1$, $w'\in \Sigma_{j-\ell-1}$ and $1^{\ell}$ is the word of $\ell$ letters all equal to one. One has  $ I_{w}+ 2^{-j} = I_{w'10^\ell}$, and    $\muu(I_{w})=\mu(I_{w'})p_1(1-p_1)^\ell \geq \mu(I_{w'})(1-p_1)p_1^\ell =\muu( I_{w}+ 2^{-j} )  $. In particular, $\muu(3I_w) \sim \muu(I_w)$.
 
\item
for $w\in \Sigma_{j-1}$,   the  mass $\nue(3 I_{w1})=\nue(I_{w1}^-)+\nue(I_{w1})+\nue(I_{w1}^+) $ is equivalent to $\nue(I_{{w1}})   + \nue(I_{{w1}}+2^{-j})$, so finally the product $\muu(3I_{w1})^{q_1}\nue(3 I_{w1})^{q_2} $ is equivalent (up to uniform constants depending on $q_2$, $p_1$ and $p_2$) to  $\muu(I_{w1})^{q_1}(\nue(I_{{w1}})^{q_2}  + \nue(I_{{w1}}+2^{-j})^{q_2}) $.

\end{itemize}
%

%
%
%
%
%

From all this one deduces that 
\begin{align*}
 \sum_{w\in \Sigma_j} \muu(3I_w)^{q_1}\nue(3 I_w)^{q_2} & \sim   \sum_{w\in\Sigma_{j-1}}  \muu(I_{w1})^{q_1}\nue(I_{w1})^{q_2} +\sum_{w\in\Sigma_{j-1}}  \muu(I_{w1})^{q_1} \nue(I_{w1}+  2^{-j})^{q_2} .  \end{align*}
       
  The first sum is equivalent to $\sum_{w\in\Sigma_{j-1}}  \muu(I_{w})^{q_1}\nue(I_{w})^{q_2} $, which was treated in Lemma \ref{lem-estimtau1} and is equivalent to $2^{-j\Taue(q_1,q_2)}$.
  
  For the second sum, denoted by $S(j,q_1,q_2)$, observe that  $\Sigma_{j-1}$ can be decomposed as $\Sigma_{j-1}= \{1^{j-1} \} \cup     \bigcup_{\ell =0}^{j-2} \bigcup_{w\in \Sigma_\ell} w01^{j-2-\ell} $. This  remark yields that  (the word $1^{j}$    yields  a zero term since $(I_{1^j}+2^{-j})\cap [0,1)=\emptyset$)
  \begin{align*}
S(j,q_1,q_2) & =  \sum_{\ell=0}^{j-2} \sum_{w\in \Sigma_{\ell} }   \muu(I_{w01^{j-2-\ell}})^{q_1} \nue(I_{w01^{j-2-\ell}}+ 2^{-j})^{q_2} \\
 & =  \sum_{\ell=0}^{j-2} \sum_{w\in \Sigma_{\ell} }   \muu(I_{w01^{j-2-\ell}})^{q_1} \nue(I_{w10^{j-2-\ell}})^{q_2} .
 \end{align*}
 For $w\in \Sigma_{\ell}$, $   \muu(I_{w01^{j-2-\ell}})  = \muu(I_w) p_1(1-p_1)^{j-2-\ell}$, and using the notations of Lemma \ref{lem-maj-cardinal},
 $$   \nue(I_{w10^{j-2-\ell}})  = \nue(I_w) (1-p_1)^{{\bf 1\!\! \!1}_{\ell\in \AAA} } (1-p_2)^{{\bf 1\!\! \!1}_{\ell\in \AAAc}}   p_1^{\#\AAA_{\ell+2,j-1} }  p_2^{j-2-\ell-\#\AAA_{\ell+2,j-1}}.$$
The terms $p_1$ and  $(1-p_1)^{{\bf 1\!\! \!1}_{\ell\in \AAA} } (1-p_2)^{{\bf 1\!\! \!1}_{\ell\in \AAAc} }  $ are bounded by above and below by uniform constants, so
  \begin{align*}
S(j,q_1,q_2) & \sim  \sum_{\ell=0}^{j-2} (1-p_1)^{q_1(j-\ell-2)} p_1^{q_2\#\AAA_{\ell+2,j-1} }  p_2^{q_2(j-\ell-2-\#\AAA_{\ell+2,j-1})} \sum_{w\in \Sigma_{\ell} }   \ \muu(I_w)^{q_1} \nue(I_{w})^{q_2} .
 \end{align*}

 By \eqref{taux-conv-tau}, for some constant $C>0$ one has for every $j\geq 1$
 $$2^{ -C ( (\ell-1)  {\log (\ell-1) }) ^{1/2}}  \leq  \frac{\sum_{w\in \Sigma_{\ell-1} }   \ \muu(I_w)^{q_1} \nue(I_{w})^{q_2}  } {2^{-(\ell-1) \Taue(q_1,q_2)}} \leq 2^{ C ( (\ell-1)  {\log (\ell-1) }) ^{1/2}},$$
so, replacing $\ell-1$ by $\ell$ in the above formula (which only introduces constants), one gets
  \begin{align}
  \label{sum-S2}
S(j,q_1,q_2) & \sim  \sum_{\ell=0}^{j-2}  2^{-\ell \Taue(q_1,q_2) + O(\sqrt{\ell \log \ell}) } (1-p_1)^{q_1(j-\ell-2)} p_1^{q_2\#\AAA_{\ell+2,j-1} }  p_2^{q_2(j-\ell-2-\#\AAA_{\ell+2,j-1})}.
 \end{align}

Next, by Lemma  \ref{lem-maj-cardinal}, with probability one there exists    an integer $J_0$ such that for every $J \geq J_0$, for every $j\geq  J^{\ep_J} $ where $\ep_J= (\log(J))^{-1/2}$,  \eqref{majA_j} holds true.  In particular, fixing $\ep>0$, if $ \ep^{-2} \leq \ell\leq j- j^\ep$, then $ \ell+\ell^{1/\sqrt{\log\ell}}  \leq \ell+\ell^{\ep} \leq j$ and  \eqref{majA_j}  (applied with $\ell+2$  instead of  $J$ and $j-1$ instead of $J+j-1$) gives for $j$ large that
 \begin{equation}
\label{majA_j2}
   | \#\mathcal{A}_{\ell+2,j-1}   - \eta  (j-\ell+1)  | \leq  \sqrt{ (j-\ell-2) \log (j-\ell-2)/\ep_{\ell+2}}  \leq (j-\ell)^{1/2+\ep}.
\end{equation}
Let us split the sum in \eqref{sum-S2} into three  sums $\Sigma_1(j)$, $\Sigma_2(j)$ and $\Sigma_3(j)$ according to the ranges of indices $\ell\in \{0,..., \lfloor \ep^{-2}\rfloor+1 \}$, $\{ \lfloor \ep^{-2}\rfloor+2,...,j-\lfloor j^\ep\rfloor \}$ and  $\{ j-\lfloor j^\ep\rfloor +1,...,j-2\}$.
For the first one, the terms $ 2^{-\ell \Taue(q_1,q_2) + O(\sqrt{\ell \log \ell}) } $ can be viewed as constants, and
 \begin{align*}
\Sigma_1(j) & \sim  \sum_{\ell=1}^{\lfloor \ep^{-2}\rfloor+1}  (1-p_1)^{q_1(j-2-\ell)}  p_1^{q_2\#\AAA_{\ell+2,j-1} }  p_2^{q_2(j-1-\ell-\#\AAA_{\ell+2,j-1})} .
 \end{align*}
The ranges of $\ell$ being finite, by the law of large numbers, for every $j$ large enough, for every $\ell \in \{1,..., \lfloor \ep^{-2}\rfloor+1 \}$, $ |\#\AAA_{\ell+2,j-1} - \eta (j-1)| \leq   j  \ep_j$ for some sequence $\ep_j$ that tends to zero when $j$ tends to infinity. Hence, for some sequence $\ep'_j$ tending to zero, and recalling \eqref{defTet}, one sees that 
 \begin{align*}
\Sigma_1(j) & \sim  \   (1-p_1)^{q_1j}  p_1^{q_2\eta j }  p_2^{q_2j(1-\eta) }2^{j\ep'_j}  = 2^{-j (\Tauet(q_1,q_2)-\ep'_j)}. 
 \end{align*}

\begin{center}
\begin{figure}
 \includegraphics[width=7.8cm,height=7cm]{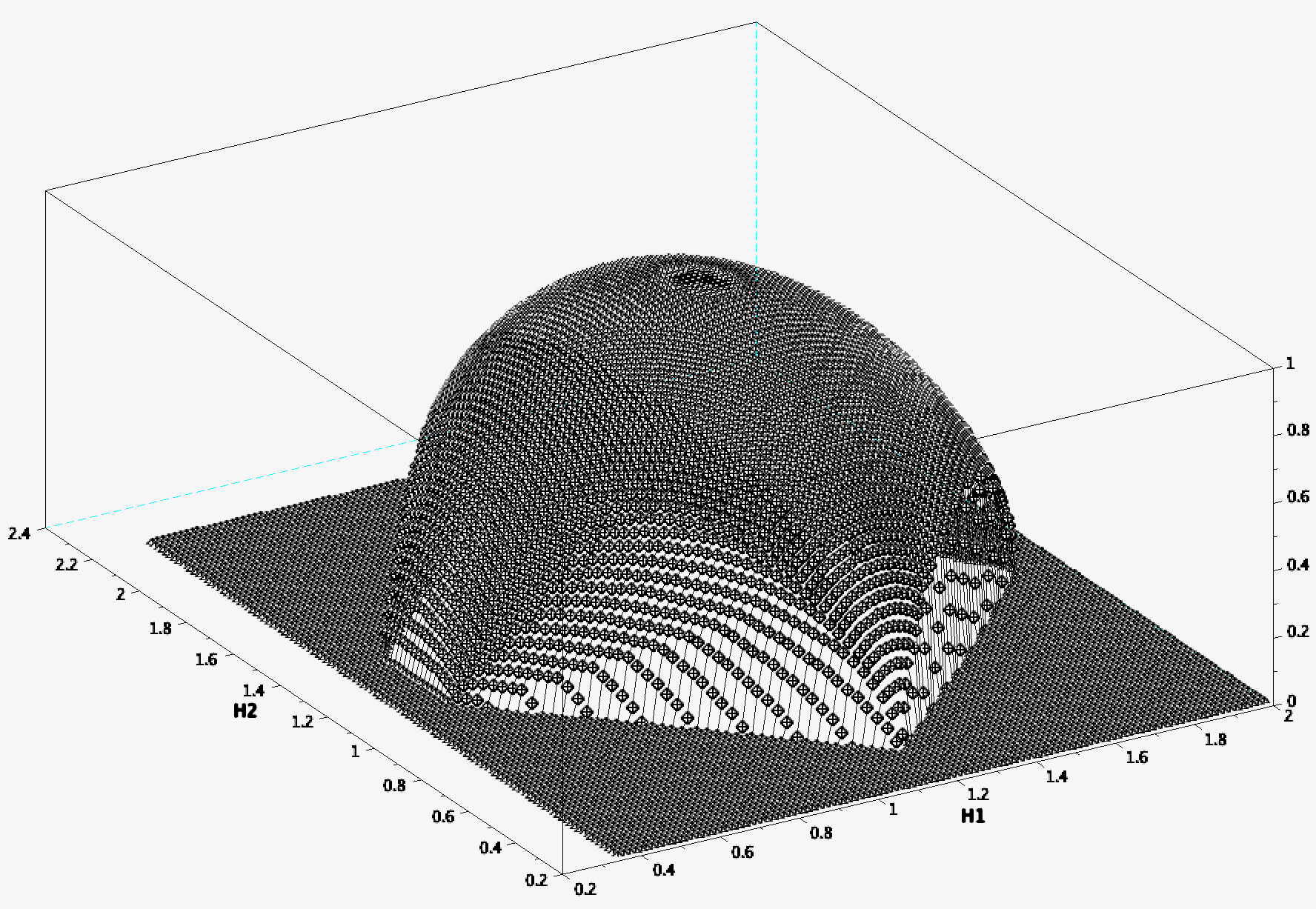} \  
 \includegraphics[width=7.8cm,height=7cm]{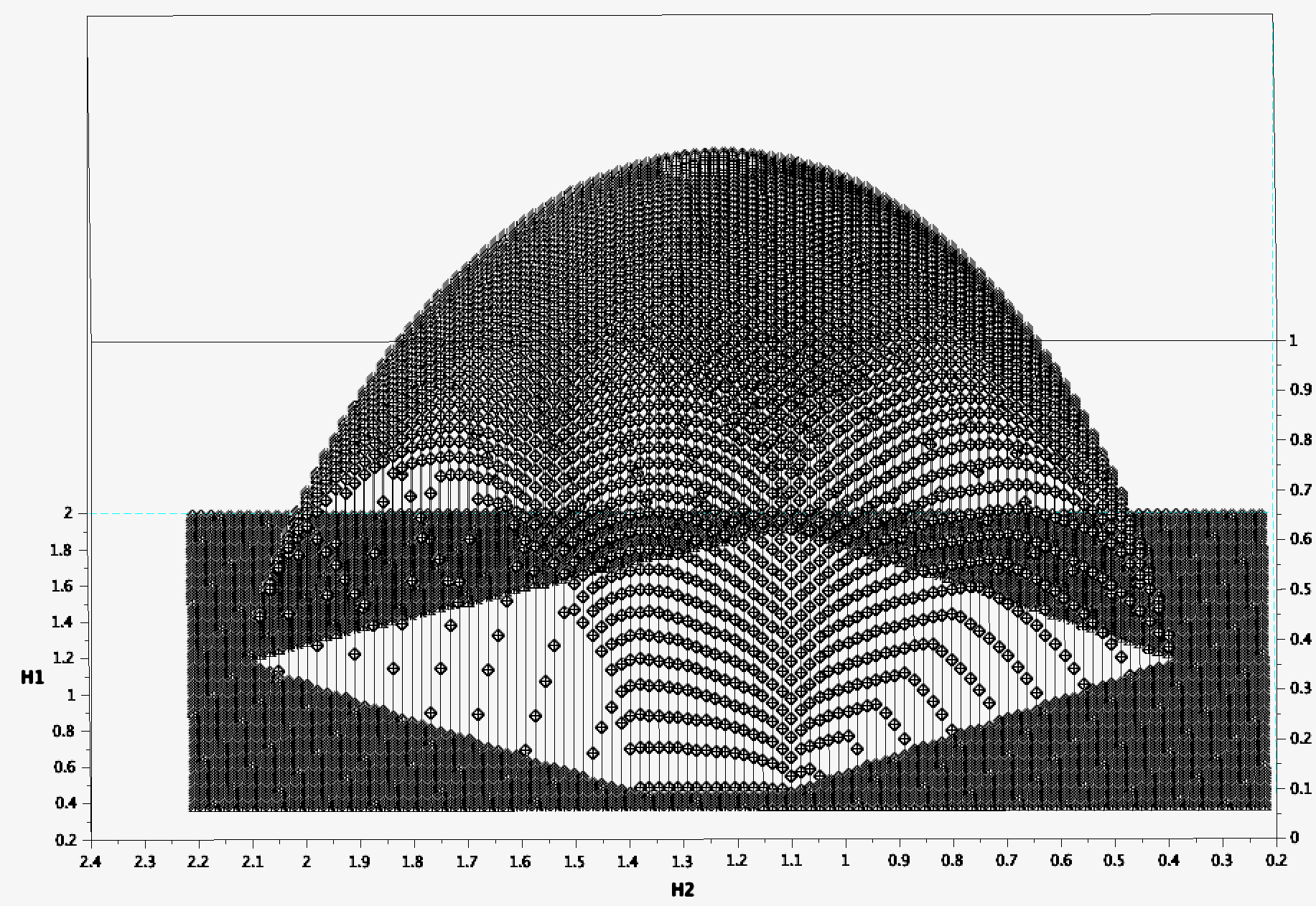}   
  \caption{Two views of the Legendre spectrum when  $0<p_1=0.2 <1/2<p_2=0.8$ with $\eta=0.5$.}
\label{fig-taueta2*}
 \end{figure}
 \end{center}
%

For the second sum, one applies \eqref{majA_j2} to get for some   $\ep''_j$ tending to zero that
\begin{align*}
\Sigma_2(j) & \sim  \sum_{\ell= \lfloor\ep^{-2}\rfloor+2}^{j-\lfloor j^\ep\rfloor }  2^{-\ell \Taue(q_1,q_2) } (1-p_1)^{q_1(j-\ell)} p_1^{q_2\eta(j-\ell) }  p_2^{q_2(1-\eta)(j-\ell) } 2^{ O(\sqrt{\ell \log \ell} +  (j-\ell)^{1/2+\ep}) }  \\
& \sim \sum_{\ell= \lfloor\ep^{-2}\rfloor+2}^{j-\lfloor j^\ep\rfloor }  2^{-\ell \Taue(q_1,q_2)  -(j-\ell) \Tauet(q_1,q_2) +j\ep''_{j} }  \sim 2^{-j \min(\Taue(q_1,q_2)  , \Tauet(q_1,q_2)) +j \ep''_{j}}.
 \end{align*}

Finally, for the third term, we observe that $ |j-1-\ell|\leq j^\ep$, and deduce that for some other sequence $\ep'''_j$ tending to zero (we also replace $(j-1)-(j-1)^\ep$ by $j-j^\ep$ which simplifies the calculation)
 \begin{align*}
 \Sigma_3(j) &  \sim  \sum_{\ell= j-\lfloor j^\ep\rfloor } ^j 2^{-\ell \Taue(q_1,q_2) } (1-p_1)^{q_1(j-\ell)} p_1^{q_2\eta(j-\ell) }  p_2^{q_2(1-\eta)(j-\ell) } 2^{ O(\sqrt{\ell \log \ell} +  (j-\ell)^{1/2+\ep}) }  \\
  & \sim    2^{-j \Taue(q_1,q_2) +j\ep'''_j } .
 \end{align*}
Combining the three previous estimates, one concludes that
 \begin{align*}
S(j,q_1,q_2)  & \sim      2^{-j \min(\Taue(q_1,q_2)  , \Tauet(q_1,q_2)) +j\tilde\ep_{j}},
 \end{align*}
  where $\tilde\ep_{j} \to 0$.
  
 Finally, 
 \begin{align*}
 \sum_{w\in \Sigma_j} \muu(3I_w)^{q_1}\nue(3 I_w)^{q_2}  & \sim      2^{-j \min(\Taue(q_1,q_2)  , \Tauet(q_1,q_2)) +j\tilde\ep_{j}}.
 \end{align*}
 
The result follows. Observe that  a phase transition occurs  at the curve $ \{(q_1,q_2):\Taue(q_1,q_2) = \Tauet(q_1,q_2)\}$.
\end{proof}

The support of $(\Tauet)^*$ is the parallelogram with same corners as in Section  \ref{sec-case1}, however there is a difference between these two figures, due to the relative positions of $p_1$ and $p_2$ and the fact that $\coefg <0$. Indeed, the four straight lines generating the parallelogram have slopes 1 and $\coefg$, which is positive or negative according to whether   $ p_1,p_2 $ are one the same side of $1/2$ or not, respectively, see   Figure \ref{fig-taueta2*}.

%
  
Next, observe that $\frac{\partial}{\partial q_1}\Tauet(q_1,q_2) = \log_2(1-p_1)=H_{1,\min}$ and $\frac{\partial}{\partial q_2}\Tauet(q_1,q_2) = - \eta\log_2(p_1)- (1-\eta) \log_2(p_2)=\eta H_{1,\max}+(1-\eta)H_{2,\min}$, so these partial derivatives are independent of $q_1$ and $q_2$. Also, the point $(H_{1,\min},\eta H_{1,\max}+(1-\eta)H_{2,\min})$ lies outside the support of $ ( \Taue)^*$. From the elementary properties of the Legendre transform we deduce that:
\begin{itemize}
\item
 $\tau_{\muu,\nue}^*$ has support the convex pentagon $\mathcal{P}^\eta_1$ whose 5 corners are   $(H_{1,\min},\eta H_{1,\max}+(1-\eta)H_{2,\min})$ and the four points of the parallelogram described above, see Figure  \ref{fig-taueta2*}.

 \item
  $\tau_{\muu,\nue}^*$ is the concave hull of the union of the singleton $(H_{1,\min},\eta H_{1,\max}+(1-\eta)H_{2,\min},0)$ and the graph of $(\Taue)^*$. In particular, it contains  a cone constituted by segments all starting from $(H_{1,\min  },\eta H_{1,\max}+(1-\eta)H_{2,\min})$ and tangent to the graph of $(\Taue)^*$, see Figures \ref{fig-taueta1*} and \ref{fig-taueta2*}.
\end{itemize}
 
\subsection {Computation of the bivariate \ml spectrum}
\label{sec-spec-decorr}

Let us write $[H_{\eta,\min}:=\eta H_{1,\min}+(1-\eta)H_{2,\min}, H_{\eta,\max}:=\eta H_{1,\max}+(1-\eta)H_{2,\max}]$ for the support of $D_{\nue}$.

Assume that for some $x$, $\dimi(\muu,x)=H_1$ and $\dimi(\nue,x)=H_2$. There   exists a sequence $(\ep_j)$ depending on $x$ and tending to 0 such that  :
\begin{itemize}
\item[(i)]
for every $j$ and $\widetilde I\in \{I,I^+,I^-\}$, $\muu( \widetilde I_j(x)) =2^{-jH'_1}$ where $H'_1\in [H_1-\ep_j, H_{1,\max}+\ep_j]$ and $\nue( \widetilde I_j(x)) =2^{-jH'_2}$ where $H'_2\in [H_2-\ep_j, H_{\eta,\max}+\ep_j]$  (here $H'_1$ and $H'_2$ depend on $j$ and $x$),
\item[(ii)]
 there are infinitely many integers $j$ such that $\muu( \widetilde I_j(x)) \geq 2^{-j(H_1+\ep_j)}$ for some $\widetilde I\in \{I,I^+,I^-\}$ and, possibly for other integers $j$, $\nue( \widetilde I_j(x)) \geq 2^{-j(H_2+\ep_j)}$ for some $\widetilde I\in \{I,I^+,I^-\}$.

\end{itemize}

The same argument as in Lemma  \ref{lemm-parall} gives, writing $H'_1=\eta\alpha+(1-\eta)\beta$ and $H'_2 =\eta\alpha+(1-\eta)\G(\beta)$ and letting $j$ tend to infinity, that one necessarily has   $(\eta\alpha+(1-\eta)\beta,\eta\alpha+(1-\eta)G(\beta)) \in [H_1,H_{1,\max}]\times [H_2,H_{\eta,\max}] $, and  the two following systems 
$$ (S_1) \ \ 
\begin{cases}
H_1 = \eta \alpha + (1- \eta)\beta  \\
H_2 \leq   \eta \alpha + (1- \eta)\G(\beta)  
\end{cases}
 \ \mbox{ and }  \ \  (S_2) \ \ 
\begin{cases}
H_1\leq   \eta \alpha + (1- \eta)\beta  \\
 H_2= \eta \alpha + (1- \eta)\G(\beta)
 \end{cases}
$$
are realized for infinitely many integers $j$ (not necessarily simultaneously).

 Let us first clarify the possible range for the pair of exponents $(H_1,H_2)$,

\begin{lemma}
\label{lemm-penta}
Almost surely, the support of $D_{\muu,\nue}$ is a (deterministic) pentagon $\mathcal{P}^\eta_2$ different from $\mathcal{P}^\eta_1$. 
\end{lemma}
\begin{proof} Recall that when $0<p_1<1/2<p_2<1$, $\coefg<0$.

\sk

$\bullet$ Consider first the system $(S_1)$. We write $\beta_{H_1,\alpha}=\frac{H_1-\eta\alpha}{1-\eta}$, so that $H_1 = \eta \alpha + (1- \eta)\beta_{H_1,\alpha}$.

The exponent $H_1$ being fixed, we look for the possible values of $H_2$. One can always take $H_2=H_{\eta,\min}$, which is the smallest  possible  value, so we only have to investigate the largest possible value for $H_2$.

One rewrites the second line of $(S_1)$ as  $H_2 \leq  H_1 + (1- \eta)(\G(\beta_{H_1,\alpha}) -\beta_{H_1,\alpha})$. Hence,     the largest $\alpha$, the smallest $\beta_{H_1,\alpha}$ (since $\coefg<0$), and the largest   $H_2$ in $(S_1)$.
  
 Observe also that by construction $\alpha$ and $\beta$ belong to $[H_{1,\min},H_{1,\max}]$, so 
 \begin{equation}
 \label{contr-alpha}
 \max\big (H_{1,\min}, \frac{H_1-(1-\eta)H_{1,\max}}{\eta} \big) \leq \alpha \leq \min\big(H_{1,\max},\frac{H_1-(1-\eta)H_{1,\min}}{\eta}  \big).
 \end{equation}

The two expressions in the  above minimum coincide when  $H_1=(1-\eta) H_{1,\min}+\eta H_{1,\max}$.

\begin{enumerate}

\sk\item
$H_{1,\min}\leq H_1 \leq (1-\eta )H_{1,\min}+ \eta H_{1,\max}$:  the largest possible value for $\alpha$ is $\frac{H_1-(1-\eta)H_{1,\min}}{\eta} $. The largest possible $H_2$ satisfying $(S_1)$ is thus $f_1(H_1):= \eta \frac{H_1-(1-\eta)H_{1,\min}}{\eta}    + (1-\eta)\G( \beta_{H_1,\frac{H_1-(1-\eta)H_{1,\min}}{\eta} } )$, and it is checked that $f_1$ is the straight line with slope 1 passing through the points $(H_{1,\min}, \eta H_{1,\min}+(1-\eta )H_{2,\max})$ and $( \eta H_{1,\max}+(1-\eta )H_{1,\min}, H_{\eta,\max})$.

\sk\item
$(1-\eta )H_{1,\min}+ \eta H_{1,\max} < H_1\leq H_{1,\max}$:  the largest possible value for $\alpha$ is now constant and equal to $H_{1,\max}$,  the corresponding $\beta$ is $\beta_{H_1,H_{1,\max}}$, and the maximal $H_2$ is $\eta H_{1,\max} +(1-\eta)\G(\beta_{H_1,H_{1,\max}}):=f_2(H_1)$. This yields the straight line with slope $\coefg$ passing through  $( \eta H_{1,\max}+(1-\eta )H_{1,\min}, H_{\eta,\max})$ and $( H_{1,\max}, \eta H_{1,\max}+(1-\eta )H_{1,\min}, )$.

 \end{enumerate}
 
Hence, the possible $(H_1,H_2)$ form a pentagon. see Figure \ref{fig-pentagons}, top left.
 
\mk

$\bullet$ Next, consider $(S_2)$. The discussion is quite similar to the one for $(S_1)$. Indeed, it is always possible to take $H_1=H_{1,\min}$ in $(S_2)$, so $H_2$ being fixed, we look for the maximal value for $H_1$.
 Call $\widetilde\beta_{H_2,\alpha}$ the unique solution to $ H_2= \eta \alpha + (1- \eta)\G(\widetilde\beta_{ H_2,\alpha})$.
 
 By construction $\alpha$ and $\beta$ belong to $[H_{1,\min},H_{1,\max}]$, so 
 \begin{equation}
 \label{contr-alpha2}
 \max\big (H_{\eta,\min}, \frac{H_2-\eta H_{1,\max}}{1- \eta} \big) \leq \G(\widetilde\beta_{ H_2,\alpha})  = \frac{H_2-\eta\alpha}{1-\eta}\leq \min\big(H_{\eta,\max},\frac{H_2- \eta H_{1,\min}}{1-\eta}  \big).
 \end{equation}

This time we write $H_1=H_2+(1-\eta)(\beta-\G(\beta))$. To get the largest $H_1$ one must find the largest $\beta$ satisfying $(S_2)$. This amounts to taking the minimum in \eqref{contr-alpha2}. Remark that the two expressions in the above maximum  coincide when   $H_2=  H^\eta_2:= \eta(1-\eta)  H_{1,\min}+ (\eta+(1-\eta)^2) H_{1,\max}$. Observe that $\eta(1-\eta)+(\eta+(1-\eta)^2) =1$, so $H_{1,\min}\leq H_2\leq H_{1,\max}$.

\begin{enumerate}

\sk\item
$H_{\eta,\min}\leq H_2 \leq H^\eta_2 $:  the largest possible value for $\G(\beta)$ is constant equal to $H_{\eta,\min}$, the corresponding $H_1$ is  $\widetilde f_3(H_2):=H_2+ (1- \eta)(\widetilde\beta_{H_2,\alpha}+\G(\widetilde\beta_{H_2,\alpha})) $. Writing the reciprocal function $H_2=f_3(H_1)$ of $\widetilde f_3$,   a calculation shows that $f_3$ is the straight line with slope 1 passing through the points $(\eta H_{1,\min}+(1-\eta )H_{1,\max},H_{\eta,\min})$ and $( H_{1,\max}, \eta H_{1,\max}+(1-\eta )H_{1,\min})$.

\sk\item
$ H^\eta_2 < H_2\leq H_{\eta,\max}$:  the largest possible value for $\G(\beta)$   is now $\frac{H_2-\eta H_{1,\max}}{1- \eta} $, and the maximal $H_1$ is  $\widetilde f_4(H_2):=H_2+ (1- \eta)(\widetilde\beta_{H_2,\alpha}+\G(\widetilde\beta_{H_2,\alpha})) $. One can check that the  reciprocal function $H_2=f_4(H_1)$ of $\widetilde f_4$ coincides with the straight line $f_2$ of item (2) obtained above when studying the system $(S_1)$.  
 \end{enumerate}
 
Hence, once again the possible range of $(H_1,H_2)$ forms a pentagon, which differs from the one following from $(S_1)$, see Figure \ref{fig-pentagons}, top right.

\mk

The intersection $\mathcal{P}^\eta_2$ between the two above pentagons happens to remain a pentagon with five edges $( H_{1,\min}, H_{\eta,\min})$, $(H_{1,\min}, \eta H_{1,\min}+(1-\eta )H_{2,\max})$, $( \eta H_{1,\max}+(1-\eta )H_{1,\min}, H_{\eta,\max})$, $( H_{1,\max}, \eta H_{1,\max}+(1-\eta )H_{1,\min}, )$ and $( \eta H_{1,\min}+(1-\eta )H_{1,\max} , H_{\eta,\min} ) $. See again Figure \ref{fig-pentagons}, bottom right, for an illustration.
\end{proof}

\begin{center}
\begin{figure}
\includegraphics[width=7.8cm,height=6cm]{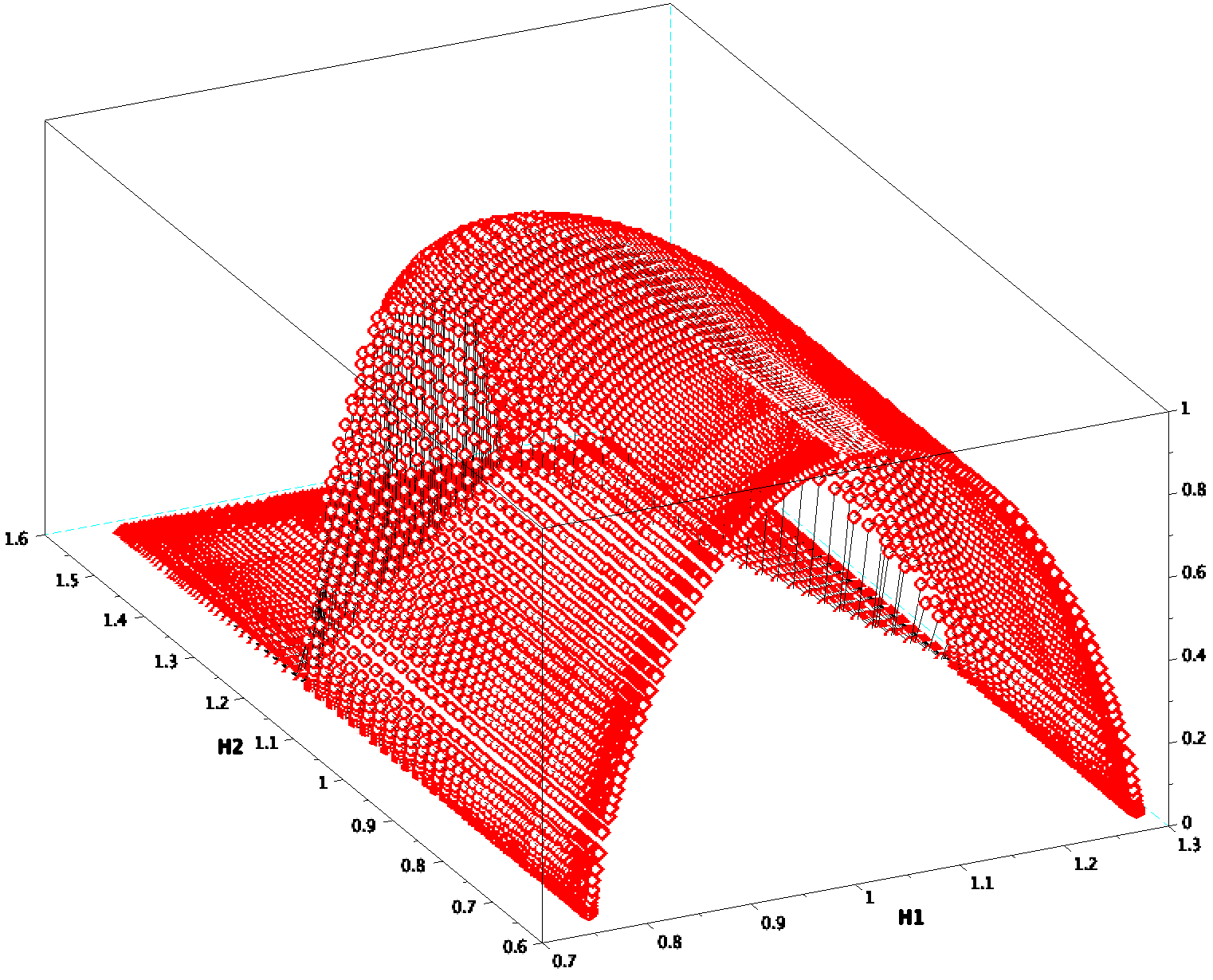} \  
\includegraphics[width=7.8cm,height=6cm]{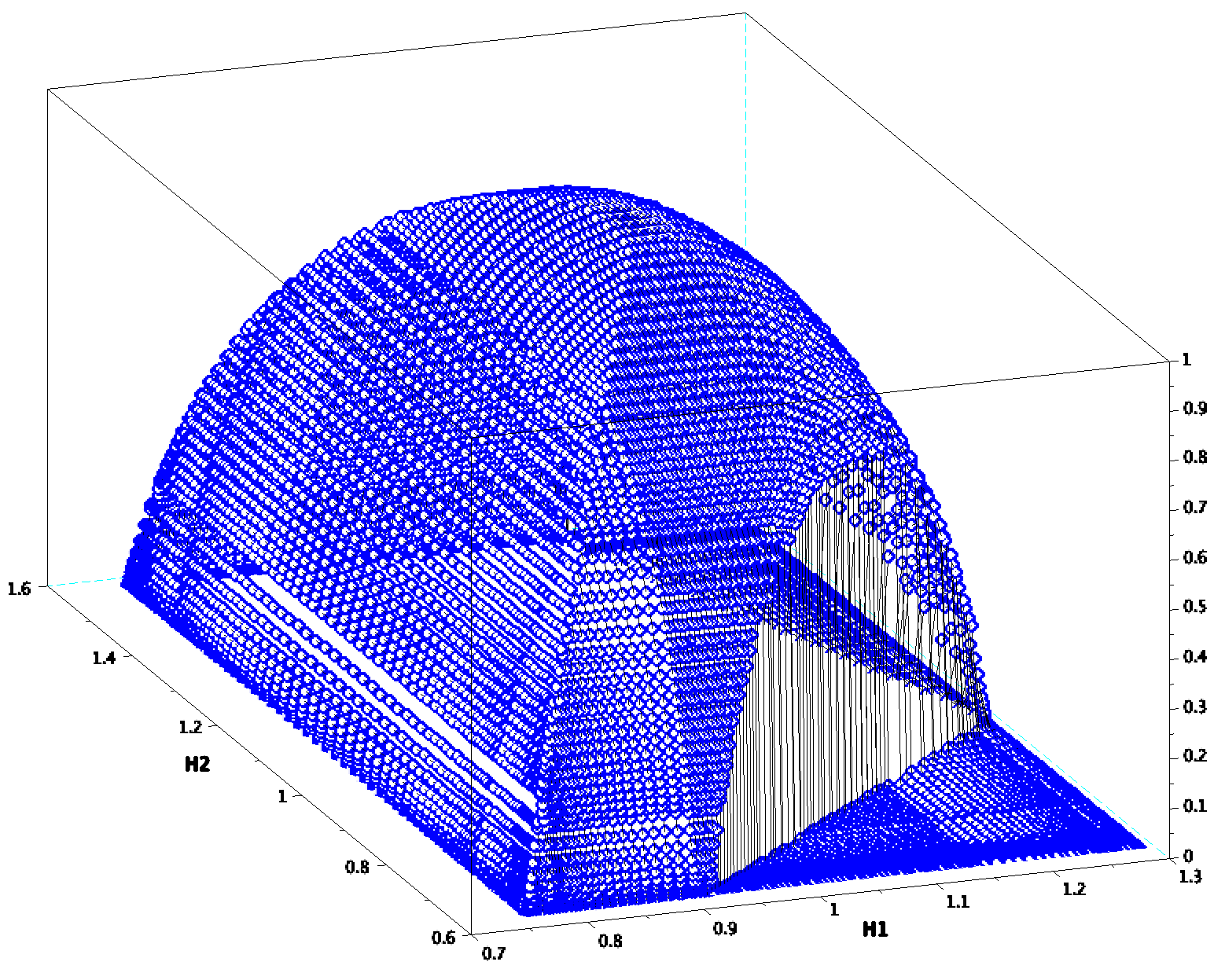}  \sk \sk  \\  
\includegraphics[width=7.8cm,height=6cm]{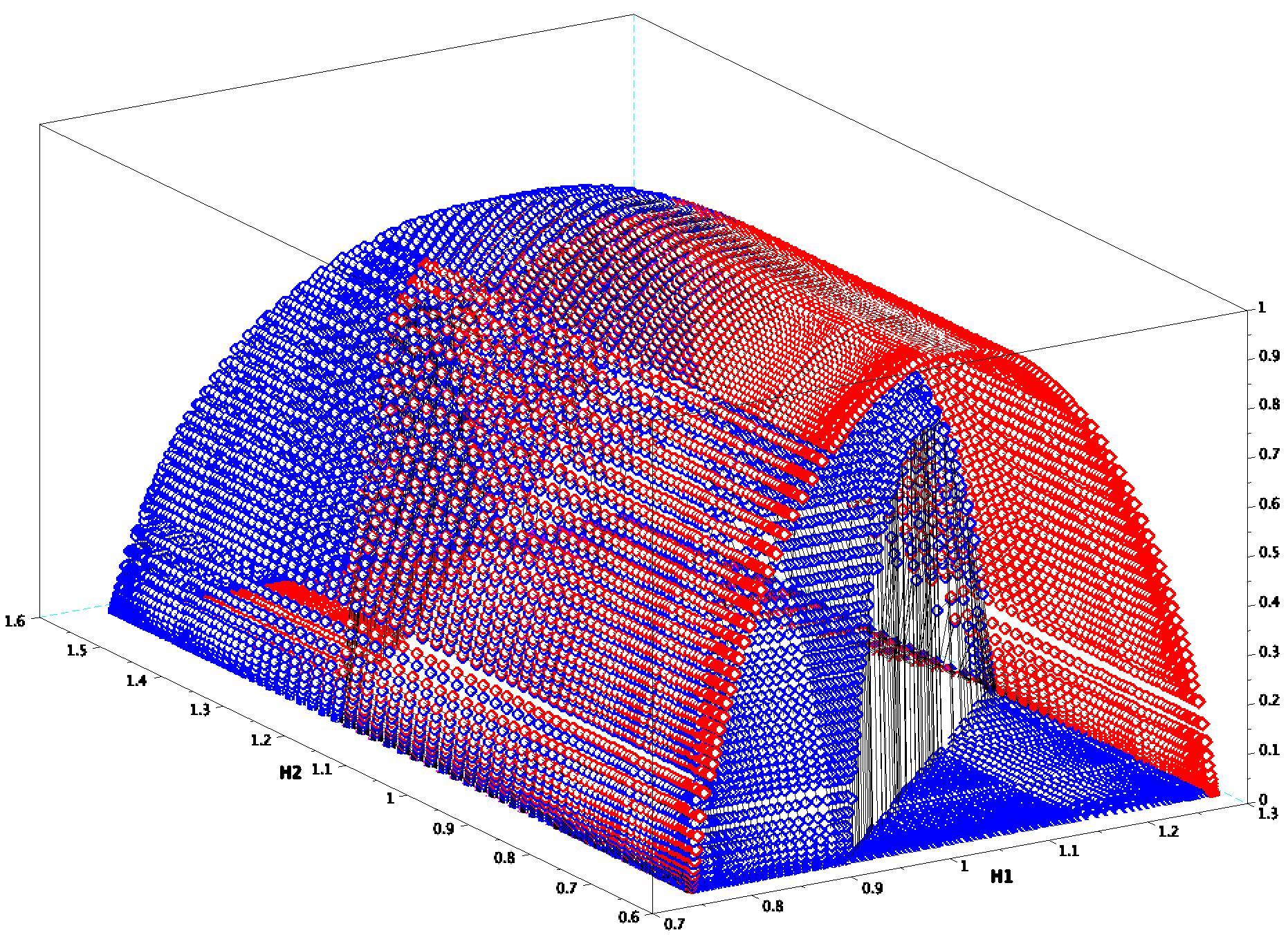} \  
\includegraphics[width=7.8cm,height=6cm]{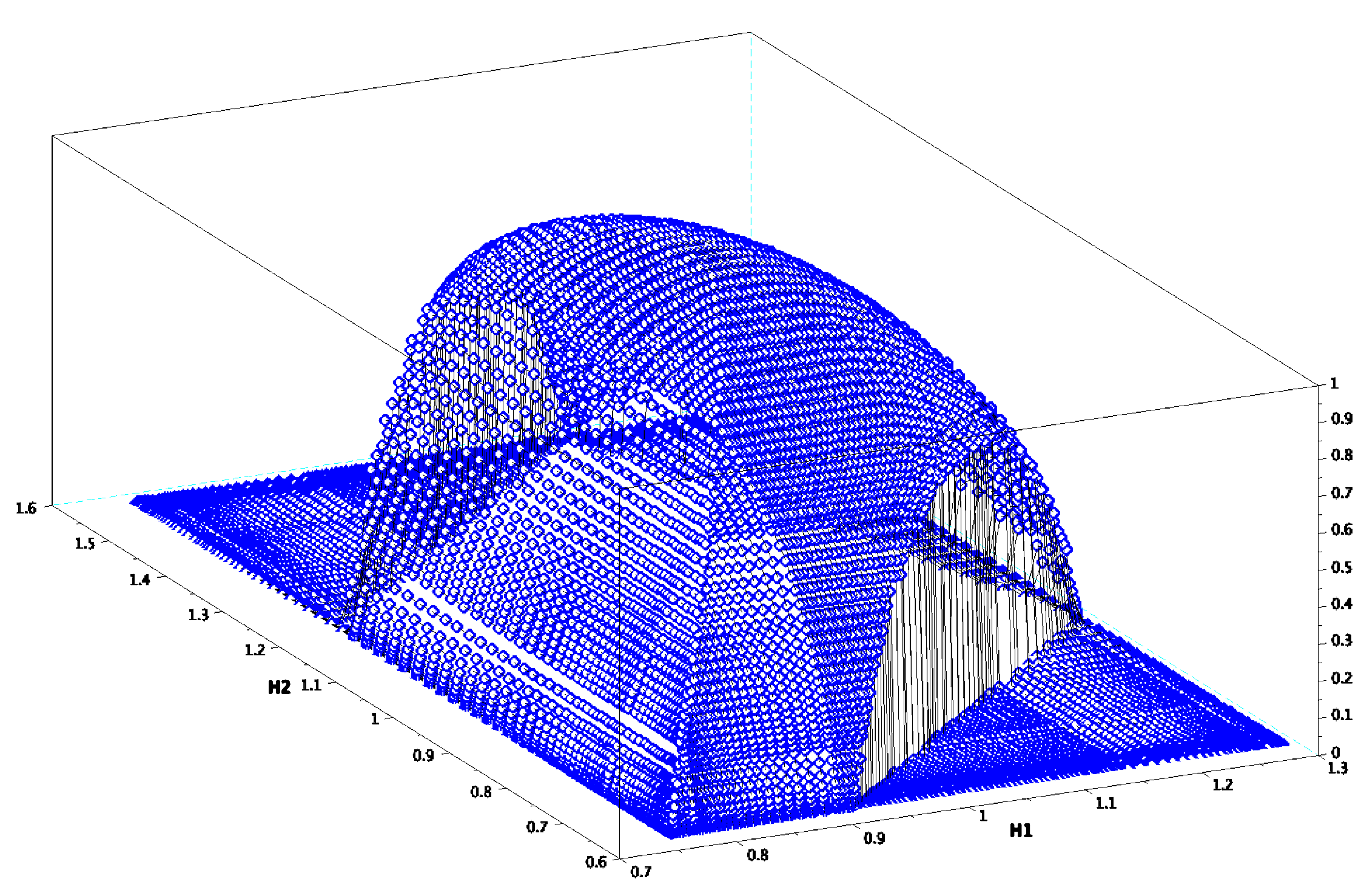}  \sk \sk  \\  
  \caption{{\bf Top:} Plot of $D_1$ and $D_2$ (equations \eqref{defD1} and \eqref{defD2}) when $\eta=0.7$, $p_1=0.4$ et $p_2=0.8$. {\bf Bottom:} Superposition of $D_1$ and $D_2$, and bivariate multifractal spectrum of $(\muu,\nue)$ (i.e. minimum of $D_1$ and $D_2$). }
\label{fig-pentagons}
 \end{figure}
 \end{center}

We move now to  the bivariate \ml spectrum $D_{\muu,\nue}$. Even if this estimate is not  sharp,  it is obvious that $D_{\muu,\nue}(H_1,H_2)\leq \min(D_{\muu}(H_1), D_{\nue}(H_2)) $. The term $D_{\muu}(H_1)$ is explicit,  let us estimate the second one.

\begin{proposition}
\label{lemm-calculdmueta2}
Almost surely, for every $ H  \in [\eta H_{1,\min}+(1-\eta)H_{2,\min}, \eta H_{1,\max}+(1-\eta)H_{2,\max}]$,  $D_{ \nue}( H )=    D_{p_1,p_2}(H) $, where
  $$D_{p_1,p_2} (H)=\eta D_{\muu} (\tau_\muu'(q_H)) + (1- \eta )D_{\mud} (\tau_\mud'(q_H)),  $$
and      $q_H $ is   the unique real number such that $H   = \eta \tau_\muu'(q_H) +(1-\eta)\tau_\mud'(q_H)$.
\end{proposition}

\begin{proof}
 
First, the existence and unicity of $q_H$ follow from the monotonicity of  $\tau_\muu'$ and $\tau_\mud'$.

If for a real number $x$ one has $\underline{\dim}(\nue,x)= H_2$, then  there   exists a sequence $(\ep_j)$ depending on $x$ and tending to 0 such that  :
\begin{itemize}
\item[(i)]
for every $j$ and $\widetilde I\in \{I,I^+,I^-\}$, $\nue( \widetilde I_j(x)) =2^{-jH' }$ where $H' \in [H_2 -\ep_j, \eta H_{1,\max}+(1-\eta)H_{2,\max} +\ep_j]$, 
\item[(ii)]
 there are infinitely many integers $j$ such that $\nue( \widetilde I_j(x)) \geq 2^{-j(H_2+\ep_j)}$ for some $\widetilde I\in \{I,I^+,I^-\}$.
\end{itemize}
 
Consider  the interval $\widetilde I\in \{I,I^+,I^-\}$ such that   $\nue( \widetilde I_j(x)) $ is maximal, and write $\frac{\log \muu( \widetilde I_j(x)^{\mathcal{A}} ) }{2^{-\eta j}} = \eta \alpha_j$ and $\frac{\log \mud( \widetilde I_j(x)^{\mathcal{A}^c} ) }{2^{-j(1-\eta)}} =  \beta_j$, so that $\nue( \widetilde I_j(x))=2^{-j(\eta\alpha_j+(1-\eta)\beta_j)}$. Similar arguments as those used in Lemmas \ref{lem-spec-correl} and \ref{lemm-penta} imply that  there is a subsequence of $(\alpha_j,\beta_j)_{j\geq 1}$ such that $(\alpha_j)$ and $(\beta_j)$ tend respectively to $\alpha$ and $\beta$ with $\eta \alpha +(1-\eta)\beta=H_2$, and that the set of such points has a dimension equal to $\eta D_{\muu}(\alpha)+(1-\eta)D_{\muu}(\beta)$. In particular, recalling Proposition \ref{prop_majdimE}, 
$$ \bigcup_{\alpha,\beta:  \, \eta\alpha+(1-\eta)\beta=H_2} \hspace{-4mm} E^{p_1,\mathcal{A}}_{\alpha,\alpha } \cap E^{p_2,\mathcal{A}^c}_{\beta,\beta }  \subset E_{\nue}(H) .$$

The same arguments as in Lemmas   \ref{lem-spec-correl} and \ref{lemm-penta} also show that with probability one, 
$$ \dim \Big ( \bigcup_{\alpha,\beta:  \, \eta\alpha+(1-\eta)\beta=H_2}  \hspace{-4mm}  E^{p_1,\mathcal{A}}_{\alpha ,\alpha} \cap E^{p_2,\mathcal{A}^c}_{\beta ,\beta}  \Big) =\max _{\alpha,\beta:  \, \eta\alpha+(1-\eta)\beta=H_2} \big( \eta D_{\muu}(\alpha)+(1-\eta)D_{\mud}(\beta\big),$$
and this last maximum is also an upper bound for the dimension of $E_{\nue}(H) $.

The derivative of  $\alpha \mapsto  L(\alpha):= \eta D_{\muu}(\alpha)+(1-\eta)D_{\mud}( \frac{H_2 -\eta\alpha}{1-\eta})$ equals   $\eta (D'_{\muu}(\alpha)-D'_{\mud}(\frac{H_2 -\eta\alpha}{1-\eta}))$, which vanishes when $D'_{\muu}(\alpha)=D'_{\mud}(\frac{H_2 -\eta\alpha}{1-\eta})=D'_{\mud}(\beta)$. Recalling that if $\alpha =\tau_\muu'(q_\alpha)$, then $D'_{\muu}(\alpha)=q_\alpha$ (and similarly for $\beta=\tau_\mud'(q_\beta)$), one concludes the maximum of $L$ is reached when $q_\alpha=q_\beta$, hence the result. 
\end{proof}

Let us estimate the value of $D_{\muu,\nue}$ at every pair $(H_1,H_2) \in \mathcal{P}^\eta_2$. It appears that the upper bound $D_{\muu,\nue}(H_1,H_2)\leq \min(D_{\muu}(H_1), D_{\nue}(H_2)) $ is sharp only on a subset of $\mathcal{P}^\eta_2$, the exact value of $D_{\muu,\nue}(H_1,H_2)$ can be strictly smaller.

Recall the notations  $\beta_{H_1,\alpha}$ and $\widetilde\beta_{H_2,\alpha}$ of the proof of Lemma \ref{lemm-penta}, and the systems $(S_1)$ and $(S_2)$ of the beginning of Section \ref{sec-spec-decorr}.

\begin{definition}
Define 
\begin{equation}
\label{def-Feta}
F_\eta(\alpha,\beta):=\eta D_{\muu}(\alpha)+(1-\eta)D_{\muu}(\beta).
\end{equation}

For $(H_1,H_2)\in [H_{1,\min},H_{1,\max}]\times [H_{\eta,\min},H_{\eta,\max}] $, let us denote
\begin{align}
\label{defD1}
D_1(H_1,H_2)& = \max\{ F_\eta(\alpha,\beta): (\alpha,\beta) \mbox{ satisfies }(S_1)\}\\
\nonumber& = \max\{ F_\eta(\alpha,\beta_{H_1,\alpha}): (\alpha,\beta_{H_1,\alpha}) \mbox{ satisfies }(S_1)\}\\
\label{defD2}
D_2(H_1,H_2)& = \max\{ F_\eta(\alpha,  \beta ): (\alpha,\beta) \mbox{ satisfies }(S_2)\}\\
\nonumber& = \max\{ F_\eta( \alpha,\widetilde \beta_{H_2,\alpha}): (\alpha,\widetilde\beta_{H_2,\alpha}) \mbox{ satisfies }(S_2)\}.
\end{align}
\end{definition}
 
Observe that in \eqref{defD1} and \eqref{defD2}, once $H_1$, $H_2$ and $\alpha$ are fixed, the parameter $\beta$ is actually entirely determined by $\alpha$ (this is what is expressed in the second equation defining $D_1$ and $D_2$), so $\alpha$ is actually the  only moving parameter in the maxima to be computed.

\begin{proposition}
\label{prop-calculdmueta2}

Almost surely, for every $ (H_1,H_2)  \in \mathcal{P}^\eta_2$,
$$D_{\muu,\nue}(H_1,H_2)  = \min (D_1(H_1,H_2),D_2(H_1,H_2)) .$$
\end{proposition}

\begin{proof}
Fix $(H_1,H_2) \in \mathcal{P}^\eta_2$.
When $\dimi(\muu,x)=H_1$ and $\dimi(\nue,x)=H_2$, items (i) and (ii) of the beginning of Section \ref{sec-spec-decorr} hold. 

Let us call $\mathcal{G}(H_1,H_2,\alpha,\beta )$ the sets of points $y\in \zu$ such that $(S_1)$ is satisfied by the quadruplet $(H_1,H_2,\alpha,\beta)$,   $\dimi(\muu,y) =H_1$ and $\dimi(\nue,y) \geq H_2$.

Similarly, call $ \widetilde {\mathcal{G}} (H_1,H_2,\alpha,\beta )$ the sets of points $y\in \zu$ such   that $(S_2)$ is satisfied by the quadruplet $(H_1,H_2,\alpha,\beta)$,  $\dimi(\muu,y) \geq H_1$ and $\dimi(\nue,y) = H_2$.

By construction, when $\dimi(\muu,x)=H_1$ and $\dimi(\nue,x)=H_2$, $x$ must belong to
$$\Big(\bigcup_{\alpha,\beta\in  [H_{1,\min},H_{1,\max}] }   \mathcal{G} (H_1,H_2,\alpha,\beta )\Big) \cap \Big(\bigcup_{\alpha,\beta\in  [H_{1,\min},H_{1,\max}] }     \widetilde {\mathcal{G}} (H_1,H_2,\alpha,\beta ) \Big).$$
 
\begin{lemma}
\label{lem-max}
There exist a unique value $\alpha_{H_1,H_2}$ such that  $D_1(H_1,H_2)=F_\eta(\alpha_{H_1,H_2},\beta_{H_1, \alpha_{H_1,H_2}})$, and  a unique value $\widetilde\alpha_{H_1,H_2}$ such that  $D_2(H_1,H_2)=F_\eta(\widetilde\alpha_{H_1,H_2},\widetilde\beta_{H_2, \widetilde\alpha_{H_1,H_2}})$.
\end{lemma}

This lemma simply follows from the strict concavity of $D_{\muu}$ and the fact that the range of possible $\alpha$'s in \eqref{defD1} and \eqref{defD2} is an  interval.

\medskip

A standard discretization argument (identical to the one used in the proof of Lemma \ref{lem-spec-correl}) shows that  the Hausdorff dimension of the union of sets $\bigcup_{\alpha,\beta\in  [H_{1,\min},H_{1,\max}] }  \mathcal{G}(H_1,H_2,\alpha,\beta )$ is reached   for the pair $(\alpha_{H_1,H_2},\beta_{H_1, \alpha_{H_1,H_2}})$, and similarly for the sets $ \widetilde {\mathcal{G}} (H_1,H_2,\alpha,\beta )$ with the exponents $(\widetilde\alpha_{H_1,H_2},\widetilde\beta_{H_2, \widetilde\alpha_{H_1,H_2}})$.

In particular, \begin{equation}
\label{last-ineq}
D_{\muu,\nue}(H_1,H_2) \leq \min\Big (F_\eta\big(\alpha_{H_1,H_2}, \beta_{H_1, \alpha_{H_1H_2} }\big), F_\eta\big (\widetilde\alpha_{H_1,H_2}, \widetilde\beta_{H_2,\widetilde\alpha_{H_1,H_2}}\big)\Big).
\end{equation}

The conclusion follows from the next proposition, that we prove just hereafter.
\begin{proposition}
\label{prop-expo-simultane}
Let $(H_1,H_2)\in \mathcal{P}^\eta_2$, and let $(\alpha,\alpha',\beta,\beta')$ be such that $(\alpha,\beta )$ and $(\alpha',  \beta ' )$ satisfy respectively the systems $(S_1)$ and $(S_2)$. Then 
$$\dim    {\mathcal{G} } (H_1,H_2, \alpha , \beta  )\cap  \widetilde {\mathcal{G} } (H_1,H_2,   \alpha', \beta ') = \min(F_\eta(\alpha,\beta ), F_\eta(\alpha', \beta ')).$$ 
\end{proposition}

From Proposition \ref{prop-expo-simultane} we deduce    that the dimension  of $ \mathcal{G} (H_1,H_2, \alpha_{H_1,H_2}, \beta_{H_1, \alpha_{H_1H_2} } )\cap \widetilde {\mathcal{G} } (H_1,H_2, \widetilde \alpha_{H_1,H_2}, \widetilde\beta_{H_1, \widetilde\alpha_{H_1H_2} } )$ equals the right-hand-side of  \eqref{last-ineq} above, hence concluding to the equality in the last equation.     
\end{proof}

We now conclude by proving Proposition \ref{prop-expo-simultane}.

\begin{proof}[Proof of Proposition \ref{prop-expo-simultane}]

Recall that in the situation of Proposition \ref{prop-expo-simultane}, $\beta=\beta_{H_1,\alpha}=\frac{H_1-\eta\alpha}{1-\eta}$ and $\beta'=\widetilde \beta_{H_2.\alpha'}$ (the unique solution to $ H_2= \eta \alpha + (1- \eta)\G(\widetilde\beta_{ H_2,\alpha})$).

Call $\delta = \min(F_\eta(\alpha,\beta ), F_\eta(\alpha', \beta' ))$. We build a measure $\nu$ of dimension $\delta$ that is supported by $  {\mathcal{G} } (H_1,H_2, \alpha , \beta )\cap \widetilde {\mathcal{G} } (H_1,H_2,   \alpha',\beta' ) $.

Lemma \ref{lem-renouv} is applied to the four parameters $\wip_\gamma$ for $ \gamma \in \{\alpha, \beta,\alpha',{\beta'}\}$.
Recall that $\wip_\gamma$ is the unique real number  such that $H_{\wip_{\gamma},e}=D_{\mu_{\wip_\gamma}}(H_{\wip_{\gamma},e})= \tau_{\muu}^*(\gamma)=D_{\muu}(\gamma)$.

Fix $\ep<1/100$, and consider  the sequence  $(\widetilde J_n)_{n\geq 1}$  defined by 
$$\widetilde J_n=\max( \max( J_{p_1,p_2,\gamma,n,\ep}: \gamma \in \{\alpha, \beta,\alpha',{\beta'}\} ) ,\widetilde J_{n-1}+2),$$
so that  the four sets   $ \mathcal{E}_{p_1,p_2,\gamma,n,\ep}$    satisfy that $\mu_{\wip_\gamma}(  \mathcal{E}_{p_1,p_2,\gamma,n,\ep}) \geq 1-\ep2^{-n-2} $ for the same sequence $\widetilde J_n$  (instead of  $J_{p_1,p_2,\gamma,n,\ep}$) in  \eqref{eq-renouv}.

 Call   $\gamma_1= \min(H_{1,\max},H_{2,\max})/2$, $\gamma_2= 2\times\max(H_{1,\max},H_{2,\max})$, and consider an increasing  sequence $(  J_n)_{n\geq 0}$ of integers with $J_0=\widetilde J_1 $ and such that for every $n\geq 1$:
\begin{itemize}

\item[(A1)]  $  {J_n}2^{-n} \geq   (J_n+2)^{\ep_{  J_n}}$ and  $ \sqrt{2^{-n}   (J_n+2) \log ( 2^{-n}   (J_n+2))/\ep_{ \widetilde  J_{n+1} }}  \leq 2^{-n-3}   J_{n}$.

\item[(A2)] 
$ J_{n} \geq   \widetilde  J_{n-1}+  \widetilde J_{n}+1$,

\item[(A3)] 
$ 2^{-n-3}(   J_{n} -  J_{n-1}) \geq \frac{\gamma_2}{\gamma_1}\widetilde J_{n}  $.

\item[(A4)] 
  $   J_{n} \geq   \max(\frac{2  J_{n-1}}{\gamma_1}, 2^{(n+3)\frac{\gamma_2}{\gamma_1}})$.

\end{itemize} 
In the first property, $\ep^J$ refers to Lemma \ref{lem-maj-cardinal}. This property, combined with  \eqref{majA_j},  implies that  for every $j\geq 2^{-n}   J_{n}$,
\begin{align}
\label{maj-card-1}
|\# \mathcal{A}_{  J_n,   J_n+j } - \eta j|  & \leq    2^{-n-2} j  \  \  \mbox{ and }  \ \ 
|\# \mathcal{A}^c_{  J_n,   J_n+j } - (1-\eta ) j|  \leq   2^{-n-2}j .
\end{align}

We set $P_1=0$, and we assume that the mass of the measure $\nu$ is given on every interval $I\in \Sigma_{P_{2n+1}}$ for some integer $n\geq 0$.

\newcommand{\wA}{{\widetilde{\mathcal{A}}}}

Recall \eqref{def-Amn}, and  set for any integer $n$
\begin{equation}
\label{def-APn}
\wA_{n}=( \mathcal{A}-{n} ) \cap \N^* \ \mbox{ and } \ \ \wA^c_{n}=( \mathcal{A}^c-{n} ) \cap \N^*.
\end{equation}

\mk
{\bf $\bullet$  Step $2n+2$:}   
Choose an integer $P_{2n+2}  $  so large that $\min(\# \mathcal{A}_{P_{2n+1},P_{2n+2}} , \#\mathcal{A}^c_{P_{2n+1},P_{2n+2} }) 
 \geq    J_{2n+2}$ (in particular, $P_{2n+2} -P_{2n+1} \geq   J_{2n+2}$).
For every $I_w\in {\Sigma}_{P_{2n+2}}$, write $w =  w_{|P_{2n+1}} w'$ for $w'\in \Sigma_{P_{2n+2}-P_{2n+1}}$, and set
 $$\nu(I_w)= \nu( I_{w_{|P_{2n+1}} } )\mu_{\wip_\alpha}(I_{w'}^{\wA_{P_{2n+1}}} )\mu_{\wip_{\beta }}(I_{w'}^{\wA_{P_{2n+1}} ^c}).$$

\mk
{\bf $\bullet$  Step $2n+3$:}  As above, choose  $P_{2n+3}  $ so large that  $\min(\#\mathcal{A}_{P_{2n+2},P_{2n+3}} ,\#\mathcal{A}^c_{P_{2n+2},P_{2n+3} }) \geq    J_{2n+3}$.
For every $I_w\in {\Sigma}_{P_{2n+3}}$, write $w =  w_{|P_{2n+2}} w'$ for $w'\in \Sigma_{P_{2n+3}-P_{2n+2}}$, and set
 $$\nu(I_w)= \nu( I_{w_{|P_{2n+2}} } )\mu_{\wip_{\alpha'}}(I_{w'}^{\wA_{P_{2n+2}  } }) \mu_{\wip_{{\beta'}}}(I_{w'}^{\wA_{P_{2n+2} }^c}).$$

The intuition is that at generation $P_{2n+2}$,  for the intervals carrying the mass of the measure $\nu$, $\muu(I)\sim |I|^{\eta\alpha+(1-\eta)\beta}$ and $\nue(I)\sim|I|^{\eta\alpha+(1-\eta)\G(\beta)}$, while at generation $P_{2n+3}$, $\muu(I)\sim |I|^{\eta\alpha'+(1-\eta)\beta'}$ and $\nue(I)\sim|I|^{\eta\alpha'+(1-\eta)\G(\beta')}$.

We now list the properties  of the measure $\nu$ we constructed, together with those of $\muu$ and $\mud$.

\begin{lemma}
\label{lemma-final1}
Write  the dyadic decomposition of   $x\in \zu$ as $x=0,x_1x_2....x_n..$, where $x_n\in \Sigma_{P_{n+1}-P_{n}}$.
Consider  
$$   \mathcal{I}_{2n,\ep}= \left \{I\in \Sigma_{P_{2n+1}-P_{2n}}:\begin{cases} I^{\wA_{P_{2n-1}}} \mbox{ contains a point }x \in  \mathcal{E}_{p_1,p_2,\alpha,n,\ep} \\
I^{\wA_{P_{2n-1}}^c} \mbox{ contains a point }x \in  \mathcal{E}_{p_1,p_2,\beta,n,\ep}\end{cases} \right \}$$
and
$$   \mathcal{I}_{2n+1,\ep}=  \left \{I\in \Sigma_{P_{2n+2}-P_{2n+1}}:\begin{cases} I^{\wA_{P_{2n}}} \mbox{ contains a point }x \in  \mathcal{E}_{p_1,p_2,\alpha',n,\ep} \\
I^{\wA_{P_{2n}}^c} \mbox{ contains a point }x \in  \mathcal{E}_{p_1,p_2,\beta',n,\ep}\end{cases} \right\}$$
Finally, let
$$E(\nu,\ep)= \Big\{x\in \zu: \ \forall \,n\geq 1 , \ x_n \in \bigcup_{I\in  \mathcal{I}_{n,\ep}}I \Big\}.$$
Then $\nu(E(\nu,\ep))\geq 1-\ep$.
\end{lemma}
\begin{proof}

By construction, for every $n$, $\mu_{\wip_\alpha}(\bigcup_{I\in  \mathcal{I}_{n,\ep}}I^{\wA_{P_{n}} }) \geq \mu_{\wip_\alpha}(\mathcal{E}_{p_1,p_2,\alpha,n,\ep}) \geq 1-\ep 2^{-n-2}$ and  $\mu_{\wip_\beta}(\bigcup_{I\in  \mathcal{I}_{n,\ep}}I^{\wA_{P_{n} }^c} ) \geq \mu_{\wip_\beta}(\mathcal{E}_{p_1,p_2,\beta,n,\ep}) \geq 1-\ep 2^{-n-2}$.

Hence,  if $x=0,x_1x_2....x_n..$, where $x_n\in \Sigma_{P_{n+1}-P_{n}}$, $\nu \Big( \Big\{x\in \zu:  \ x_n \in \bigcup_{I\in  \mathcal{I}_{n,\ep}}I \Big\} \Big) \geq 1- \ep 2^{-n-1}$. Finally, 
\begin{align*}
\nu(E(\nu,\ep)^c) & \leq  \sum_{n=1}^{\infty} \nu \Big( \Big\{x\in \zu:  \ x_n \in \bigcup_{I\in  \mathcal{I}_{n,\ep}}I \Big\} \Big) \leq  2  \sum_{n=1}^{\infty} \ \ep 2^{-n-1}  \le \ep.
\end{align*}

\end{proof}

From the definitions and the assumptions (A1-4) on $J_n$, for every $x\in  E(\nu,\ep)$, for every $n\geq 1$:
\begin{enumerate}
\item $I^\mathcal{A}_{P_{2n+1}}(x) $ satisfies $\mathcal{P}(p_1,p_2,\alpha', \ep, P_{2n+1}^{\mathcal{A}} )$ and $I^{\mathcal{A}^c}_{P_{2n+1}}(x) $ satisfies $\mathcal{P}(p_1,p_2,{\beta'}, \ep,P_{2n+1}^{\mathcal{A}^c})$.
\item
 $I^\mathcal{A}_{P_{2n+2}}(x) $ satisfies $\mathcal{P}(p_1,p_2,\alpha, \ep, P_{2n+2}^{\mathcal{A}} )$ and $I^{\mathcal{A}^c}_{P_{2n+2}}(x) $ satisfies $\mathcal{P}(p_1,p_2,\beta, \ep,P_{2n+2}^{\mathcal{A}^c})$.
\end{enumerate}

\begin{lemma}
\label{lemma-final4}
For every $n\geq 1$,   every $\alpha$,    every $I\in \Sigma_{P_{2n+1}} $, if $I^\mathcal{A}$ satisfies $\mathcal{P}(p_1,p_2,\alpha',  n,P_{2n+1} ^{\mathcal{A}} )$ and $I^{\mathcal{A}^c} $ satisfies $\mathcal{P}(p_1,p_2,{\beta'},n,P_{2n+1}^{\mathcal{A}^c})$, then for every $\widetilde I\in \{I,I^+,I^-\}$,
\begin{eqnarray*}
2^{-{P_{2n+1}} (\eta \alpha'+(1-\eta){\beta'}+2^{-n})}&  \leq  \muu(\widetilde I) \leq & 2^{-{P_{2n+1}} (\eta \alpha'+(1-\eta){\beta'}-2^{-n})} \\
2^{-{P_{2n+1}} (\eta \alpha'+(1-\eta)\G({\beta'})+2^{-n})} & \leq \nue(\widetilde I) \leq & 2^{-{P_{2n+1}} (\eta \alpha'+(1-\eta)\G({\beta'})-2^{-n})} \\
2^{-{P_{2n+1}} (F_\eta(\alpha',{\beta'})+2^{-n})} & \leq \nu(\widetilde I) \leq & 2^{-{P_{2n+1}} (F_\eta(\alpha',{\beta'})-2^{-n})} .
\end{eqnarray*}

\end{lemma}
\begin{proof}
This directly follows from the following arguments:
\begin{itemize}
\item
$|P_n^{\mathcal{A}} - \eta P_n|\leq 2^{-n-2}P_n $ and $|P_n^{\mathcal{A}^c} - (1-\eta) P_n|\leq 2^{-n-2}P_n$ by $(A1)$,
\item
 the previous lemmas,  
 \item
  the multiplicative formulas $\muu(I)= \muu(I^{\mathcal{A}})\muu(I^{\mathcal{A}^c})$, $\nue(I)= \muu(I^{\mathcal{A}})\mud(I^{\mathcal{A}^c})$ and  $\nu(I)= \mu_{\wip_\alpha}(I^{\mathcal{A}})\mu_{\wip_{\beta}}(I^{\mathcal{A}^c})$, 
\item
the assumptions $(A3)$ and $(A4)$ on $P_n$,
  \item  if $\widetilde I$ is one of the two neighbors of $I$, then $\widetilde I^{\mathcal{A}}$  and  $\widetilde I^{\mathcal{A}^c}$ also  respectively   satisfy  \eqref{triplet}.
  
  \end{itemize}
\end{proof}
 
\begin{lemma}
\label{lemma-final4bis}
For every $n\geq 1$,  every $\alpha$,    every $I\in \Sigma_{P_{2n+2}}$, if $I^\mathcal{A}$ satisfies $\mathcal{P}(p_1,p_2,\alpha,  n,P_{2n+2}^{\mathcal{A}} )$ and $I^{\mathcal{A}^c} $ satisfies $\mathcal{P}(p_1,p_2, \beta, n,P_{2n+2}^{\mathcal{A}^c})$, then for every $\widetilde I\in \{I,I^+,I^-\}$,
\begin{eqnarray*}
2^{-P_{2n+2}(\eta \alpha+(1-\eta)\beta+2^{-n})}&  \leq  \muu(\widetilde I) \leq & 2^{-P_{2n+2}(\eta \alpha+(1-\eta)\beta-2^{-n})} \\
2^{-P_{2n+2}(\eta \alpha+(1-\eta)\G(\beta)+2^{-n})} & \leq \nue(\widetilde I) \leq & 2^{-P_{2n+2}(\eta \alpha+(1-\eta)\G(\beta)-2^{-n})} \\
2^{-P_{2n+2}(F_\eta(\alpha,\beta)+2^{-n})} & \leq \nu(\widetilde I) \leq & 2^{-P_{2n+2}(F_\eta(\alpha, \beta_{\alpha})-2^{-n})} .
\end{eqnarray*}

\end{lemma}
\begin{proof}
It is identical to the previous lemma  replacing $P_{2n+1}$ by $P_{2n+2}$, $\alpha'$ by $\alpha$ and ${\beta'}$ by $\beta$.\end{proof}

\begin{lemma}
\label{lemma-final5}
With probability one, for every $x\in  E(\nu,\ep)$, $\dimi(\muu,x)=H_1$ and $\dimi(\nue,x)=H_2$, and  $\dimi(\nu,x)= \min(F_\eta(\alpha,\beta), F_\eta(\alpha',\beta' ))$.
\end{lemma}
\begin{proof}

We prove the claim for $\nue$, the proof is similar for $\muu$ and $\nu$.

By the second equations of Lemma \ref{lemma-final4} and  Lemma \ref{lemma-final4bis}, and recalling that 
\begin{equation}
\label{eq-H2}
H_2=\eta\alpha'+(1-\eta)\G(\beta') \leq \eta \alpha+(1-\eta)\G(\beta)
\end{equation}
 by the systems  $(S_1)$ and   $(S_2)$, it is enough to show that when $j\in [P_{n}+1, P_{n+1}-1]$, $\nue(\widetilde I)  \leq 2^{-j(H_2-K 2^{-n})}   $ (for some constant $K$ independent of $n$ and $j$).

Let us prove it for  $j\in [P_{2n}+1, P_{2n+1}-1]$,   the other case $j\in [P_{2n+1}+1, P_{2n+2}-1]$ is identical.

At generation $P_{2n}$, the second equation of  Lemma \ref {lemma-final4bis} holds.

Then,  let $I\in \Sigma_{P_{2n}}$  containing some $x\in E(\nu,\ep)$, and write $x=0,x_1x_2....x_n..$, with $x_n\in \Sigma_{P_{n+1}-P_{n}}$. So , using \eqref{def-APn}, $x_{2n}$ belongs to some interval $L \in \Sigma_{P_{2n+1}-P_{n}}$ such that $L^{\mathcal{A}_{P_{2n} } } \in E_{p_1,p_2,\alpha, n,\ep}$ and  $L^{\mathcal{A}_{P_{2n} } ^c} \in E_{p_1,p_2,\beta, n,\ep}$.

 Write the dyadic coding of $I_j(x)$ as $x_1x_2...x_{2n}w_j(x)$, and call $L_j(x) \in \Sigma_{j-P_{2n}}$ the dyadic interval whose dyadic coding is $w_j(x)$.
One has by construction 
\begin{eqnarray*}
  \nue(I_{j}(x) ) = &   \nue (I) \muu( L^{\mathcal{A}_{P_{2n}}} _j(x) ) \mud(L_j^{\mathcal{A}_{P_{2n}}^c}(x)) 
  \end{eqnarray*}
 where  by Lemma \ref {lemma-final4bis},
\begin{eqnarray*}
  2^{-j(\eta \alpha+(1-\eta)\G(\beta)+2^{-n})} & \leq \nue(  I) \leq & 2^{-j(\eta \alpha+(1-\eta)\G(\beta)-2^{-n})} .\end{eqnarray*}

\begin{itemize}

\sk\item
{\bf Assume that  $j\in [P_{2n}+1, P_{2n}(1+2^{-n})]$:}  one has by construction 
\begin{eqnarray*}
  \nue(I_{j}(x) ) & \leq    \nue (I)  \leq 2^{-P_{2n}(\eta \alpha+(1-\eta)\G(\beta)-2^{-n})}= 2^{-j \frac{P_{2n}}{j}(\eta \alpha+(1-\eta)\G(\beta)-2^{-n})}.
  \end{eqnarray*}
 But by (A4), $\frac{P_{2n}}{j}\geq 1-2^{-n}$, so 
$$
  \nue(I_{j}(x) ) \leq 2^{-j (\eta \alpha+(1-\eta)\G(\beta)-2^{-n})(1-2^{-n})}\leq 2^{-j (\eta \alpha+(1-\eta)\G(\beta)-(\gamma_2+1)2^{-n})}\leq 2^{-j(H_2-(\gamma_2+1)2^{-n})}.
$$

\sk\item
{\bf Assume that  $j\in [ P_{2n}(1+2^{-n}), P_{2n+1}-1]$:}  By definition of $E(\nu,\ep)$, $ \muu(L_j^{\mathcal{A}_{P_{2n}}}(x)) \leq  2^{-  \#\mathcal{A}_{P_{2n},j}  ( \alpha'-2^{-n-2 })} $  and $ \mud(  L_j^{\mathcal{A}_{P_{2n}}^c}(x)) \leq    2^{- \#\mathcal{A}_{P_{2n},j}   ^c  (\G(\beta')-2^{-n-2})} $. Hence,
\begin{eqnarray*}
  \nue(I_{j}(x) ) & \leq  &  2^{-P_{2n}(\eta \alpha+(1-\eta)\G(\beta)-2^{-n})} 2^{-  \#\mathcal{A}_{P_{2n},j}  ( \alpha'-2^{-n-2 })}  2^{- \#\mathcal{A}_{P_{2n},j}   ^c  (\G(\beta')-2^{-n-2})} .
    \end{eqnarray*}
By definition of $P_{2n}$, $|  \#\mathcal{A}_{P_{2n},j} -\eta(j-P_{2n})| \leq 2^{-n-2} (j-P_{2n})$ and $|  \#\mathcal{A}^c_{P_{2n},j} -(1-\eta) (j-P_{2n})| \leq 2^{-n-2} (j-P_{2n})$, so
\begin{eqnarray*}
 \nue(I_{j}(x) )
  &     \leq  & 2^{-P_{2n}(\eta \alpha+(1-\eta)\G(\beta)) +  (j-P_{2n}) ( \eta   \alpha' +(1-\eta)\G(\beta'))} 2^{t}, 
   \end{eqnarray*}
  where $t= 2^{-n}P_{2n} + 2^{-n-2}(j-P_{2n})  ( \alpha'-2^{-n-2 })  +  \eta(j-P_{2n}) 2^{-n-2}
+  2^{-n-2}(j-P_{2n})  ( \G(\beta')-2^{-n-2 } + (1-\eta)(j-P_{2n}) 2^{-n-2}   $. A quick estimate shows that  $t \leq (\gamma_2+2)2^{-n} j$. Hence, recalling \eqref{eq-H2}, one concludes that \begin{eqnarray*}
 \nue(I_{j}(x) )
  &     \leq  & 2^{-j( H_2-(\gamma_2+2)2^{_n})} . 
     \end{eqnarray*}

\end{itemize}

One concludes that $\ep>0$ being fixed, for every sufficiently large integer $j$, $\nue(I_j(x))\leq 2^{-j(H_2-\ep)}$. This yields that $\underline{\dim}(\nue,x)\geq H_2-\ep$. Also, by Lemma \ref{lemma-final4} $\nue(I_j(x))\geq 2^{-j(H_2+\ep)}$ for infinitely many integers $j$, so $\underline{\dim}(\nue,x)\leq H_2+\ep$. Letting $\ep$ tend to zero gives the result.

\end{proof}

Finally, we  proved that for all $x\in E(\nu,\ep)$, $\underline{\dim}(\nu,x)=\delta$. This holds for every $\ep$, and since $\nu(E(\nu,\ep))\geq 1-\ep$, one derives that for $\nu$-almost every $x$, $\underline{\dim}(\nu,x)=\delta$, $\underline{\dim}(\muu,x)=H_1$ and $\underline{\dim}(\nue,x)=H_2$.   Applying  Billingsley's Lemma allows then  to conclude the proof of Proposition \ref{prop-expo-simultane}.
\end{proof}

\subsection {Failure of the bivariate \ml formalism}

Since the multifractal  spectrum has a support strictly larger than the support of the Legendre spectrum, it is obvious that the pair $(\muu,\nue)$ does not satisfy the bivariate  \ml formalism. 

 Even if we do not elaborate on this, the two spectra do not coincide on the  intersection of their support either (see Figure \ref{fig-superposition}).
Indeed, it can be checked that $D_{\muu,\nue}$ is strictly concave. Since the Legendre spectrum has a conic shape on one part of its support,   the two spectra may coincide only on the other part of the support, i.e. where the Legendre spectrum $\tau_{\muu,\nue}^*$ is given by $(\Taue)^*$.

Recalling \eqref{speccor} which gives an alternative formula for $(\Taue)^*(H_1,H_2)$,  one sees  that $D_1$ and $(\Taue)^*$ coincide when   in Lemma \ref{lem-max} $\mathcal{M}(\alpha_{H_1,H_2}, \beta_{H_1,\alpha_{H_1,H_2}}) =(H_1,H_2)$, and  that $D_2$ and $(\Taue)^*$ coincide when    $\mathcal{M}(\widetilde \alpha_{H_1,H_2}, \widetilde\beta_{H_1,\widetilde\alpha_{H_1,H_2}}) =(H_1,H_2)$. This defines two distinct regions in the plane, with non-empty intersection. In particular, the two spectra coincide on the segment

\begin{center}
\begin{figure}
\includegraphics[width=7.8cm,height=6cm]{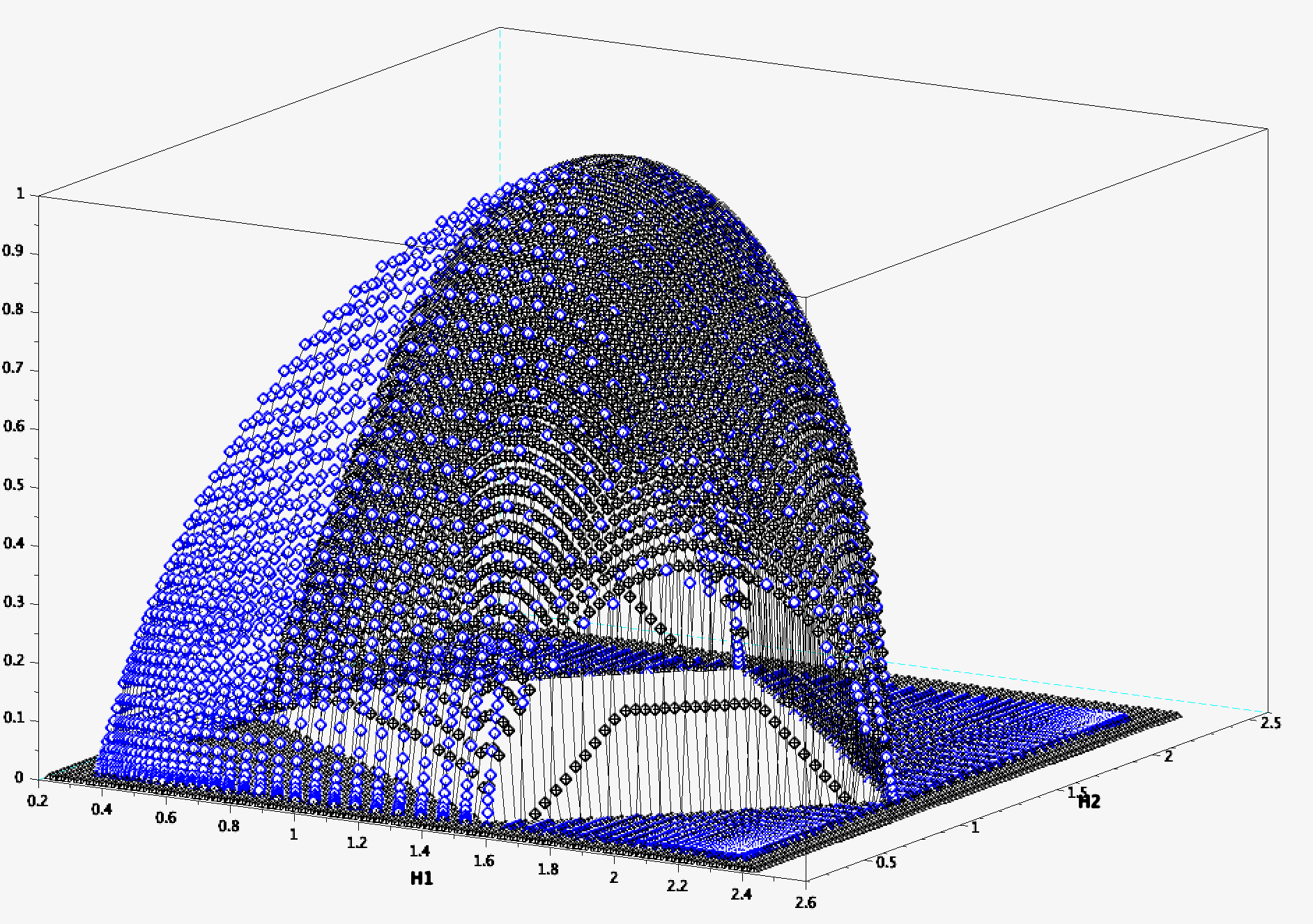} \  \ 
\includegraphics[width=7.8cm,height=6cm]{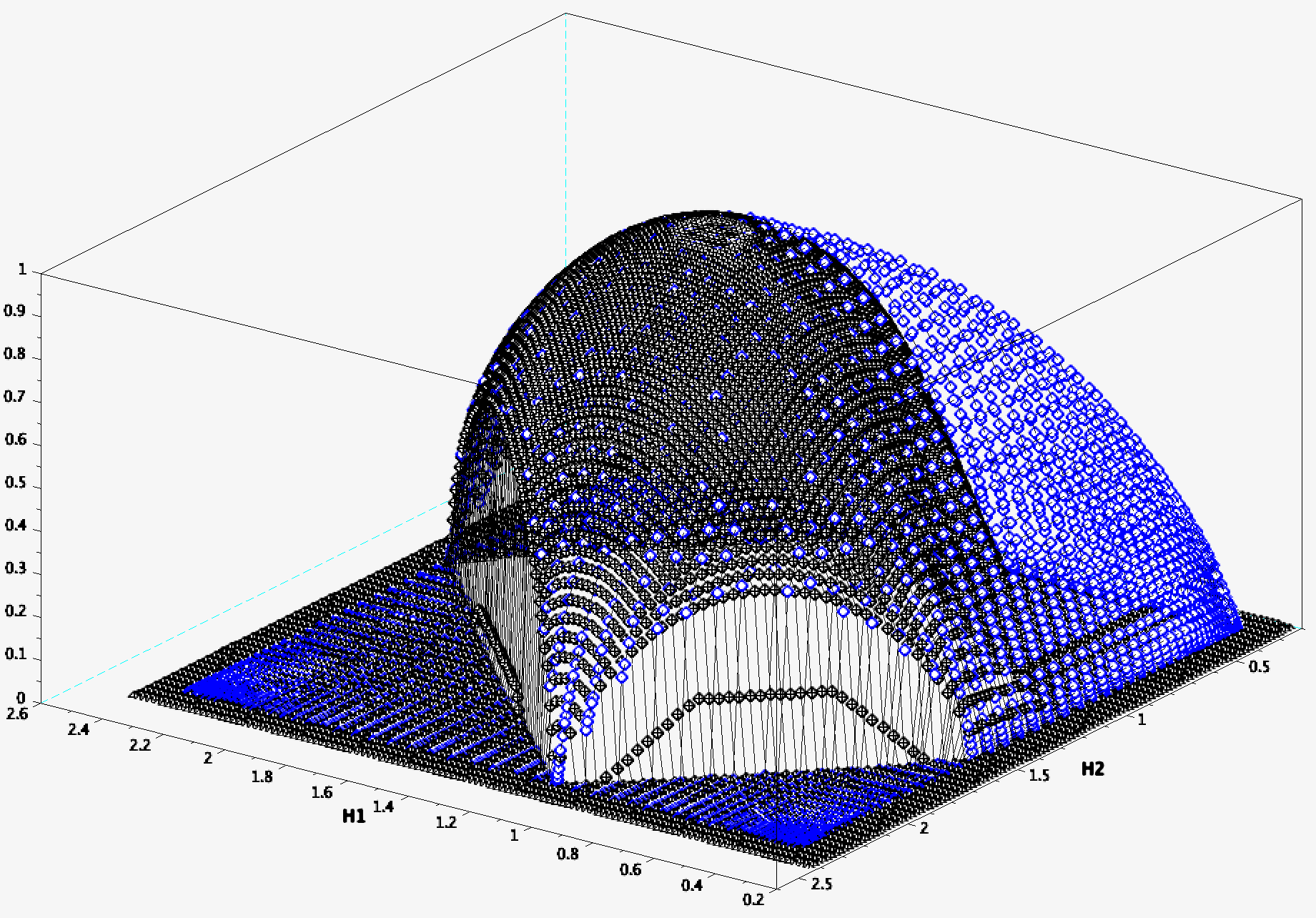}  
  \caption{Two views of the superposition of the Legendre (in black) and multifractal (in blue) spectrum of the pair $(\muu,\nue)$ for $p_1=0.2$, $p_2=0.8$, $\eta=0.4$. The two spectra coincide on a region of the plane, and $D_{\muu,\nue}$ has a wider support than $\tau_{\muu,\nue}^*$.}
\label{fig-superposition}
 \end{figure}
 \end{center}

\bibliographystyle{plain}
\bibliography{Biblio_BME}

\end{document}